\documentclass[a4paper,11pt]{amsart}
 \usepackage[english]{babel}
  \usepackage{amsmath,amsfonts,amssymb,amsthm,relsize,setspace,nicefrac, yhmath, amscd,eucal}
  \usepackage[left=1.5cm,right=1.5cm,top=2.0cm,bottom=2.5cm]{geometry}
  \usepackage{enumitem}
  \usepackage{amsrefs, mathrsfs}
 \usepackage[colorlinks, linkcolor=red, citecolor=blue, urlcolor=blue, hypertexnames=true]{hyperref}
  \usepackage{subfig}
    \usepackage{cases}
\usepackage{graphicx}
\usepackage{float}    
\usepackage{amsmath}
\usepackage{amssymb}                 
\usepackage{graphicx}

  \usepackage{hyperref} 

 \usepackage[active]{srcltx} 
 \usepackage{color,soul} 
\newtheorem{theorem}{Theorem}[section]

\newtheorem{lemma}[theorem]{Lemma}
\newtheorem{proposition}[theorem]{Proposition}

\theoremstyle{definition}
\newtheorem{definition}[theorem]{Definition}
\newtheorem{remark}{Remark}

\newcommand{\be}{\begin{equation}}
\newcommand{\bel}[1]{\begin{equation}\label{#1}}
\newcommand{\ee}{\end{equation}}

\newcommand{\barr}{\begin{eqnarray}}
\newcommand{\earr}{\end{eqnarray}}
\newcommand{\bars}{\begin{eqnarray*}}
\newcommand{\ears}{\end{eqnarray*}}

\newtheorem{subn}{\name}

\newcommand{\bsn}[1]{\def\name{#1}\begin{subn}}
\newcommand{\esn}{\end{subn}}

\newtheorem{sub}{\name}[section]

\newcommand{\bs}{\begin{sub}}
\newcommand{\es}{\end{sub}}

\newcommand{\bth}[1]{\def\name{Theorem}
\begin{sub}\label{t:#1}}
\newcommand{\blemma}[1]{\def\name{Lemma}
\begin{sub}\label{l:#1}}
\newcommand{\bcor}[1]{\def\name{Corollary}
\begin{sub}\label{c:#1}}
\newcommand{\bdef}[1]{\def\name{Definition}
\begin{sub}\label{d:#1}}
\newcommand{\bprop}[1]{\def\name{Proposition}
\begin{sub}\label{p:#1}}

\newcommand{\BA}{\begin{array}}
\newcommand{\EA}{\end{array}}
\newcommand{\BAN}{\renewcommand{\arraystretch}{1.2}
\setlength{\arraycolsep}{2pt}\begin{array}}
\newcommand{\BAV}[2]{\renewcommand{\arraystretch}{#1}
\setlength{\arraycolsep}{#2}\begin{array}}
\newcommand{\BSA}{\begin{subarray}}
\newcommand{\ESA}{\end{subarray}}

\newcommand{\BAL}{\begin{aligned}}
\newcommand{\EAL}{\end{aligned}}
\newcommand{\BALG}{\begin{alignat}}
\newcommand{\EALG}{\end{alignat}}
\newcommand{\BALGN}{\begin{alignat*}}
\newcommand{\EALGN}{\end{alignat*}}
\def\angb<#1>{\langle #1 \rangle}

      \def\CN{{\mathcal N}}
\def\CR{{\mathcal R}}      
      
      \def\CF{{\mathcal F}}
      
      \def\CL{{\mathcal L}}
   \def\CU{{\mathcal U}}   \def\CV{{\mathcal V}}

\def\R{\mathbb{R}}

\let\.=\cdot
\let\0=\emptyset

\def\O{\Omega}

\numberwithin{equation}{section}

\theoremstyle{definition}

\let\.=\cdot
\let\0=\emptyset

\def\O{\Omega}

\newenvironment{formula}[1]{\begin{equation}\label{eq:#1}}
                       {\end{equation}\noindent}

\def\Fi#1{\begin{formula}{#1}}
\def\Ff{\end{formula}\noindent}

\setstcolor{red}

\usepackage{wrapfig}

\begin{document}
\fontsize{11pt}{5pt}\selectfont


  \title[Dynamics for a system of double free boundary]{Dynamics for a  two phases free boundaries system in an epidemiological model with nonlocal dispersals}

\author{Thanh-Hieu Nguyen}
\address{Faculty of Mathematics and Applications, Saigon University, 273 An Duong Vuong st., Ward 3, Dist.5, Ho Chi Minh City, Viet Nam}
\email{thanhhieukhtn@gmail.com}


\author{Hoang -Hung Vo$^{*}$}
\address{Faculty of Mathematics and Applications, Saigon University, 273 An Duong Vuong st., Ward 3, Dist.5, Ho Chi Minh City, Viet Nam}
\email{vhhung@sgu.edu.vn}
\thanks{$^*$ Corresponding author}


\begin{abstract}
The present paper is devoted to the investigation of the long time dynamics for a double free boundary system with nonlocal diffusions, which models the  infectious diseases transmitted via digestive system such as fecal-oral diseases, cholera, hand-foot and mouth, etc...We start by proving the existence and uniqueness of the Cauchy problem, which is not a trivial step due to  presence of couple dispersals and new types of nonlinear reaction terms. Next, we provide simple conditions on comparing the basic reproduction numbers $\CR_0$ and $\CR^*$ with some certain numbers to characterize the global dynamics, as $t\to\infty$. We further obtain the sharp criteria for the spreading and vanishing in term of the initial data.  This is also called the  vanishing-spreading phenomena. The  couple dispersals yield significant obstacle that we cannot employ the approach of Zhao, Zhang, Li, Du \cite{WWZ1} and Du-Ni \cite{DN}.  To overcome this, we must  prove the existence and the variational characterization for the principal eigenvalue of a linear system with  nonlocal dispersals, then use it to obtain the right limits as the dispersal rates and domain tend to zero or infinity. The maximum principle and sliding method  for the nonlocal operator are ingeniously employed to achieve the desired results.
\end{abstract}


\date{\today.}



\maketitle
\tableofcontents
\section{ Introduction and statement of the results}
Infectious diseases have for centuries ranked with wars and famine as major challenges to human progress and survival. They remain among the leading causes of death and disability worldwide. Against a constant background of established infections, epidemics of new and old infectious diseases periodically emerge, greatly magnifying the global burden of infections. For example, in during the 19th century cholera spread
 across the world from its original reservoir in the Ganges delta in India. Six subsequent pandemics killed millions of people across all continents. The current pandemic started in South Asia in 1961, and reached
 Africa in 1971 and the European Mediterranean region in the summer of 1973. Latter, the cholera epidemic that occurred the Americas in 1991. 
Studies of these emerging infections reveal the evolutionary properties of pathogenic microorganisms and the dynamic relationships between microorganisms, their hosts and the environment.  In general, fecal-oral diseases, especially cholera, have a key feature of positive feedback interaction between the  
infected human population and the concentration of bacteria (or virus) in the environment. The infectious agent is transmitted to the human population via contaminated food consumption while the infected human plays a role as a multiplicator of the infectious agent, which is then sent to the environment via fecal excretion .

The outbreak and spread of diseases have been questioned and studied for many years, and a number of mathematical models have been proposed to describe the evolution of man-environment-man epidemics. In a pioneering work \cite{CF}, Capasso and Paveri-Fontana  have proposed a mathematical model for the cholera epidemic, which was spreading in European Mediterranean regions in 1973. Their model reads as follows 
\begin{align}\label{intro1}
\begin{cases}
\dfrac{d u}{d t}=-au+\alpha v,\; t>0,\\
\dfrac{d v}{d t}=-bv+G(u), \; t>0,
\end{cases}
\end{align}
with homogeneous initial conditions.  Here $u(t)$ and $v(t)$ stand for the (average) concentration of infectious agent in the environment and the infective  population, respectively, at time $t\geq 0$; $-au$ is the nature death rate of the bacterial population; $-bv$ is the natural diminishing rate
of the infective population due to the finite mean duration of the infectious population; $\alpha v$ is the contribution of the infective humans to the growth rate of bacteria. The last term $G(u)$ is the “force of infection" on the human
population under the assumption that the total susceptible human population is constant during the evolution of the epidemic. As explained in \cite{CS}, the force of infection is a linearly monotone increasing function of the small concentration of infectious agent, but it  appears quite unrealistic that for large $u$ can stills hold. Therefore, $G$ is often assumed to be a nonlinear and stricly increasing function with constant concavity. It was shown in \cite{CF}, under this assumptions of $G$, the threshold parameter $\Theta=G'(0)\alpha/(ab)$ can determine the asymptotic behaviour of the solutions. 
If $0<\Theta<1$ then the problem \eqref{intro1} admits the only equilibrium $(0,0)$, which is globally asymptotically stable. In this case, the epidemic  eventually tends to extinction. Furthermore, if $\Theta>1$, the problem \eqref{intro1} has a unique positive constant equilibrium and the epidemic finally turns into an endemic with this nontrivial endemic level.

To make the model be more realistic, Capasso et al. \cite{CM,CK,CW} has improved the model (\ref{intro1}) to study the spreading of bacterial diseases for oral-faecal transmitted diseases  and obtained   threshold parameters such that for suitable values of it the epidemic eventually tends to extinction, otherwise a globally asymptotically stable spatially inhomogeneous stationary endemic state appears. Their model reads as follows

\begin{align}\label{intro2}
\left\{\begin{array}{ll}u_{t}=d_1\Delta u-a u+c v, & t>0, x \in \Omega, \\ v_{t}=d_2\Delta v-b v+G(u), & t>0, x \in \Omega, \\ \frac{\partial u}{\partial n}+\alpha u=0, & t>0, x \in \partial \Omega, \\ u(0, x)=u_{0}(x), v(0, x)=v_{0}(x), & x \in \overline{\Omega},
\end{array}\right.
\end{align}
where  $u(t,x)$ and $v(t,x)$ respectively denote the spatial density of the bacterial population and  the infective human population, in an urban community at time $t\geq 0$ and the point $x$ in habit region; $d_i\geq 0$ for $i=1,2$ are the diffusion coefficients;  $a$ is the natural death rate of bacteria and $cv$ is the contribution of infectious population to the density of bacteria; $b$ is the natural diminishing rate of infected individuals. 
 The nonlinear term $G(u)$ gives the “force of infection" on human due to the concentration of bacteria. It is defined by $G: \mathbb{R}\mapsto \mathbb{R}$ and satisfies the following
assumptions. 
\begin{enumerate}
\item if $0 < z' < z''$ then $0 < G\left(z'\right) < G\left(z''\right)$;
\item $G(0)=0$;
\item $G$ is continuously differentiable and, $\forall \tau\in\left(0, 1\right), \forall z\in \mathbb{R}_+ \setminus\left\{0\right\}:\; \tau G(z) < G\left(\tau z\right)$; and $\forall \rho > 0$, $\exists k_{\rho} > 0$ s.t. $\forall z', z''\in \mathbb{R},\;  0\leq z'\leq z'' < \rho:\;
G\left(z'\right) - G\left(z''\right)\geq k_{\rho}\left(z' - z''\right)$. 

\end{enumerate} 
Beside that,  it is worth mentioning that the model \eqref{intro2} with $d_2=0$ and $G$ is of bistable nonlinearity has been investigated by Xu and Zhao \cite{CM1}, who  obtained the  existence of bistable waves using results for semiflows and the method of sub- and super-solutions and further showed the global attractivity of the unique steady state up to translations of travelling waves of system  connecting the two stable nodes for the associated reaction system.
On the other hand, considering the contribution
of the infective humans to the growth rate of the bacteria is of concave nonlinearity, says  $H(v)$, Hsu and  Yang in \cite{HY} proved the existence, uniqueness, monotonnicity and asymptotic behaviour of travelling waves to the model
\begin{align}\label{intro3}
\begin{cases}
\dfrac{\partial u}{\partial t}=d_{1} \Delta u-au+H(v),\; t\geq 0,\; x\in\R,\\
\dfrac{\partial v}{\partial t}=d_{2} \Delta v-bv+G\left(u\right),\; t\geq 0,\; x\in\R,
\end{cases}
\end{align} 
where $G,\; H: \R^+\mapsto \R^+$ satisfy the following assumptions.
\begin{align}\label{conditionH_G}
\left\{\begin{array}{lll}
 H,G \in C^{2}([0, \infty)),\;\; G(0)=H(0)=0\;\;\text{and}\; H^{\prime}(z), G^{\prime}(z)>0 \; \forall z \geq 0,\\
 \text{There exists a} \;\overline{z} \;\text{such that}\; G\left(H(\overline{z})/a\right)<b\overline{z},\\
 H^{\prime\prime}(z),\;G^{\prime\prime}(z)<0 \;\text{for all}\; z\in \left(0, \infty\right).
\end{array}\right.
\end{align}
From these conditions, if $0<H'(0)G'(0)/(ab)\leq 1$ then the trivial solution is the only equilibrium of \eqref{intro3}. However, if $H'(0)G'(0)/(ab)>1$ then \eqref{intro3} has exactly two equilibria in the closure of the positive quadrant: $(0,0)$ and $(K_1,K_2)$, where $K_1$ and $K_2$ satisfy
\begin{align}\label{K}
H(K_2)=aK_1\quad \text{and} \quad G(K_1)=bK_2.
\end{align}
We further point out that the entire solution and spreading speed of  traveling  fronts for this model with two delays have been also studied in the deep work of Wu and Hsu \cite{WH}.

Newly emerging infectious diseases can be defined, due to emerging and re-emerging infectious bacteria, as infections that have  appeared in a population or have existed but are rapidly increasing in incidence or geographic range. Therefore, the study of several epidemiological models plays an important role in  understanding the spreading of disease. Many  important and interesting problems still remain open,  in particular,  answering the question on the spreading of infections to larger domain if the infectious bacteria are initially limited to a specific area is one of the major concerns of applied mathematics.
 To give the precise spreading front of the fecally–orally transmitted disease, Ahn et al. in \cite{ABL} used a free boundary to describe the moving infected region in one  dimensional space by the equation
 \begin{align}\label{intro4} \left\{\begin{array}{ll}u_{t}=d u_{x x}-a u+c v, & t>0,\; x \in\left(g(t), h(t)\right), \\ v_{t}=-b v+G(u), & t>0,\; x \in\left(g(t), h(t)\right), \\ u(t, x)=v(t, x)=0, & t>0,\; x=g(t) \text { or } x=h(t), \\ g(0)=-h_{0}, g^{\prime}(t)=-\mu u_{x}(t, g(t)), & t>0, \\ h(0)=h_{0}, h^{\prime}(t)=-\mu u_{x}\left(t, h(t)\right), & t>0, \\ u(0, x)=u_{0}(x), v(0, x)=v_{0}(x), & x \in\left[-h_{0}, h_{0}\right],\end{array}\right. \end{align} 
and showed that either the spreading or vanishing occurs depending on the largeness of basic reproduction number, initial number of bacteria, the length of the initial habitat, the diffusion rate.
Very recently, Zhao et al. \cite{WWZ1} studied the nonlocal version of \eqref{intro4} of the form
\begin{align}\label{intro6}
\left\{\begin{array}{lll}
u_t =  d_1\left[\displaystyle\int\limits_{g(t)}^{h(t)}J(x-y)u(t,y)dy - u(t,x)\right] - au(x,t) + cv(t, x), & t>0,\,\,\,\, \, g(t)<x<h\left(t\right), \\
v_t =  -bv(t,x)+ G\left(u(t,x)\right), & t>0,\,\,\,\, \, g(t)<x<h\left(t\right), \\
u\left(t, x\right) = v\left(t, x\right)=0, & t>0,\,\, x = g(t) \,\,\,\text{or}\,\,\, x = h(t),   \\  
h^{\prime}(t) = \mu \displaystyle\int\limits_{g(t)}^{h(t)}\displaystyle\int\limits_{h(t)}^{\infty}J(x-y)u(t, x)dydx  ,& t>0,\\
g^{\prime}(t) = -\mu \displaystyle\int\limits_{g(t)}^{h(t)}\displaystyle\int\limits_{-\infty}^{g(t)}J(x-y)u(t, x)dydx,& t>0,\\
-g(0) = h(0) = h_0,\,\, u(0, x) = u_0(x),\,\,v(0, x) = v_0(x),& x\in \left[-h_0, h_0\right],
\end{array}\right.
\end{align}
where the kernel function $J:\R\rightarrow \R$ and the initial conditions  are assumed to satisfy
\begin{align*}
\begin{array}{lll}
J \in C(\mathbb{R}) \cap L^{\infty}(\mathbb{R}),\; J \text{ is symmetric and nonnegative},\; J(0)>0,\; \int_{\mathbb{R}} J(x) d x=1,\\
u_{0}, v_{0} \in C\left(\left[-h_{0}, h_{0}\right]\right), u_{0}\left(\pm h_{0}\right)=v_{0}\left(\pm h_{0}\right)=0, u_{0}, v_{0}>0 \text{ in } \left(-h_{0}, h_{0}\right).
\end{array}
\end{align*}
They obtained the following results:
\begin{align*}
\begin{array}{lll}
\text{(i)}\; \text{If} \; R_{0} \leq 1, \text{then vanishing happens}.\\
\text{(ii)}\; \text{If}\; R_{0} \geq 1+\frac{d}{a},\;\text{then spreading happens}.\\ 
\text{(iii)}\; \text{If}\; 1<R_{0}<1+\frac{d}{a}, \; \text{then there exists}\; l^{*}>0\;\text{such that spreading happens when}\; 2 h_{0} \geq l^{*},\;\text{ and if}\;\\ \hspace{0.5cm} 2 h_{0}<l^{*}, \;\text{then there exists}\; \mu^{*}>0\; \text{such that spreading happens if and only if}\; \mu>\mu^{*}.
\end{array}
\end{align*}

 Motivated by the works \cite{CF, CM, HY, WWZ1, LX, WH}, in this paper, we investigate the following cooperative free boundaries system with  nonlocal dispersals read by
 \begin{align}\label{main} \left\{\begin{array}{lll} u_t =  d_1\left[\displaystyle\int\limits_{g(t)}^{h(t)}J_1(x-y)u(t,y)dy - u(t,x)\right] - au(x,t) + H\left(v(t, x)\right), & t>0,\,\,\,\, \, g(t)<x<h\left(t\right), \\ v_t = d_2\left[\displaystyle\int\limits_{g(t)}^{h(t)}J_2(x-y)v(t,y)dy - v(t,x)\right] -bv(t,x)+ G\left(u(t,x)\right), & t>0,\,\,\,\, \, g(t)<x<h\left(t\right), \\ u\left(t, x\right) = v\left(t, x\right)=0, & t>0,\,\, x = g(t) \,\,\,\text{or}\,\,\, x = h(t),   \\ h^{\prime}(t) = \mu\left( \displaystyle\int\limits_{g(t)}^{h(t)}\displaystyle\int\limits_{h(t)}^{\infty}J_1(x-y)u(t, x)dydx  +\rho\displaystyle\int\limits_{g(t)}^{h(t)}\displaystyle\int\limits_{h(t)}^{\infty}J_2(x-y)v(t, x)dydx\right),& t>0,\\ g^{\prime}(t) = -\mu\left( \displaystyle\int\limits_{g(t)}^{h(t)}\displaystyle\int\limits_{-\infty}^{g(t)}J_1(x-y)u(t, x)dydx- \rho\displaystyle\int\limits_{g(t)}^{h(t)}\displaystyle\int\limits_{-\infty}^{g(t)}J_2(x-y)v(t,x)dydx\right),& t>0,\\ -g(0) = h(0) = h_0,\,\, u(0, x) = u_0(x),\,\,v(0, x) = v_0(x),& x\in \left[-h_0, h_0\right], \end{array}\right. 
\end{align}
where $x = g(t)$ and $x = h(t)$ are the moving boundaries to be determined together with $u(t, x)$ and $v(t, x)$, which are always assumed to be identically $0$ for $x \in \R\setminus\left[g(t), h(t)\right]$; $a, b, d_1, d_2, \mu, \rho$ given real positive constants.
Throughout this paper, we assume that the initial function pair $\left(u_0, v_0\right)$ satisfies
\begin{align}\label{initial}
u_0, v_0\in C\left(\left[-h_0, h_0\right]\right), \, u\left(\pm h_0\right)= v\left(\pm h_0\right),\,\, u_0, v_0>0\,\,\,\text{in}\,\,\, \left(-h_0, h_0\right),
\end{align}
 and $G, H: \R\mapsto \R$ satisfy  assumptions  \eqref{conditionH_G}.
 
A model example of the nonlinearities we consider is 
\begin{align*}
G(z)= \dfrac{z}{1+z}\;\;\text{and}\;\; H(z) = z-z^2,\;\forall z\geq 0.
\end{align*}
The dispersal kernels $J_i: \R\to \R$ satisfies the following assumption.
\begin{align*}
\begin{array}{lll}
({\bf J})\; J_i\in C\left(\R\right)\bigcap L^{\infty}\left(\R\right)\; \text{is nonnegative, symmetric}\; \text{such that}\;\; J_i(0) > 0, \; \displaystyle\int_{\R} J_i(x)dx = 1,\; i=1,2.
\end{array}
\end{align*}
On the epidemiological aspect, one can use (\ref{main}) to model the evolution of the bacteria and infective human  population comes from
the observation that the intrinsic variability in the capacity of the individuals to disperse
generates, at the scale of  populations, a long range dispersal of the population. The different dispersal kernels $J_i$, $i\in\{1,2\}$, describe the different probabilities of the bacteria and infective human to jump from one location to another.

One of the main tool to study the problem (\ref{main}) is the spectral theory of the linearized operators. It is known that a random dispersal operator always admits a principal eigenvalue, but a nonlocal dispersal operator may not have a principal eigenvalue depending on the smoothness of the coefficient near the maximum point. The principal spectral theory for elliptic-type nonlocal dispersal operator and their properties have been extensively investigated in \cite{BCV, BCV1, SX, BLS } and references therein. In particular, Berestycki-Coville-Vo  \cite{ BCV} proved  the equivalence of different definitions of the principal eigenvalue and used it to investigate several important qualitative properties of the eigenvalue while Shen and Xie proved in \cite{SX} a necessary and sufficient spectral condition for the existence of the principal eigenvalue using a dynamical system approach. The existence and simplicity for time periodic eigenvalue of nonlocal cooperative system were proved by Bao and Shen \cite{BS} based on the Krein-Rutmann theorem. However, the existence of the principal eigenvalue for elliptic-type nonlocal dispersals cooperative system has not been obtained. In this paper, in the spirit of \cite{ BCV}, we shall first prove the existence and its variational formula of the principal eigenvalue  for linearized system to (\ref{main}), the eigenvalue problem reads as follow
\begin{align}\label{eq:PEV}
\mathbf{D}\left(\mathbf{N}\pmb{\varphi}-\pmb{\varphi}\right)+\mathbf{A}\pmb{\varphi}+\lambda \pmb{\varphi}=0,\;\text{in}\;{\bf E},
\end{align}
where $\mathbf{D} = \text{diag}\left(d_1, d_2\right),$ $d_1, d_2$ are positive constants; $\mathbf{N}=\text{diag}\left(\CN_1,\; \CN_2\right),$ with $\CN_i$ being nonlocal operators, namely $\CN_i[\varphi_i](x): = 
	\displaystyle\int_{\O} J_i(x-y)\varphi_i(y)dy, \;\text{for}\; i=1,2$; $\pmb{\varphi}=\left(\varphi_1,\varphi_2\right)^T\in \mathbf{E}$; $\mathbf{E} = L^2\left(\Omega\right)\times L^2\left(\Omega\right)$; $\Omega \subset \R$ is a bounded domain; $\mathbf{A}= \left(a_{ij}\right)\in M_{2\times 2}\left(\R\right),\, i,j=1,2$, satisfies $a_{ij}\geq 0$, and $a_{ij}=a_{ji}$ when $i\neq j$. We emphasize that, due to the non-compactness of nonlocal operators and their
resolvents, one cannot use the Krein-Rutmann theorem as in \cite{BS}. We need to employ new idea based on Lax-Milgram's theorem to prove the existence of the principal eigenvalue, moreover, if one denotes by $\lambda_p=\lambda_p\left(-{\bf D(N-I)}  - {\bf A}\right) $  the principal eigenvalue of problem  \eqref{eq:PEV}, we further show that $\lambda_p$ can be variationally characterized by 
  $$\lambda_p= -\sup_{\left\|\pmb{\varphi}\right\|_{E} =1}\left \langle {\bf D(N-I)}\pmb{\varphi}  + {\bf A}\pmb{\varphi}, \pmb{\varphi}\right \rangle,$$ where $\left \langle \cdot, \cdot \right\rangle$ denotes the scalar product of $L^2\left(\Omega\right)$. 
Thanks to this variational formula, we further obtain the limiting properties of the principal eigenvalue as dispersal rates and domain tend to zero or infinity.

In fact, mathematical modelling plays an important role in helping to quantify possible disease control strategies by focusing on the important aspects of a disease, determining threshold quantities for disease survival, and evaluating the effect of particular control strategies. A very important threshold quantity is the basic reproduction number, which is usually denoted by $\CR_0$. The epidemiological definition of $\CR_0$ is the average number of secondary cases produced by one infected individual introduced into a population of susceptible individuals, where an infected individual has acquired the disease, and susceptible individuals are healthy but can acquire the disease. 
In this paper, we introduce the threshold dynamical behaviour of the problem \eqref{main}.  This is followed by the investigation of the global dynamics in terms of the basic reproduction numbers $\CR_0$ and $\CR^*$ defined to be the method of next generation matrix. More precisely, we define
\begin{align}\label{1.05}
\CR_0 = \dfrac{H^{\prime}(0)G^{\prime}(0)}{ab},
\end{align}
and
\begin{align}\label{1.06}
 \CR^{*}= \dfrac{H'(0)G'(0)}{(a+d_1)(b+d_2)}.
 \end{align}
 In reality, the value of $\CR_0$ for a specific disease depends on many variables, such as location and density of population.
 
Let us now state our results. The first result concerns  the existence and uniqueness of solutions to problem \eqref{main}. This  is a natural extension of the previous models in previously prescribed works (see e.g. \cite{ABL,CDLL,DN,WWZ1}) and it is not a trivial step due to the presence of couple dispersals and new types of nonlinear reaction terms.

\begin{theorem}[Global existence and uniqueness]\label{theo_1} 
Assume ${(\bf J)}$ and \eqref{conditionH_G}. Then for any given $h_0 > 0$ and $u_0(x), v_0(x)$ satisfying \eqref{initial}, problem \eqref{main} admits a unique solution
$\left(u(t,x), v(t,x), g(t), h(t)\right)$ defined for all $t > 0$.
\end{theorem}
The second result describes  a dichotomy of the bacteria $u$ in the environment and the infective human population $v$, respectively. We prove that $(u,v)$ converges either to the disease-free equilibrium or the endemic equilibrium in the long run. Epidemiologically,  this result  reveals that the bacteria and the infective human can either be co-existence in the new adaptive environment (called spreading) or fail to establish and both eventually go extinct  (called vanishing).
\begin{theorem}[Spreading-Vanishing dichotomy]\label{theorem_1.2}
Assume ${\bf (J)}$, \eqref{conditionH_G}, and \eqref{initial}. Let $\left(u, v, g, h\right)$ be the solution of \eqref{main}, then
one of the following alternative  happens :
\begin{enumerate}
\item {\bf{Vanishing}}: $\lim\limits_{t\to\infty}h(t)-g(t)<\infty$ and
\begin{align*}
\lim\limits_{t\to\infty}\left(u(t, x), v(t, x)\right)=\left(0, 0\right)\;\;\text{locally\;uniformly\;in}\;\left[g(t), h(t)\right],
\end{align*}
\item {\bf{Spreading}}: $\lim\limits_{t\to \infty}h(t)=-\lim\limits_{t\to \infty}g(t)=\infty$ and
\begin{align*}
\lim\limits_{t\to\infty}\left(u(t, x), v(t, x)\right)=\left(K_1, K_2\right)\;\;\text{locally\;uniformly\;in}\;\R.
\end{align*}
\end{enumerate}
\end{theorem}
Next we derive a threshold-type dynamics for system \eqref{main} in terms of $\CR_0$ and $\CR^*$. More specifically, we will prove that vanishing
always happens if the basic reproduction number $\CR_0\leq 1$, while if $\CR^*\geq 1$, then spreading always happens. Moreover, in a mediate case $\CR^*< 1< \CR_0$, the spreading/vanishing will depend largeness of initial habitat, more precisely, we found a critical length $\CL^*$ such that the spreading happens if  $h_0\geq \CL^*/2$, otherwise, the we found a critical spreading speed on the boundary $\mu^*$ such that the spreading occurs  if and only if $\mu>\mu^*$. More precisely, our  result reads as follows :
\begin{theorem}\label{theorem1.4}
Assume ${\bf (J)}$, \eqref{conditionH_G},  and \eqref{initial}. Let $\left(u, v\right)$ be the solution of \eqref{main}, and $\CR_0, \; \CR^*$ be given by \eqref{1.05}, \eqref{1.06}, respectively. The following statements
hold.
\begin{enumerate}
\item If $\CR_0 \leq 1$, then vanishing happens.\\
\item If $\CR^* \geq 1$, then spreading always happens.\\
\item If $\CR^*< 1< \CR_0$ then there exists a unique $\CL^*>0$ such that\\
$\bullet$ Spreading always happens if $h_0\geq \CL^*/2$.\\
$\bullet$ On the other hand, if $0<h_0<\CL^*/2$, there exists $\mu^*>0$ such that the spreading occurs if and only if $\mu>\mu^*$, otherwise, the vanishing happens exactly when $\mu\leq \mu^*$.
\end{enumerate}
\end{theorem}
Note that $\CL^*$ depends only on $(a,b,H'(0),G'(0),d_1,d_2,J_1,J_2)$, which is determined by an eigenvalue problem, but $\mu^*$ depends only on the initial data.

\begin{remark}
 The condition $\CR^*\geq 1$ in part (2) of Theorem \ref{theorem1.4} is equivalent to
	\begin{align*}
	\CR_0\geq \left(1+\dfrac{d_1}{a}\right)\left(1+\dfrac{d_2}{b}\right).
	\end{align*}
	This confirms the impact of the dispersal rate of the bacteria to the spreading phenomena of both infective human and bacteria. Moreover, if $d_2=0$, we also recover the spreading condition obtained by Zhao, Zhang, Li, Du \cite{WWZ1} for degenerate system of epidemic model.
\end{remark}

To this end, our last result is about the sharp criteria for spreading/vanishing phenomena with respect to the dispersal rate  $d\left(=d_1=d_2\right)$. More precisely, with some extra condition, we found a threshold value $d^*$ such that the spreading hold if  $d\in(0,d^*)$ while vanishing occurs if $d>d^*$. 

\begin{theorem}\label{theorem_1.5}
Under the condition of Theorem \ref{theorem1.4}, if $0<d<d^*$ and $\CR^*\geq 1$ then spreading always happens. If $d>d^*$, $\CR^*< 1<\CR_0$,  and $\left(u_0(x), v_0(x)\right)$ $(\text{satisfying}\;\; \eqref{initial})$ is sufficiently small, then vanishing occurs.
\end{theorem}
From the biological point of view, this theorem implies that the slow dispersal rate is always beneficial for persistence of bacteria in the long run, while the large dispersal rate together with small initial density drive the
bacteria to extinction. This type of result was also obtainted by Zhou and Xiao in \cite{ZX} for single-species model in heterogeneous environment, the authors \cite{ZX} found critical dispersal rate such that the species is spreading when  $0 < d \leq d^*$, the spreading speed is also obtained. The case of nonlocal time periodic equation has been investigated by Shen and Vo  \cite{SV1}, in which the authors asserted that under the seasonal effect, the small dispersal rates are favored, while the large dispersal rates are always unfavored and leads the species to extinction. This is the first time that the result "slow  dispersal  rates always win"  is obtained for an epidemic model of couple dispersals and again confirmed this conjecture proposed by Yuan Lou in the series of works \cite{DL,LL, JLLLW}.

{\bf {Organization of the paper}}. The paper is organized as follows. In Sect. \ref{sec.2} we prove Theorem \ref{theo_1}  by using the maximum principle and contraction mapping theorem.  Section \ref{sec.3} is devoted to the study of the eigenvalue problem \eqref{eigenproblem}, which plays the vital role in our analysis.   In Sect. \ref{sec.4}, we establish the spreading-vanishing dichotomy result, and furthermore, we obtain the sharp criteria for the spreading and vanishing with respect to the initial condition. We prove Theorem \ref{theorem_1.2} and Theorem \ref{theorem1.4} in this section. Finally in Sect. \ref{sec.5}, we study the influences of the principal spectrum point $\lambda_p\left(d\right)$ on the global dynamics and prove Theorem \ref{theorem_1.5} here. 
\section{Global existence and Uniqueness}\label{sec.2}
Throughout this section, we assume that $h_0 > 0$ and $\left(u_0,v_0\right)$ satisfy \eqref{initial}. For any given
$T > 0$, we introduce the following notations:
\begin{align*}
&\mathbb{H}_{h_{0}}^{T}=\left\{h \in C([0, T]): h(0)=h_{0}, h(t) \text { is strictly increasing }\right\},\\& \mathbb{G}_{h_{0}}^{T}=\left\{g \in C([0, T]):-g \in \mathbb{H}_{h_{0}}^{T}\right\}.
\end{align*}
For  $g \in \mathbb{G}_{h_{0}}^{T}, h \in \mathbb{H}_{h_{0}}^{T} $  and  $W_{0}=\left(u_{0}, v_0\right)$ satisfying $\eqref{initial}$
\begin{align*}
&D_{T}=D_{g, h}^{T}:=\left\{(t, x) \in \mathbb{R}^{2}: 0<t \leq T, g(t)<x<h(t)\right\},\\
&\mathbb{X}_{W_{0}}^{T}=\mathbb{X}_{W_{0}, g, h}^{T}:=\left\{\varphi \in\left[C\left(\overline{D}_{T}\right)\right]^{2}: \varphi \geq 0,\left.\varphi\right|_{t=0}=W_{0}(x),\left.\varphi\right|_{x=g(t), h(t)}=0\right\}.
\end{align*}
Furthermore, for $\varphi=\left(\varphi_{1}, \varphi_{2}\right) \geq 0$, we mean $\varphi_{1} \geq 0$ and $\varphi_2\geq 0$ in $D_T$. According to assumptions \eqref{conditionH_G}, when $H'(0)G'(0)>ab$ the following problem
\begin{align*}
\begin{cases}
H(z)=az,\\
G(z)=bz,
\end{cases}
\end{align*}
has exactly two solutions $(0,0)$ and $(K_1,K_2)$. Moreover, $H(z)/z$ and $G(z)/z$ are non-increasing functions due to \eqref{conditionH_G}.

 Next, we give two important constants whose definition is divided into two significant cases
\begin{align}\label{mixi}
C_1=\left\{\begin{array}{lll}
\max\left\{\|u_0\|,\dfrac{K_1}{K_2}\|v_0\|,K_1\right\},&\text{ if}\,\ ab<H'(0)G'(0),\\
\max\left\{\|u_0\|,\dfrac{b}{G'(0)}\|v_0\|,K_1\right\},&\text{ if}\,\ ab\geq H'(0)G'(0), 
\end{array}\right.  C_2=\left\{\begin{array}{lll}
\dfrac{K_2}{K_1}C_1, &\text{if}\,\ ab<H'(0)G'(0),\\
\dfrac{G'(0)}{b}C_1,& \text{if}\,\ ab\geq H'(0)G'(0).
\end{array}\right.
\end{align}
\begin{lemma}[Maximum principle] Assume that $({\bf J})$  holds, $A, B, C, D\in L^\infty\left(D_{g,h}^T\right)$ with $B$ and $C$ nonnegative, $ (u(t,x), v(t,x))$ as well as
$\left(u_t(t,x), v_t(t,x)\right)$ are continuous in $\overline{D}_{g,h}^T$  and satisfy
\begin{equation}\label{maxi1}
\left\lbrace\begin{array}{ll}
u_t(t, x)\geq d_1\displaystyle\int_{g(t)}^{h(t)}J_1(x-y)u(t, x)dy -d_1u + Au + Bv &(t,x)\in \O^T,\\
v_t(t, x)\geq d_2\displaystyle\int_{g(t)}^{h(t)}J_2(x-y)v(t, x)dy -d_2v  + Dv + Cu &(t,x)\in \O^T,\\
u(t, x)\geq0,\,\,\ v(t, x)\geq0,& 0<t\leq T,\,\, x = g(t)\,\,\text{or}\,\,\ x = h(t),\\
u(0,x)\geq0,\,\,\ v(0,x)\geq0,&\left|x\right|\leq h_0.
\end{array}\right.
\end{equation}
 Then $\left( u(t,x), v(t, x)\right)\geq \left(0, 0\right)$ for $(t,x)\in D^T_{g,h}$. Moreover, if $\left(u\left(0, x\right), v\left(0, x\right)\right)\not\equiv \left(0, 0\right)$ for $|x|\leq h_0$ then $\left(u(t, x), v(t, x)\right)>0$ in $D^T_{g,h}.$ 
\end{lemma}
\begin{proof}[\bf Proof]
The lemma can be proved by the comparison arguments similar to the proof of Lemma 3.1 in \cite{DN}. Here we omit the
details.
\end{proof}

\begin{lemma}\label{lem_2.2}
	For any $T>0$ and $(g, h) \in \mathbb{G}_{h_{0}}^{T} \times \mathbb{H}_{h_{0}}^{T}$, the problem
	\begin{align}\label{local1}
	\left\{\begin{array}{ll}
	u_{t}=d_1\displaystyle\int_{g(t)}^{h(t)} J_1(x-y)u(t, y)dy-d_1 u-au+H(v), & t>0, x \in(g(t), h(t)), \\
	v_{t}=d_2\displaystyle\int_{g(t)}^{h(t)} J_2(x-y) v(t, y)dy-d_2v-bv+G(u), & t>0, x \in(g(t), h(t)),\\
	u(t, x)=v(t, x)=0, & t>0, x=g(t) \text { or } x=h(t), \\
	-g(0)=h(0)=h_{0}, u(0, x)=u_{0}(x), v(0, x)=v_{0}(x), & x \in\left[-h_{0}, h_{0}\right],
	\end{array}\right.
	\end{align}
	has a unique solution $(u_{g,h},v_{g,h})\in \mathbb{X}^T_{U_0,g,h}$. Furthermore,
	\begin{align}\label{b}
	0<(u_{g,h},v_{g,h})\leq (C_1,C_2) \text{ for any } (t,x)\in D^{T}_{g,h},
	\end{align}
where $C_1$ and $C_2$ are defined in $\eqref{mixi}$.
\end{lemma}
\begin{proof}[\bf Proof]
We break the proof into three steps.

\textbf{Step 1:} \textit{A parametrized ODE problem}.

For  $x\in [g(T),h(T)]$, we define
\begin{align*}
\widehat{u}_{0}(x):=\left\{\begin{array}{ll}
u_{0}(x), & x \in\left[-h_{0}, h_{0}\right], \\
0, & x \notin\left[-h_{0}, h_{0}\right],
\end{array}\right.\;\;\;
\widehat{v}_{0}(x):=\left\{\begin{array}{ll}
v_{0}(x), & x \in\left[-h_{0}, h_{0}\right], \\
0, & x \notin\left[-h_{0}, h_{0}\right],
\end{array}\right.
\end{align*}
and
\begin{align*}
t_{x}:=\left\{\begin{array}{ll}
t_{x}^{g}, & x \in\left[g(T),-h_{0}\right) \text { and } x=g\left(t_{x}^{g}\right), \\
0, & x \in\left[-h_{0}, h_{0}\right], \\
t_{x}^{h}, & x \in\left(h_{0}, h(T)\right] \text { and } x=h\left(t_{x}^{h}\right).
\end{array}\right.
\end{align*}
Clearly $t_x=T$ for $x=g(T)$ and $x=h(T)$, and $t_x<t\leq T$ for $x\in (g(t),h(t))$. 
For any given $0<s\leq T$ and $\varphi=\left(\varphi_{1}, \varphi_{2}\right) \in \mathbb{X}_{W_{0}}^{s}$, we first consider the initial value problem of
the following ordinary differential system with parameter $x\in (g(t),h(t))$.
\begin{align}\label{phi}
\left\{\begin{array}{ll}
U_{t}=d_1\displaystyle\int_{g(t)}^{h(t)} J_1(x-y) \varphi_1(t,y)dy-d_1 U-aU+\widetilde{H}(V), & t_x<t\leq s, \\
V_{t}=d_2\displaystyle\int_{g(t)}^{h(t)} J_2(x-y) \varphi_2(t, y)dy-d_2V-bV+\widetilde{G}(U), & t_x<t\leq s, \\
U(t_x, x)=\widehat{u}_0(x),V(t_x, x)=\widehat{v}_0(x), & g(s)<x<h(s),
\end{array}\right.
\end{align}
where 
\begin{align*}
\widetilde{H}(V)=\begin{cases}
0,&\text{ for } V<0,\\
H(V),&\text{ for } V\geq 0,
\end{cases}\;\;\text{and}\;\;\; \widetilde{G}(U)=\begin{cases}
0,&\text{ for } U<0,\\
G(U),&\text{ for } U\geq 0.
\end{cases}
\end{align*}
Denote\\
\begin{align*}
&F_1(t,x,U,V)=d_1\displaystyle\int_{g(t)}^{h(t)} J_1(x-y) \varphi_1(t,y)dy-d_1 U(t,x)-aU(t,x)+\widetilde{H}(V(t,x)),\\
&F_2(t,x,U,V)=d_2\displaystyle\int_{g(t)}^{h(t)} J_2(x-y) \varphi_2(t, y)dy-d_2V(t,x)-bV(t,x)+\widetilde{G}(U(t,x)),
\end{align*}
and
\begin{align*}
&L^{\varphi}_1=\begin{cases}
1+K_1+\|\varphi_1\|_{C(\overline{D}_s)}+\dfrac{K_1}{K_2}\|\varphi_2\|_{C(\overline{D}_s)},& \text{ if } ab<H'(0)G'(0),\\
\|\varphi_1\|_{C(\overline{D}_s)}+\dfrac{b}{G'(0)}\|\varphi_2\|_{C(\overline{D}_s)}+1,& \text{ if } ab\geq H'(0)G'(0),
\end{cases}\\
&L^{\varphi}_2=\begin{cases}
\dfrac{K_2}{K_1}L^{\varphi}_1=\dfrac{K_2}{K_1}+K_2+\dfrac{K_2}{K_1}\|\varphi_1\|_{C(\overline{D}_s)}+\|\varphi_2\|_{C(\overline{D}_s)}, & \text{ if } ab<H'(0)G'(0),\\ \newline\\
\dfrac{G'(0)}{b}L^{\varphi}_1=\dfrac{G'(0)}{b}\|\varphi_1\|_{C(\overline{D}_s)}+\|\varphi_2\|_{C(\overline{D}_s)}+\dfrac{G'(0)}{b},& \text{ if } ab\geq H'(0)G'(0).
\end{cases}
\end{align*}
Since $G,H\in C^1([0,+\infty))$, with $L_1^{\varphi},L_2^{\varphi}>0$, there exists two constants $\rho(L_1^{\varphi}),\rho(L_2^{\varphi})>0$ such that
\begin{align*}
|\widetilde{G}(z_1)-\widetilde{G}(z_1)|\leq \rho(L_1^{\varphi})\left|z_1-z_2\right|\; \text{ for } z_1,z_2\in \left[0, L_1^{\varphi}\right],
\end{align*}
and 
\begin{align*}
\left|\widetilde{H}(z_1)-\widetilde{H}(z_1)\right|\leq \rho(L_2^{\varphi})\left|z_1-z_2\right|\;\text{ for } z_1,z_2\in [0,L_2^{\varphi}].
\end{align*}
Setting $L^\varphi=\max\{a+d_1,b+d_2,\rho(L_1^{\varphi}),\rho(L_2^{\varphi})\}$. For any $(U_1,V_1),(U_2,V_2)\in (0,L_1^\varphi]\times(0,L_2^\varphi]$, we have
\begin{align*}
\left|F_1(t,x,U_1,V_1)-F_1(t,x,U_1,V_1)\right|&\leq (a+d_1)\left|U_1-U_2\right|+\left|\widetilde{H}\left(V_1\right)-\widetilde{H}\left(V_2\right)\right|
\\
&\leq (a+d_1)\left|U_1-U_2\right|+\rho(L_2^\varphi)|V_1-V_2| \leq  L^\varphi\left(\left|U_1-U_2\right|+\left|V_1-V_2\right|\right),
\end{align*}
and
\begin{align*}
\left|F_2(t,x,U_1,V_1)-F_2(t,x,U_1,V_1)\right|&\leq (b+d_2)\left|V_1-V_2\right|+\left|\widetilde{G}\left(U_1\right)-\widetilde{G}\left(U_2\right)\right|
\\
&\leq (b+d_2)\left|V_1-V_2\right|+\rho(L_1^\varphi)|U_1-U_2| \leq  L^\varphi\left(\left|U_1-U_2\right|+\left|V_1-V_2\right|\right).
\end{align*}
Hence, the function $F_i(t,x,U,V)$ is Lipschitz continuous in $W=(U,V)$ for $(U,V)\in \left(0, L_1^\varphi\right]\times\left(0, L_2^\varphi\right]$ with Lischitz constant $L^\varphi$, uniformly for $t\in [0,s]$ and $x\in [g(s), h(s)]$, $i=1,2$. Moreover, $F_i\left(t, x, U, V\right)$ is continous in all its variable in this range, $i=1,2$. By the fundamental theorem of ODEs, problem $\eqref{phi}$ has a unique solution $W^\varphi=\left(U^\varphi, V^\varphi\right)$ defined in some interval $\left[t_x, s\right)$ and  $W^\varphi=\left(U^\varphi, V^\varphi\right)$ is continuous in both $t$ and $x$.\\

To claim that $t\rightarrow W^\varphi(\cdot, x)$ can be uniquely extended to $[t_x,s]$, it suffices to show that if $W^\varphi$ is defined for $t\in \left[t_x, t_0\right]$ with $t_0\in \left(t_x, s\right]$ then
\begin{align}\label{c}
0\leq W^\varphi=\left(U^\varphi, V^\varphi\right)<\left(L_1^\varphi, L_2^\varphi\right) \text{ for } t\in \left(t_x, t_0\right).
\end{align}
Here by $\left(U^\varphi, V^\varphi\right)<\left(L_1^\varphi, L_2^\varphi\right)$ we mean $U^\varphi<L_1^\varphi$ and $V^\varphi<L_2^\varphi$.
We claim that $\left(U^\varphi, V^\varphi\right)<\left(L_1^\varphi, L_2^\varphi\right)$ in $\left(t_x, t_0\right]$. If this inequality does not hold throughout the interval $\left(t_x, t_0\right]$ and hence, in view of 
\begin{align*}
\left(U^\varphi\left(t_x, x\right),V^\varphi(t_x,x)\right)=(\widehat{u}_0(x),\widehat{v}_0(x))\leq (\|\varphi_1\|_{C(\overline{D}_s)},\|\varphi_2\|_{C(\overline{D}_s)})<(L_1^\varphi,L_2^\varphi),
\end{align*} 
there exists a value $t'\in \left(t_x, t_0\right]$ such that $\left(U^\varphi, V^\varphi\right)<\left(L_1^\varphi,L_2^\varphi\right)$ in $\left(t_x, t'\right)$, $\left(U^\varphi(t', x), V^\varphi(t', x) \right)\leq \left(L_1^\varphi,L_2^\varphi\right)$, and $U^\varphi(t', x)=L_1^\varphi$ or $V^\varphi(t', x)=L_2^\varphi$. Without loss of generality, we may assume that 
$U^\varphi(t',x)=L_1^\varphi,$
then it follows that $(U^\varphi)_t(t',x)\geq 0$. Now, we split into two cases for the convenience of pointing out a contradiction. \\

\textbf{Case 1:} $ab<H'(0)G'(0)$.

Since $H(z)$ is increasing and $V^{\varphi}(t',x)\leq L_2^{\varphi}$, we obtain  $\widetilde{H}(V^{\varphi}(t',x))\leq \widetilde{H}(L_2^{\varphi})$. Due to $L_2^{\varphi}>K_2$ and $H(z)/z$ is non-increasing, there holds $\dfrac{\widetilde{H}(L_2^{\varphi})}{L_2^{\varphi}}\leq \dfrac{H(K_2)}{K_2}=\dfrac{aK_1}{K_2} $. Hence, from the equation satisfied by $U^\varphi$, we can deduce
\begin{align*}
0&\leq d_1\int_{g(t')}^{h(t')} J_1(x-y) \varphi_1(t,y)dy-d_1 U^{\varphi}(t',x)-aU^{\varphi}(t',x)+\widetilde{H}(V^{\varphi}(t',x))\\
&\leq d_1\|\varphi_1\|_{C(\overline{D}_T)}-(a+d_1)L_1^\varphi+\dfrac{aK_1}{K_2}L_2^\varphi < d_1L_1^\varphi-(a+d_1)L_1^\varphi+\dfrac{aK_1}{K_2}L_2^\varphi=-aL_1^\varphi+a\dfrac{K_1}{K_2}L_2^\varphi=0.
\end{align*} 
This is impossible.\\

\textbf{Case 2:} $ab\geq H'(0)G'(0)$.

Due to $H(z)/z$ is non-increasing, we have $\widetilde{H}(L_2^{\varphi})\leq H'(0)L_2^{\varphi}$. Combining this with the equation satisfied by $U^\varphi$, we have similarly a contradiction
\begin{align*}
0&\leq d_1\int_{g(t')}^{h(t')} J_1(x-y) \varphi_1(t,y)dy-d_1 U^{\varphi}(t',x)-aU^{\varphi}(t',x)+\widetilde{H}(V^{\varphi}(t',x))\\
&\leq d_1\|\varphi_1\|_{C(\overline{D}_T)}-(a+d_1)L_1^\varphi+H'(0)L_2^\varphi < d_1L_1^\varphi-(a+d_1)L_1^\varphi+\dfrac{H'(0)G'(0)}{b}L_1^\varphi=\dfrac{H'(0)G'(0)-ab}{b}L_1^\varphi\leq 0.
\end{align*} 
We now prove $(U^\varphi, V^\varphi)\geq 0$ in $(t_x,t_0]$.
Note that $\widetilde{H}(z)\geq 0$ for $z\in \mathbb{R}$. Hence, from the equation satisfied by $U^\varphi$, we can deduce
\begin{align*}
U^\varphi(t,x)\geq d_1\int_{g(t)}^{h(t)} J_1(x-y) \varphi_1(t,y)dy-d_1 U^{\varphi}-aU^{\varphi}\geq (-a-d_1)U^{\varphi}(t,x).
\end{align*}
Therefore,
$\left(U^\varphi(t,x)e^{\left(a+d_1\right)t}\right)_t\geq 0, \text{ for } t\in (t_x,t_0],$
and  $U^\varphi(t,x)e^{\left(a+d_1\right)t}\geq U^\varphi\left(t_x, x\right)e^{\left(a+d_1\right)t_x}
\geq 0$.
This implies that $U^\varphi(t,x)\geq 0,\text{ for }t\in \left(t_x, t_0\right]$. Similarly, we have $V^{\varphi}(t,x)\geq 0\text{ for }t\in \left(t_x, t_0\right]$.
This proves our claim in $\eqref{c}$, and therefore the solution $W^\varphi$ of $\eqref{phi}$ is uniquely defined for $t\in \left[t_x, s\right]$.

\textbf{Step 2:} \textit{A fixed point problem.}

Let us note that $W^\varphi(0,x)=W_0(x)$ for $x\in \left[-h_0, h_0\right]$, and $W^\varphi(t,x)=0$ for $t\in [0,s]$ and $x\in {\left(g(t),h(t)\right)}$. Moreover, by the continous dependence of the ODE solution on parameters, $W^\varphi$ is continous in $\overline{D_s}$, and so $W^\varphi\in \mathbb{X}^s_{W_0}$. Define a mapping $\Gamma_s: \mathbb{X}^s_{W_0}\rightarrow \mathbb{X}^s_{W_0}$ by
\begin{align*}
\Gamma_s\varphi=W^\varphi.
\end{align*}
Clearly, if $\Gamma_s\varphi=\varphi$ then $\varphi$ solves $\eqref{phi}$, and vice versa.\\
We claim that for sufficiently small $s>0$, $\Gamma_s$ has a unique fixed point in $\mathbb{X}^s_{W_0}$. We derive this conclusion by the contraction mapping theorem; namely we prove that for such s, $\Gamma_s$ is a contraction mapping on a closed subset of $\mathbb{X}^s_{W_0}$, and any fixed point of $\Gamma_s$ in $\mathbb{X}^s_{W_0}$ lies in this closed subset.\\
Firstly we note that $\mathbb{X}^s_{W_0}$ is a complete metric space with a metric
\begin{align*}
d(\varphi,\phi)=\|\varphi-\phi\|_{C^2(\overline{D}_T)}.
\end{align*}
Take $C_1$ and $C_2$ as in \eqref{mixi}. Denote
\begin{align*}
&\alpha=\left\{\begin{array}{lll}
\dfrac{K_2}{K_1}, &\text{ if } ab<H'(0)G'(0),\\
\dfrac{G'(0)}{b}, & \text{ if } ab\geq H'(0)G'(0), 
\end{array}\right.,\;\; C=\left(C_1,C_2\right),
\end{align*} 
and define
\begin{align*}
\mathcal{X}_{C}^s:=\{\varphi=\left(\varphi_1,\varphi_2\right)\in \mathbb{X}_{U_0}^s:\left(\|\varphi_1\|_{C(\overline{D}_s)},\|\varphi_2\|_{C(\overline{D}_s)}\right)\leq \left(C_1, C_2\right)\}.
\end{align*}
Clearly $\mathcal{X}_C^s$ is a closed subset of $\mathbb{X}_{W_0}^s$. We will find a $\delta>0$ small depending on $C_1,C_2$ such that for every $s\in \left(0,\delta\right]$, $\Gamma_s$ maps $\mathcal{X}_C^s$ into itself and is a contraction.\\
Let $\varphi\in \mathcal{X}_C^s$ and denote $W^\varphi=\Gamma_s\varphi$. Then $W^\varphi=\left(U^\varphi, V^\varphi\right)$ solves $\eqref{phi}$ and satisfies $\eqref{c}$ with $t_0$ replaced by $s$. Now we prove that 
\begin{align}\label{c2}
U^\varphi(t,x)\leq C_1,\;V^\varphi(t,x)\leq C_2,\,\,\, \forall\; g(s)\leq x\leq h(s),\,\,\, t_x\leq t\leq s,
\end{align}
which is equivalent to $\left(\|U^\varphi\|_{C(\overline{D}_s)},\|V^\varphi\|_{C(\overline{D}_s)}\right)\leq (C_1,C_2)$. It suffices to show that $(U^\varphi,V^\varphi)< (C_1+\varepsilon,C_2+\alpha\varepsilon)$ in $D_s$, for any given $\varepsilon>0$. Suppose this is not true. Due to 
\begin{align*}
\left(U^\varphi\left(t_x, x\right),\;V^\varphi\left(t_x, x\right)\right)=\left(\widehat{u}_0(x), v\left(t_x, x\right)\right)\leq (\|u_0\|_{\infty},\|v_0\|_{\infty})<\left(C_1+\varepsilon,C_2+\alpha\varepsilon\right),
\end{align*}
there exist $x_0\in \left(g(s), h(s)\right)$, $t_0\in \left(t_{x_0}, s\right)$ such that $\left(U^\varphi(t, x), V^\varphi(t,x)\right)<\left(C_1+\varepsilon, C_2+\alpha\varepsilon\right)$ for $x\in \left(g\left(t_0\right), h\left(t_0\right)\right),t\in \left(t_{x_0}, t_0\right)$. We have
$\left(U^\varphi\left(t_0, x_0\right),V^\varphi\left(t_0, x_0\right)\right)\leq \left(C_1+\varepsilon, C_2+\alpha\varepsilon\right)$, and $U^\varphi\left(t_0, x_0\right)=C_1+\varepsilon$ or $V^\varphi\left(t_0, x_0\right)=C_2+\alpha\varepsilon$. Without loss of generality, we may assume that $U^\varphi\left(t_0, x_0\right)=C_1+\varepsilon$, then $\left(U^\varphi\right)_t\left(t_0, x_0\right)\geq 0$. We present the details below in two cases to get a contradiction since significant changes are needed.

\textbf{Case 1:} $ab<H'(0)G'(0)$.
It is easy to check that
\begin{align*}
\widetilde{H}\left(V^{\varphi}\left(t_0, x_0\right)\right)\leq H\left(C_2+\dfrac{K_2}{K_1}\varepsilon\right)< \dfrac{H(K_2)}{K_2}\left(C_2+\dfrac{K_2}{K_1}\varepsilon\right)=\dfrac{aK_1}{K_2}\left(C_2+\dfrac{K_2}{K_1}\varepsilon\right)=a\left(C_1+\varepsilon\right).
\end{align*}
 Due to the equation satisfied by $U^\varphi$, we obtain that
\begin{align*}
0&\leq d_1\int_{g\left(t_0\right)}^{h\left(t_0\right)} J_1\left(x_0-y\right) \varphi_1\left(t_0, y\right)dy-d_1 U^\varphi\left(t_0, x_0\right)-aU^\varphi\left(t_0, x_0\right)+\widetilde{H}\left(V^\varphi\left(t_0, x_0\right)\right)\\
&\leq  d_1\|\varphi_1\|_{C\left(\overline{D}_s\right)} dy-\left(a+d_1\right)\left(C_1+\varepsilon\right)+a\left(C_1+\varepsilon\right)
 \leq  d_1C_1-\left(a+d_1\right)\left(C_1+\varepsilon\right)+a\left(C_1+\varepsilon\right)=-d_1\varepsilon <0.
\end{align*} 

\textbf{Case 2:}  $ab\geq H'(0)G'(0)$.

 Note that
$\widetilde{H}\left(V^{\varphi}\left(t_0, x_0\right)\right)\leq H'(0)V^{\varphi}\left(t_0, x_0\right)\leq H'(0)\left(C_2+\dfrac{G'(0)}{b}\varepsilon\right)=\dfrac{H'(0)G'(0)}{b}\left(C_1+\varepsilon\right).$
 Hence, from the equation satisfied by $U^\varphi$, we can deduce
\begin{align*}
0&\leq d_1\int_{g(t_0)}^{h(t_0)} J_1(x_0-y) \varphi_1(t_0,y)dy-d_1 U^\varphi(t_0,x_0)-aU^\varphi(t_0,x_0)+\widetilde{H}(V^\varphi(t_0,x_0))\\
&\leq  d_1\|\varphi_1\|_{C(\overline{D}_s)} dy-(a+d_1)(C_1+\varepsilon)+\dfrac{H'(0)G'(0)}{b}(C_1+\varepsilon)\\
& \leq  d_1C_1-(a+d_1)(C_1+\varepsilon)+a(C_1+\varepsilon)\leq \dfrac{H'(0)G'(0)-ab}{b}(C_1+\varepsilon)-d_1\varepsilon <0.
\end{align*} 
The contradiction in both cases proves $\eqref{c2}$. 
Therefore, $W^\varphi=\Gamma_s\varphi\in \mathcal{X}_{C}^s$, so $\Gamma_s$ maps $\mathcal{X}_{C}^s$ into itself. We next show that by $\delta$ small enough, $\Gamma_s$ is a contradiction on $\mathcal{X}_{C}^s$.\\ Let $\varphi,\rho\in \mathcal{X}_{C}^s$, then $(U,V)=W=W^\varphi-W^\rho=(U^\varphi-U^\rho,V^\varphi-V^\rho)$ satisfies
\begin{align*}
\left\{\begin{array}{ll}
U_{t}=d_1\displaystyle\int_{g(t)}^{h(t)} J_1(x-y) \varphi_1(t,y)dy-d_1 U-aU+H(V^{\varphi})-H(V^{\rho}), & t_x<t\leq s, \\
V_{t}=d_2\displaystyle\int_{g(t)}^{h(t)} J_2(x-y) \varphi_2(t, y)dy-d_2V-bV+G(U^\varphi)-G(U^\rho), & t_x<t\leq s, \\
U(t_x, x)=0,V(t_x, x)=0, & g(s)<x<h(s).
\end{array}\right.
\end{align*}
It follows that, for $x\in \left(g(s), h(s)\right)$ and $t_x\leq t\leq s$,
\begin{align*}
U(t,x)\leq e^{-(a+d_1)(t_x-t)}\displaystyle\int_{t_x}^{t}e^{-(a+d_1)(t_x-l)}\left(d_1\displaystyle\int_{g(l)}^{h(l)}J_1(x-y)(\varphi_1-\rho_1)(l,y)dy+cV(l,x)\right)dl.
\end{align*}
Since $\left(g(t),h(t)\right)\subset \left(g(s),h(s)\right)$ when $t\leq s$, we deduce that, for any $x\in (g(t),h(t)),$
\begin{align*}
\left|U(t,x)\right|&\leq e^{(a+d_1)(t-t_x)}\int_{t_x}^{t}e^{(a+d_1)(l-t_x)}dl\left(d_1\|\varphi_1-\rho_1\|_{C(\overline{D}_s)}+c\|V\|_{C(\overline{D}_s)}\right)\\
&\leq (t-t_x)e^{2(a+d_1)(t-t_x)}\left(d_1\|\varphi_1-\rho_1\|_{C(\overline{D}_s)}+c\|V\|_{C(\overline{D}_s)}\right)\\& \leq se^{2(a+d_1)s}\left(d_1\|\varphi_1-\rho_1\|_{C(\overline{D}_s)}+c\|V\|_{C(\overline{D}_s)}\right).
\end{align*}
This implies that
$\|U\|_{C(\overline{D}_s)}\leq se^{2(a+d_1)s}\left(d_1\|\varphi_1-\rho_1\|_{C(\overline{D}_s)}+c\|V\|_{C(\overline{D}_s)}\right).$
Since $G,H\in C^1([0,\infty)$, for $C_1,C_2>0$, there exists two constants $\alpha(C_1),\alpha(C_2)>0$ such that
\begin{align*}
&|G(z_1)-G(z_2)|\leq \alpha(C_1)\left|z_1-z_2\right|,\, \text{ for } z_1,z_2\in \left[0,C_1\right],\\
&|H(z_1)-H(z_2)|\leq \alpha(C_2)\left|z_1-z_2\right|,\, \text{ for } z_1,z_2\in \left[0,C_2\right].
\end{align*}
It follows that
\begin{align*}
V(t,x)\leq e^{-(b+d_2)(t_x-t)}\displaystyle\int_{t_x}^{t}e^{-(b+d_2)(t_x-l)}\left(d_2\displaystyle\int_{g(l)}^{h(l)}J_2(x-y)(\varphi_2-\rho_2)(l,y)dy+G(U^\varphi(l,x))-G(U^\rho(l,x))\right)dl.
\end{align*}
Since $\left(g(t), h(t)\right)\subset \left(g(s), h(s)\right)$ when $t\leq s$, we deduce that, for $x\in (g(t),h(t)),$
\begin{align*}
|V(t,x)|&\leq e^{(b+d_2)(t-t_x)}\displaystyle\int_{t_x}^{t}e^{(b+d_2)(l-t_x)}\left(d_2\|\varphi_2-\rho_2\|_{C(\overline{D}_s)}+|G(U^\varphi(l,x)-G(U^\rho(l,x)|\right)dl\\
& \leq se^{2(b+d_2)s}\left(d_2\|\varphi_2-\rho_2\|_{C(\overline{D}_s)}+\alpha(C_1)\|V\|_{C(\overline{D}_s)}\right).
\end{align*}
This leads to 
$
\|V\|_{C(\overline{D}_s)}\leq se^{2(b+d_2)s}\left(d_2\|\varphi_2-\rho_2\|_{C(\overline{D}_s)}+\alpha(C_1)\|V\|_{C(\overline{D}_s)}\right).
$
Similarly, we also obtain
\begin{align*}
\|U\|_{C(\overline{D}_s)}\leq se^{2(a+d_1)s}\left(d_1\|\varphi_2-\rho_2\|_{C(\overline{D}_s)}+\alpha(C_2)\|V\|_{C(\overline{D}_s)}\right).
\end{align*}
Setting $\beta=\max\{a+d_1,b+d_2\},\gamma=\max\{\alpha(C_1),\alpha(C_2)\},d=\max\{d_1,d_2\}$. Choose $\delta>0$ satisfies
\begin{align*}
\delta se^{2\beta\delta}\gamma\leq \dfrac{1}{2}, \quad \delta\gamma e^{2\beta\delta}d\leq \dfrac{1}{4} .
\end{align*}
For such $\delta$ and $s\in (0,\delta)$ we have
\begin{align*}
\|\Gamma_s\varphi-\Gamma_s\rho\|_{C(\overline{D}_s)}&\leq \|U\|_{C(\overline{D}_s)}+\|V\|_{C(\overline{D}_s)}\leq \dfrac{1}{2}\left(\|\varphi_1-\rho_1\|_{C(\overline{D}_s)}+\|\varphi_2-\rho_2\|_{C(\overline{D}_s)}\right)=\dfrac{1}{2}\|\varphi-\rho\|_{C(\overline{D}_s)},
\end{align*} 
which implies that $\Gamma_s$ is a contraction on $\mathcal{X}_{C}^s$. We may now apply the contraction mapping theorem to conclude that $\Gamma_s$ has a unique fixed point $W=(U,V)$ in $\mathcal{X}_{C}^s$. Then $(u,v)=(U,V)$ solves $\eqref{local1}$ for $0\leq t\leq s$. If we can show that any solution $(u,v)$ of $\eqref{main}$ satisfies 
\begin{align}\label{c3}
0\leq u\leq C_1,0\leq v\leq C_2 \text{ in } D_s,
\end{align}
then $(u,v)$ must coincides with the unique fixed point $(U,V)$ of $\Gamma_s$ in $\mathcal{X}_{C}^s$. We next prove such an estimate for $(u,v)$. Note that $u,v\geq 0$ already follows from $\eqref{c}$.  We now show $(u,v)\leq (C_1,C_2)$ in $D_s$. It suffices to show that $(u,v)\leq (C_1+\varepsilon,C_2+\alpha\varepsilon)$ in $D_s$, for any given $\varepsilon>0$. Suppose this is not true. Due to
\begin{align*}
\left(u\left(t_x, x\right), v\left(t_x, x\right)\right)=\left(\widehat{u}_0(x), \widehat{v}_0(x)\right)\leq (\|u_0\|_{\infty},\|v_0\|_{\infty})<\left(C_1+\varepsilon, C_2+\alpha\varepsilon\right),
\end{align*}
there exist $x_0\in (g(s),h(s))$, $t_0\in (t_{x_0},s)$ such that $(u(t,x),v(t,x))<\left(C_1+\varepsilon,\; C_2+\alpha\varepsilon\right)$ for $x\in \left(g\left(t_0 \right), h\left(t_0\right)\right),$\\$t\in (t_{x_0},t_0)$, then $\left(u\left(t_0,x_0\right),v\left(t_0, x_0\right)\right)\leq \left(C_1+\varepsilon,\; C_2+\alpha\varepsilon\right)$ and $u\left(t_0, x_0\right)=C_1+\varepsilon$ or $v\left(t_0, x_0\right)=C_2+\dfrac{a}{c}\varepsilon$. For  $u\left(t_0, x_0\right)=C_1+\varepsilon$, then $u_t\left(t_0, x_0\right)\geq 0$.  Without loss of generality, we may assume that $u\left(t_0, x_0\right)=C_1+\varepsilon$, then $u_t\left(t_0, x_0\right)\geq 0$. A contradiction can be clearly achieved by dividing the next argument into two cases.

\textbf{Case 1:} $ab<H'(0)G'(0)$.

 From the equation satisfied by $u$, it is clear that
\begin{align*}
0&\leq d_1\int_{g(t_0)}^{h(t_0)} J_1\left(x_0-y\right) u\left(t_0,y\right)dy-d_1 u(t_0,x_0)-au\left(t_0,x_0\right)+H\left(v(t_0,x_0)\right)
\\&
\leq d_1(C_1+\varepsilon)\int_{g(t_0)}^{h(t_0)} J_1\left(x_0-y\right)dy-\left(a+d_1\right)\left(C_1+\varepsilon\right)+a\left(C_1+\varepsilon\right)\\&
 <d_1\left(C_1+\varepsilon\right) -\left(a+d_1\right)\left(C_1+\varepsilon\right)+ a\left(C_1+\varepsilon\right)=0.
\end{align*} 

\textbf{Case 2:} $ab\geq H'(0)G'(0)$.
In this case, we have
\begin{align*}
0&\leq d_1\int_{g(t_0)}^{h(t_0)} J_1(x_0-y) u(t_0,y)dy-d_1 u(t_0,x_0)-au(t_0,x_0)+H(v(t_0,x_0))\\
&\leq d_1(C_1+\varepsilon)\int_{g(t_0)}^{h(t_0)} J_1(x_0-y) dy-(a+d_1)(C_1+\varepsilon)+\dfrac{G'(0)H'(0)}{b}(C_1+\varepsilon)\\
& <d_1(C_1+\varepsilon) -(a+d_1)(C_1+\varepsilon)+ \dfrac{G'(0)H'(0)}{b}(C_1+\varepsilon)=\dfrac{G'(0)H'(0)-ab}{b}(C_1+\varepsilon)\leq 0.
\end{align*} 
From above contradictions, we deduce $0\leq u\leq C_1,0\leq v\leq C_2$ in $D_s$. Therefore $(u,v)$ coincides with the unique fixed point of $\Gamma_s$ in $\mathcal{X}_{C}^s$. We have proved the fact that for every $s\in (0,\delta]$, $\Gamma_s$ has a unique fixed point in $\mathcal{X}_{C}^s$. \\
\textbf{Step 3:} \textit{Completion of the proof}. From  $\textit{Step 2}$, we know that $\eqref{local1}$ has a unique solution $(u,v)$ defined for $t\in [0,\delta]$ and $(u,v)$ satisfies $\eqref{phi}$ with $s=\delta$. It follows from the equation $\eqref{c2}$, we get
\begin{align*}
\max\left\{\max_{[g(\delta),h(\delta)]}U^\varphi(\delta,x),\dfrac{K_1}{K_2}\max_{[g(\delta),h(\delta)]}V^\varphi(\delta,x),K_1\right\}\leq \max\{C_1,K_1\}=C_1.
\end{align*}
Hence, we may apply $\textit{Step 2}$ to $\eqref{local1}$ but with the initial time $t=0$ replace by $t=\delta$ to conclude that the unique solution can be extended to a slightly large domain $D^{\overline{\delta}}_{g,h}$. Moreover, by $\eqref{c3}$ and the definition of $\delta$ in $\textit{Step 2}$, we see that $\overline{\delta}$ depends only on $d_1,d_2,a,b,C_1$, and it can take any value in $(0,2\delta]$. Furthermore, from the above proof of $\eqref{c3}$ we easily see that the extended solution $(u,v)$ satisfies $\eqref{c3}$ in $D^{\overline{\delta}}_{g,h}$. Thus the extended can be repeated. By repeating this process finitely many times, the solution $(u,v)$ of $\eqref{local1}$ will be uniquely extended to $D^T_{g,h}$. As explained above, now $\eqref{c3}$ holds with $s=T$, and hence to prove that $(u,v)$ satisfies $\eqref{b}$.
\end{proof}
\begin{proof}[ \bf Proof of Theorem \ref{theo_1}]
Following the approach of \cite{CDLL}, we make use of Lemma \ref{lem_2.2} and a fixed point argument. For convenience, though this demonstration we denote
\begin{align*}
\mu_1=\mu\; \text{ and }\; \mu_2=\mu\rho.
\end{align*}
By Lemma \ref{lem_2.2}, for any $T>0$ and  $\left(g, h\right) \mathbb{G}^T_{h_0}\times\mathbb{H}^T_{h_0}$
, we can find a unique
$\left(u_{g, h}, v_{g, h}\right)\in \mathbb{X}^T_{U_0, g, h}$ that solves \eqref{local1} and satisfies \eqref{b}.
Using such a $\left(u_{g, h}, v_{g, h}\right)$, we define the mapping $\CF$ by $\CF(g, h) = \left(\widetilde{g}, \widetilde{h}\right)$
, where 
\begin{align*}
\begin{array}{lll}
\widetilde{h}= h_0+ \displaystyle\int_{0}^{t}\left[\mu_1 \displaystyle\int_{g(\tau)}^{h(\tau)} \displaystyle\int_{h(\tau)}^{+\infty} J_1(x-y) u_{g,h}(t, x) d y d x+\mu_2\displaystyle\int_{g(\tau)}^{h(\tau)} \displaystyle\int_{h(\tau)}^{+\infty} J_2(x-y) v_{g,h}(t, x) d y d x\right]d\tau,\\
\widetilde{g}= -h_0-\displaystyle\int_{0}^{t}\left[ \mu_1\displaystyle\int_{g(\tau)}^{h(\tau)} \displaystyle\int_{h(\tau)}^{+\infty} J_1(x-y) u_{g,h}(t, x) d y d x-\mu_2\displaystyle\int_{g(\tau)}^{h(\tau)} \displaystyle\int_{-\infty}^{g(\tau)} J_2(x-y) v_{g,h}(t, x) d y d x\right]d\tau,
\end{array}
\end{align*}
for $0 < t \leq T$. 
To prove this theorem, we first show that if $T$ is small enough, then $\CF$ maps a suitable closed
subset $\Sigma_T$ of $\mathbb{G}^T_{h_0} \times \mathbb{H}^T_{h_0}$ into itself, and is a contraction mapping. This clearly implies that $\CF$ has
a unique fixed point in $\Sigma_T$ , which gives a solution $\left(u_{g,h}, v_{g,h}, g, h\right)$ of \eqref{main} defined for $t \in \left(0, T\right]$. Then
we prove that any solution $\left(u_1, v_1, g, h\right)$ of \eqref{main} with $\left(g, h\right)\in \mathbb{G}^T_{h_0} \times \mathbb{H}^T_{h_0}$ must satisfy $\left(g, h\right) \in \Sigma_T$. Thus,
 $\left(g, h\right)$ must coincide with the unique fixed point of $\CF$ in $\Sigma_T$ , which then implies that the
solution $\left(u_1, v_1, g, h\right)$ of \eqref{main} is unique. Finally we extend this unique local solution to a global
one. This plan will be carried out in several steps.\\
\textbf{Step 1:} \textit{Properties of $(\widetilde{g},\widetilde{h})$ and a closed subset of $\mathbb{G}_T\times \mathbb{H}_T$}. 

 Let $(g,h)\in \mathbb{G}_T\times \mathbb{H}_T$. Then $\widetilde{g},\widetilde{h}\in C^1([0,T])$ and for $0<t<T$,
\begin{align}\label{h,g}
\begin{cases}
\widetilde{g}^{\prime}(t)=-\mu_1\displaystyle\int_{g(t)}^{h(t)} \displaystyle\int_{-\infty}^{g(t)} J_1(x-y) u_{g,h}(t, x) d y d x-\mu_2\displaystyle\int_{g(t)}^{h(t)} \displaystyle\int_{-\infty}^{g(t)} J_2(x-y) v_{g,h}(t, x) d y d x, \\
\widetilde{h}^{\prime}(t)=\mu_1\displaystyle\int_{g(t)}^{h(t)} \displaystyle\int_{h(t)}^{+\infty} J_1(x-y) u_{g,h}(t, x) d y d x+\mu_2\displaystyle\int_{g(t)}^{h(t)} \displaystyle\int_{h(t)}^{+\infty} J_2(x-y) v_{g,h}(t, x) d y d x.
\end{cases}
\end{align}
By assumption \eqref{conditionH_G}, we have  $(u_{g,h},v_{g,h})$ be a solution of the equation $\eqref{local1}$ and satisfies $\eqref{b}$ that
\begin{align*}
\left\{\begin{array}{ll}
(u_{g,h})_{t}\geq -d_1 u_{g,h}-au_{g,h}, & t>0,\; x \in\left(g(t), h(t)\right), \\
(v_{g,h})_{t}\geq -d_2v_{g,h}-bv_{g,h}, & t>0,\; x \in\left(g(t), h(t)\right), \\
u_{g,h}(t, x)=v_{g,h}(t, x)=0, & t>0,\; x=g(t) \text { or } x=h(t), \\
u_{g,h}(0, x)=u_{0}(x), v_{g,h}(0, x)=v_{0}(x), & x \in\left[-h_{0}, h_{0}\right],
\end{array}\right.
\end{align*}
which implies that
\begin{align}\label{theorem1-2}
\left\{\begin{array}{lll}
u_{g,h}(t,x)\geq e^{-(d_1+a)t}u_0(x)\geq e^{-(d_1+a)T}u_0(x),\,\, t\in (0,T],\,\,\,\, \left|x\right|\leq h_0,\\
v_{g,h}(t,x)\geq e^{-(d_2+b)t}v_0(x)\geq e^{-(d_2+b)T}v_0(x),\,\, t\in (0,T],\,\,\,\, \left|x\right|\leq h_0.
\end{array}\right.
\end{align}
Since $J_i$ is continuous and $J_i(0) > 0$, there exists $\varepsilon\in\left(0,h_0/4\right)$ and $\delta_*>0$ such that
\begin{align}\label{theorem1-3}
J_i(x-y)\geq \delta_*\;\text{ for }\; \left|x-y\right|\leq \varepsilon_0,\;\forall i=1,2.
\end{align}
Combining $\eqref{h,g}$ and $\eqref{b}$, we have
\begin{align*}
&\left[\widetilde{h}(t)-\widetilde{g}(t)\right]^{\prime}\\
= &\mu_1\displaystyle\int_{g(t)}^{h(t)} \left(\displaystyle\int_{-\infty}^{g(t)}+\displaystyle\int_{h(t)}^{+\infty}\right) J_1(x-y) u_{g,h}(t, x) d y d x
+\mu_2\displaystyle\int_{g(t)}^{h(t)}\left(\displaystyle\int_{-\infty}^{g(t)}+\displaystyle\int_{h(t)}^{+\infty}\right)J_2(x-y) v_{g,h}(t, x) d y d x\\
\leq& \mu_1\displaystyle\int_{g(t)}^{h(t)} \left(\displaystyle\int_{-\infty}^{g(t)}+\displaystyle\int_{h(t)}^{+\infty}\right) J_1(x-y) C_1 d y d x
+\mu_2\displaystyle\int_{g(t)}^{h(t)}\left(\displaystyle\int_{-\infty}^{g(t)}+\displaystyle\int_{h(t)}^{+\infty}\right)J_2(x-y)C_2 d y d x\\
\leq &\left(\mu_1C_1+\mu_2C_2\right)\left[h(t)-g(t)\right], \;\; \text{for}\; t\in (0,T].
\end{align*}
We can choose $0<T \ll 1$, depending on $\mu_i,C_i,h_0,\varepsilon$ such that $h(T)-g(T)\leq 2h_0+\varepsilon/4$ and
\begin{align*}
	\widetilde{h}(t)-\widetilde{g}(t) & \leq 2 h_{0}+T\left(\mu_{1} C_{1}+\mu_{2} C_{2}\right)\left(2 h_{0}+\varepsilon_{0} / 4\right) \leq 2 h_{0}+\varepsilon_{0} / 4, \quad \forall t \in[0, T], \\
	h(t) & \in\left[h_{0}, h_{0}+\varepsilon_{0} / 4\right], \quad g(t) \in\left[-h_{0}-\varepsilon_{0} / 4,-h_{0}\right], \quad \forall t \in[0, T].
\end{align*}
This together with $\eqref{theorem1-2}$, $\eqref{theorem1-3}$ allow us to derive 
\begin{align*}
&\mu_1\displaystyle\int_{g(t)}^{h(t)} \displaystyle\int_{h(t)}^{+\infty} J_1(x-y) u_{g,h}(t, x) d y d x+\mu_2\displaystyle\int_{g(t)}^{h(t)} \displaystyle\int_{h(t)}^{+\infty} J_2(x-y) v_{g,h}(t, x) d y d x\\
\geq &
 \mu_1 \displaystyle\int_{h(t)-\frac{\varepsilon_{0}}{2}}^{h(t)} \displaystyle\int_{h(t)}^{h(t)+\frac{\varepsilon_{0}}{2}} J_1(x-y) u_{g,h}(t, x) \mathrm{d} y \mathrm{d} x+\mu_2 \displaystyle\int_{h(t)-\frac{\varepsilon_{0}}{2}}^{h(t)} \displaystyle\int_{h(t)}^{h(t)+\frac{\varepsilon_{0}}{2}} J_2(x-y) v_{g,h}(t, x) \mathrm{d} y \mathrm{d} x\\
 \geq &  \mu_1e^{-(a+d_1)T}\displaystyle\int_{h_0-\frac{\varepsilon_{0}}{4}}^{h_0} \displaystyle\int_{h_0+\frac{\varepsilon_0}{4}}^{h_0+\frac{\varepsilon_{0}}{2}} \delta_* u_0(x) \mathrm{d} y \mathrm{d} x+\mu_2e^{-(b+d_2)T}\displaystyle\int_{h_0-\frac{\varepsilon_{0}}{4}}^{h_0} \displaystyle\int_{h_0+\frac{\varepsilon_0}{4}}^{h_0+\frac{\varepsilon_{0}}{2}}\delta_* v_0(x) \mathrm{d} y \mathrm{d} x\\ 
 =:&\mu_1\alpha_1+\mu_2\alpha_2,
\end{align*} 
where $\alpha_1,\alpha_2$ are positive constants depending only $J_i,U_0$, with $i=1,2$. Hence, for sufficiently small $T_0$, with
$
T_0 =T(\mu_1,\mu_2,a,b,d_1,d_2,h_0,\varepsilon_0)>0,$ then
\begin{align}\label{2.16}
\widetilde{h}'(t)\geq \alpha_1\mu_1+\alpha_2\mu_2:=h_{*}>0,\quad t\in [0,T_0].
\end{align}
Similarly,
\begin{align}\label{2.17}
\widetilde{g}'(t)\leq -(\widetilde{\alpha}_1\mu_1+\widetilde{\alpha}_2\mu_2):=g_{*}<0,\quad t\in [0,T_0],
\end{align}
for some positive constants $\widetilde{\alpha}_1$ and $\widetilde{\alpha}_2$ depending only on $J_i,U_0$, with $i=1,2$. 
For $0\leq T\leq T_0$, we define
\begin{align*}
	\Sigma_{T}:=\left\{(g, h) \in \mathbb{G}_{T} \times \mathbb{H}_{T}: \frac{g\left(t_{2}\right)-g\left(t_{1}\right)}{t_{2}-t_{1}} \leq g_{*}, \frac{h\left(t_{2}\right)-h\left(t_{1}\right)}{t_{2}-t_{1}} \geq h_{*} \text { for } 0 \leq t_{1}<t_{2} \leq T\right.\\
	\text { and }\left.h(t)-g(t) \leq 2 h_{0}+\frac{\varepsilon_{0}}{4} \text { for } 0 \leq t \leq T\right\}.
\end{align*}
It follows from the above analysis that $\mathcal{F}\left(\Sigma_{T}\right) \subset \Sigma_{T}$.\\
\textbf{Step 2:} \textit{$\mathcal{F}$ is a contraction mapping on $\Sigma_T$ for $0<T\ll 1$}. For $(g_i,h_i)\in \Sigma_T,i=1,2$, we define
\begin{align*}
\Omega_T=D_{g_1,h_1}^T\bigcup D_{g_2,h_2}^T, u_i=u_{g_i,h_i}, v_i=v_{g_i,h_i}, \mathcal{F}\left(g_i,h_i\right)=\left(\widetilde{g}_i,\widetilde{h}_i\right),\\
\widehat{u}=u_1-u_2,\widehat{v}=v_1-v_2,\widehat{g}=g_1-g_2,\widehat{h}=h_1-h_2,\widetilde{g}=\widetilde{g}_1-\widetilde{g}_2,\widetilde{h}=\widetilde{h}_1-\widetilde{h}_2.
\end{align*}
Making the zero extension of $u_i,v_i$ in $([0,T]\times \mathbb{R})\setminus D^T_{g_i,h_i}$. It then follows that
\begin{align*}
\begin{aligned}
\left|\widetilde{h}^{\prime}(t)\right| \leq &  \mu_{1}\left|\displaystyle\int_{g_{1}(\tau)}^{h_{1}(\tau)} \displaystyle\int_{h_{1}(\tau)}^{\infty} J_{1}(x-y) u_1(\tau, x) \mathrm{d} y \mathrm{d} x-\displaystyle\int_{g_{2}(\tau)}^{h_{2}(\tau)} \displaystyle\int_{h_{2}(\tau)}^{\infty} J_{1}(x-y) u_2(\tau, x) \mathrm{d} y \mathrm{d} x\right| \\
&+\mu_{2}\left|\displaystyle\int_{g_{1}(\tau)}^{h_{1}(\tau)} \displaystyle\int_{h_{1}(\tau)}^{\infty} J_{2}(x-y) v_1(\tau, x) \mathrm{d} y \mathrm{d} x-\int_{g_{2}(\tau)}^{h_{2}(\tau)} \displaystyle\int_{h_{2}(\tau)}^{\infty} J_{2}(x-y) v_2(\tau, x) \mathrm{d} y \mathrm{d} x\right| \\
\leq & \mu_{1} \displaystyle\int_{g_{1}(\tau)}^{h_{1}(\tau)} \displaystyle\int_{h_{1}(\tau)}^{\infty} J_{1}(x-y)\left|\hat{u}(\tau, x)\right| \mathrm{d} y \mathrm{d} x+ \mu_{2} \displaystyle\int_{g_{1}(\tau)}^{h_{1}(\tau)} \displaystyle\int_{h_{1}(\tau)}^{\infty} J_{2}(x-y)\left|\hat{v}(\tau, x)\right| \mathrm{d} y \mathrm{d} x\\
&+\mu_{1}\left|\left(\displaystyle\int_{g_{1}(\tau)}^{g_{2}(\tau)}\displaystyle\int_{h_{1}(\tau)}^{\infty}+\displaystyle\int_{h_{2}(\tau)}^{h_{1}(\tau)} \displaystyle\int_{h_{1}(\tau)}^{\infty}+\displaystyle\int_{g_{2}(\tau)}^{h_{2}(\tau)} \displaystyle\int_{h_{1}(\tau)}^{h_{2}(\tau)}\right) J_{1}(x-y) u_2(\tau, x) \mathrm{d} y \mathrm{d} x\right| \\
&+\mu_{2}\left|\left(\displaystyle\int_{g_{1}(\tau)}^{g_{2}(\tau)} \displaystyle\int_{h_{1}(\tau)}^{\infty}+\displaystyle\int_{h_{2}(\tau)}^{h_{1}(\tau)}\displaystyle\int_{h_{1}(\tau)}^{\infty}+\displaystyle\int_{g_{2}(\tau)}^{h_{2}(\tau)} \displaystyle\int_{h_{1}(\tau)}^{h_{2}(\tau)}\right) J_{2}(x-y) v_2(\tau, x) \mathrm{d} y \mathrm{d} x\right| \\
\leq &  \mu_{1}\left(3 h_{0}\left\|\hat{u}\right\|_{C\left(\Omega_{T}\right)}+\|\hat{g}\|_{C([0, T])} C_1+\left(C_1+3 h_{0} C_1\left\|J_{1}\right\|_{\infty}\right)\|\hat{h}\|_{C([0, T])}\right)\\
&+  \mu_{2}\left(3 h_{0}\left\|\hat{v}\right\|_{C\left(\Omega_{T}\right)}+\|\hat{g}\|_{C([0, T])} C_2+\left(C_2+3 h_{0} C_2\left\|J_{2}\right\|_{\infty}\right)\|\hat{h}\|_{C([0, T])}\right).
\end{aligned}
\end{align*}
This means that
$\left|\widetilde{h}(t)\right| \leq A_{0} T\left(\left\|\hat{u}, \hat{v}\right\|_{C\left(\overline{\Omega}_{T}\right)}+\|\hat{g}, \hat{h}\|_{C([0, T])}\right),\text{for}\;0<t \leq T,$
where $A_0$ depends only on $h_0, \mu_i, C_i, J_i$.
Similarly, we have
$\left|\widetilde{g}(t)\right| \leq A_{0} T\left(\left\|\hat{u}, \hat{v}\right\|_{C\left(\overline{\Omega}_{T}\right)}+\|\hat{g}, \hat{h}\|_{C([0, T])}\right),\text{for}\; 0<t \leq T.$
Therefore,
\begin{align}\label{step2}
\|\widetilde{g}, \widetilde{h}\|_{C([0, T])} \leq A_0T\left(\left\|\hat{u}, \hat{v}\right\|_{C\left(\overline{\Omega}_{T}\right)}+\|\hat{g}, \hat{h}\|_{C([0, T])}\right).
\end{align}
Now, we estimate $\left\|\hat{u}, \hat{v}\right\|_{C\left(\overline{\Omega}_{T}\right)}$. Fix $(s,x)\in \Omega_T$. We first estimate $|\widehat{u}(t,x)|$ and $|\widehat{v}(t,x)|$ in all the possible cases.\\
\textbf{\underline{Case 1:}} $x \in\left(g_{1}(s), h_{1}(s)\right) \backslash\left(g_{2}(s), h_{2}(s)\right)$. In this case either $g_{1}(s)<x \leq g_{2}(s)$ or $h_2(s)\leq x\leq h_1(x)$, and $u_2(s,x)=v_2(s,x)=0$.
If $h_2(s)\leq x<h_1(s)$, there exist $0<s_1<s$ such that $x=h(s_1)$, and $h_1(s)>h_1(s_1)=x\geq h_2(s)$. It follows that $g(t)\leq h_1(s_1)\leq h_1(t)$ for all $t\in \left[s_1,s\right]$. Since $H(z)/z$ is non-increasing, we have
$H(v)\leq H'(0)v\leq H'(0)C_2.$
Therefore, by integrating the equation satisfied by $u_1$ from $s_1$ to $s$ we obtain
\begin{align*}
|\widehat{u}(s,x)|&=|u_1(s,x)|=\displaystyle\int_{s_1}^{s}\left(d_1\displaystyle\int_{g_1(t)}^{h_1(t)}J_1(s-y)u_1(t,y)dy-d_1u_1-au_1+H(v_1)\right)dt\\
&\leq (s-s_1)(d_1C_1+H'(0)C_2)\leq h_{*}^{-1}[h_1(s)-h_1(s_1)](d_1C_1+H'(0)C_2)\\&\leq h_{*}^{-1}\left[h_1(s)-h_2(s)\right](d_1C_1+H'(0)C_2)\leq A_1\|h_1-h_2\|_{C\left([0,s]\right)}=A_1\|\widehat{h}\|_{C\left(\left[0,T\right]\right)}.
\end{align*}
When $g_1(s)<x\leq g_2(s)$, we can analogously obtain
$|\widehat{u}(s,x)|=|u_1(s,x)|\leq A_1\|\widehat{g}\|_{C\left(\left[0,T\right]\right)}.$
Therefore,
$|\widehat{u}(s,x)|\leq A_1\|\widehat{g},\widehat{h}\|_{C([0,T])}.$
Note that $G(u)\leq G'(0)u\leq G'(0)C_1$, we derive  $\|\widehat{v}(s,x)\|=v_1(s,x)\leq A_1\|\widehat{g},\widehat{h}\|_{C([0,T])}$. Thus, in such case,
\begin{align}\label{A1}
\|\widehat{v}(s,x)\|,\|\widehat{v}(s,x)\|\leq A_1\|\widehat{g},\widehat{h}\|_{C([0,T])}.
\end{align}
\textbf{\underline{Case 2:}} $x \in\left(g_{2}(s), h_{2}(s)\right) \backslash\left(g_{1}(s), h_{1}(s)\right)$. Similar to $\textit{Case 1}$ we have $|\widehat{u}(s,x)|=u_2(s,x)\leq A_1\|\widehat{g},\widehat{h}\|_{C([0,T])}$ and $|\widehat{v}(s,x)|=v_2(s,x)\leq A_1\|\widehat{g},\widehat{h}\|_{C([0,T])}$. Therefore $\eqref{A1}$ holds in this case.\\
\textbf{\underline{Case 3:}} $x \notin\left(g_{2}(s), h_{2}(s)\right) \bigcup\left(g_{1}(s), h_{1}(s)\right)$. In this cases, $\widehat{u}(s,x)=\widehat{v}(s,x)=0$ and so $\eqref{A1}$ holds trivailly.\\
\textbf{\underline{Case 4:}} $x \in\left(g_{1}(s), h_{1}(s)\right) \bigcap\left(g_{2}(s), h_{2}(s)\right)$. If $x \in\left(g_{1}(t), h_{1}(t)\right) \bigcap\left(g_{2}(t), h_{2}(t)\right)$ for all $0<t<s$, that is, $x \in\left[-h_{0}, h_{0}\right]$, we obtain
\begin{align}\label{case4}
\hat{u}_t(t, x)= &d_{1} \displaystyle\int_{g_{1}(t)}^{h_{1}(t)} J_{1}(x-y) \hat{u}(t, y) \mathrm{d} y+d_{1}\left\{\displaystyle\int_{g_{1}(t)}^{g_{2}(t)}+\displaystyle\int_{h_{2}(t)}^{h_{1}(t)}\right\} J_{1}(x-y) u_2(t, y) \mathrm{d} y-d_1\widehat{u}(t,x)+H(u_1)-H(u_2).
\end{align}
Note that $\widehat{u}(0,x)=0,0\leq u_1,u_2\leq C_1$ and 
\begin{align*}
\left|G\left(u_1\right)-G\left(u_2\right)\right|\leq \alpha(C_1)|u_1-u_2|=\alpha(C_1,C_2)|\widehat{u}|\leq \alpha(C_1)\|\widehat{u},\widehat{v}\|_{C(\Omega_T)},
\end{align*}
where $\alpha(C_1,C_2)=\max\{\alpha(C_1),\alpha(C_2)\}$. Intergating $\eqref{case4}$ from $0$ to $s$ we derive
\begin{align*}
\left|\hat{u}(s, x)\right| \leq\left(2 d_{1}\left\|\hat{u}\right\|_{C\left(\Omega_{T}\right)}+d_{1} C_{1}\left\|J_{1}\right\|_{\infty}\|\hat{g}, \hat{h}\|_{C([0, T])}+\alpha(C_1,C_2)\left\|\hat{u}, \hat{v}\right\|_{C\left(\Omega_{T}\right)}\right) T.
\end{align*}
Similarly,
$\left|\hat{v}(s, x)\right| \leq\left(2 d_{2}\left\|\hat{v}\right\|_{C\left(\Omega_{T}\right)}+d_{2}C_{2}\left\|J_{2}\right\|_{\infty}\|\hat{g}, \hat{h}\|_{C([0, T])}+\alpha(C_1,C_2)\left\|\hat{u}, \hat{v}\right\|_{C\left(\Omega_{T}\right)}\right) T.$
If there exists $0<t<s$ such that $x \notin\left(g_{1}(t), h_{1}(t)\right) \bigcap\left(g_{2}(t), h_{2}(t)\right)$, then we can choose the largest $t_0\in (0,t)$ such that
\begin{align}\label{case4a}
x \in\left(g_{1}(t), h_{1}(t)\right) \bigcap\left(g_{2}(t), h_{2}(t)\right), \quad \forall t_{0}<t \leq s,
\end{align}
and
$x \in\left(g_{1}\left(t_{0}\right), h_{1}\left(t_{0}\right)\right) \backslash\left(g_{2}\left(t_{0}\right), h_{2}\left(t_{0}\right)\right), \text { or }  x \in\left(g_{2}\left(t_{0}\right), h_{2}\left(t_{0}\right)\right) \backslash\left(g_{1}\left(t_{0}\right), h_{1}\left(t_{0}\right)\right).$
According to $\textit{Case 1}$ and $\textit{Case 2}$, we have $\left|\hat{u}\left(t_{0}, x\right)\right| \leq C_{4}\|\hat{g}, \hat{h}\|_{C([0, T])}$. In view of $\eqref{case4a}$, we can obtain that $\eqref{case4}$ holds for all $t_0<t\leq s$. Integrating $\eqref{case4}$ from $0$ to $s$, we have
\begin{align*}
\begin{aligned}
\left|\hat{u}(s, x)\right| & \leq\left|\hat{u}\left(t_{0}, x\right)\right|+\left(2 d_{1}\left\|\hat{u}\right\|_{C\left(\Omega_{T}\right)}+d_{1} C_{1}\left\|J_{1}\right\|_{\infty}\|\hat{g}, \hat{h}\|_{C([0, T])}+\alpha(C_1,C_2)\left\|\hat{u}, \hat{v}\right\|_{C\left(\Omega_{T}\right)}\right)\left(s-t_{0}\right) \\
& \leq A_1\|\hat{g}, \hat{h}\|_{C([0, T])}+A_{2}\left(\|\hat{g}, \hat{h}\|_{C([0, T])}+\left\|\hat{u}, \hat{v}\right\|_{C\left(\Omega_{T}\right)}\right) T \\
&=: A_3\|\hat{g}, \hat{h}\|_{C([0, T])}+A_2 T\left\|\hat{u}, \hat{v}\right\|_{C\left(\Omega_{T}\right)}.
\end{aligned}
\end{align*}
Similarly,
$\left|\hat{v}(s, x)\right| \leq A_3\|\hat{g}, \hat{h}\|_{C([0, T])}+A_2 T\left\|\hat{u}, \hat{v}\right\|_{C\left(\Omega_{T}\right)}.$
Summarizing the above discussions, we have
\begin{align*}
\left\|\hat{u}, \hat{v}\right\|_{C\left(\Omega_{T}\right)} \leq A^{\prime}\|\hat{g}, \hat{h}\|_{C([0, T])}+A^{\prime} T\left\|\hat{u}, \hat{v}\right\|_{C\left(\Omega_{T}\right)} \leq 2 A^{\prime}\|\hat{g}, \hat{h}\|_{C([0, T])}
\end{align*}
if $A'T<1/2$. This combined with $\eqref{step2}$ yields
$\|\widetilde{g}, \widetilde{h}\|_{C([0, T])} \leq A_0\left(2 A^{\prime}+1\right) T\|\hat{g}, \hat{h}\|_{C([0, T])} \leq \frac{1}{2}\|\hat{g}, \hat{h}\|_{C([0, T])}$
if $A_0(2A'+1)T<1/2$. This leads to $\mathcal{F}$ is a contraction mapping on $\Sigma_T$.\\
\textbf{Step 3:} \textit{Local existence and uniqueness}. By $\textit{Step 2}$ and the contraction mapping theorem we know that $\eqref{main}$ has a solution $(u,v,g,h)$ defined for $t\in (0,T]$. If we can show that $(g,h)\in \Sigma_T$ holds for any solution $(u,v,g,h)$ of $\eqref{main}$ defined for $t\in (0,T]$, then $(g,h)$ must coincide with the unique fixed point of $\mathcal{F}$ in $\Sigma_T$. This shows that the local solution $(u,v,g,h)$ to $\eqref{main}$ is unique.
Let $(u,v,g,h)$ be an artitrary solution of $\eqref{main}$ defined for $t\in (0,T]$. Then
\begin{align*}
&g^{\prime}(t)=-\mu_1\displaystyle\int_{g(t)}^{h(t)} \displaystyle\int_{-\infty}^{g(t)} J_1(x-y) u(t, x) d y d x-\mu_2\displaystyle\int_{g(t)}^{h(t)} \displaystyle\int_{-\infty}^{g(t)} J_2(x-y) v(t, x) d y d x,\\
&h^{\prime}(t)=\mu_1\displaystyle\int_{g(t)}^{h(t)} \displaystyle\int_{h(t)}^{+\infty} J_1(x-y) u(t, x) d y d x+\mu_2\displaystyle\int_{g(t)}^{h(t)} \displaystyle\int_{h(t)}^{+\infty} J_2(x-y) v(t, x) d y d x.
\end{align*}
By Lemma \ref{lem_2.2}, we deduce that $0\leq u\leq C_1,0\leq v\leq C_2$ in $D_{g,h}^T$. Thus
\begin{align*}
&h^{\prime}(t)-g^{\prime}(t)\\
&=\mu_{1} \displaystyle\int_{g(t)}^{h(t)}\left(\displaystyle\int_{h(t)}^{\infty}+\displaystyle\int_{-\infty}^{g(t)}\right) J_{1}(x-y) u(t, x) \mathrm{d} y \mathrm{d} x +\mu_{2} \displaystyle\int_{g(t)}^{h(t)}\left(\displaystyle\int_{h(t)}^{\infty}+\displaystyle\int_{-\infty}^{g(t)}\right) J_{2}(x-y) v(t, x) \mathrm{d} y \mathrm{d} x\\
&\leq \left(\mu_1C_1+\mu_2C_2\right)\left[h(t)-g(t)\right],
\end{align*}
which leads to
$h(t)-g(t)\leq 2h_0e^{\left(\mu_1C_1+\mu_2C_2\right)t}.$
Shrink $T$ so that $2h_0e^{\left(\mu_1C_1+\mu_2C_2\right)t}\leq 2h_0+\varepsilon_0/4$, then $h(t)-g(t)\leq 2h_0+\varepsilon_0/4$ on $[0,T]$. In addition, the proof of $\eqref{2.16}$ and $\eqref{2.17}$ gives $h'(t)\geq h_*$ and $g'(t)\leq g_*$ on $(0,T]$. So $(g,h)\in \Sigma_T$ as we required.

\textbf{Step 4:} \textit{Global existence and uniqueness.}
This is almost identical to the proof of the Step 4  in \cite{CDLL}; we omit the details.
\end{proof}

\section{ Characterizations of the principal eigenvalue}\label{sec.3}

In this section, we first give and prove some results that will be used hereafter. Throughout this section, we always assume that $J_i,\; i=1,2$ satisfy ${\bf (J)}$, and $H, G$ satisfy \eqref{conditionH_G}. 

 We start with two comparison results.
\subsection{Comparison principle}
\begin{lemma}\label{lem_3.1}
Let $h_0, \,T > 0$ and $\O_0 := \left[0,T\right] \times \left[-h_0, h_0\right]$. Suppose that $A, B, C, D\in L^{\infty}\left(\O_0\right)$,   $\left(u(t,x), v(t,x)\right)$ as well as $\left(u_t(t,x), v_t(t,x)\right)$ are continuous in $\O_0$ and satisfy
\begin{equation}\label{maxi11}
\left\lbrace\begin{array}{ll}
u_t(t, x)\geq d_1\displaystyle\int_{-h_0}^{h_0}J_1(x-y)u(t, x)dy -d_1u + Au + Bv, &0<t\leq T,\; -h_0\leq x\leq h_0,\\
v_t(t, x)\geq d_2\displaystyle\int_{-h_0}^{h_0}J_2(x-y)v(t, x)dy -d_2v  + Dv + Cu, &0<t\leq T,\; -h_0\leq x\leq h_0,\\
u(0,x)\geq0,\,\,\ v(0,x)\geq0,&\left|x\right|\leq h_0.
\end{array}\right.
\end{equation}
 Then $\left( u(t,x), v(t, x)\right)\geq \left(0, 0\right)$ for all $0<t\leq T$, and $\; -h_0\leq x\leq h_0$.
\end{lemma}
\begin{proof}[\bf Proof]
This follows from a simple variation of the argument in the proof of of Lemma 3.1 in \cite{DN}. We omit
the details. 
\end{proof}
\begin{lemma}[Comparison principle]\label{comparison_1}
Assume that ${\bf (J)}$, \eqref{conditionH_G} hold, and $u_0, v_0$ satisfy \eqref{initial}. For $T\in (0, +\infty)$, suppose that $\overline{h}, \overline{g}\in C\left([0,T]\right)$ and $\overline{u}, \overline{v}\in C\left(\overline{D}^{T}_{\overline{g},\overline{h}}\right)$ satisfy
\begin{align}
\left\{\begin{array}{lll}
 \overline{u}_t \geq  d_1\left[\displaystyle\int_{\overline{g}(t)}^{\overline{h}(t)}J_1(x-y)\overline{u}(t,y)dy - \overline{u}(t,x)\right] - a\overline{u}(x,t) + H\left(\overline{v}(t, x)\right),& t>0, x\in\left( \overline{g}(t), \overline{h}\left(t\right)\right), \\
\overline{v}_t \geq d_2\left[\displaystyle\int_{\overline{g}(t)}^{\overline{h}(t)}J_2(x-y)\overline{v}(t,y)dy - \overline{v}(t,x)\right] -b\overline{v}(t,x)+ G\left(\overline{u}(t,x)\right),& t>0,\, x\in\left( \overline{g}(t), \overline{h}\left(t\right)\right), \\
\overline{u}\left(t, x\right)\geq 0,\,\,\,\, \overline{v}\left(t, x\right)\geq 0,&0<t<T,   \\  
\overline{h}^{\prime}(t) \geq \mu\left( \displaystyle\int_{\overline{g}(t)}^{\overline{h}(t)}\displaystyle\int_{\overline{h}(t)}^{\infty}J_1(x-y)\overline{u}(t, x)dydx  +\rho\displaystyle\int_{\overline{g}(t)}^{\overline{h}(t)}\displaystyle\int_{\overline{h}(t)}^{\infty}J_2(x-y)\overline{v}(t, x)dydx\right),& 0<t<T,\\
\overline{g}^{\prime}(t) \leq -\mu\left( \displaystyle\int_{\overline{g}(t)}^{\overline{h}(t)}\displaystyle\int_{-\infty}^{\overline{g}(t)}J_2(x-y)\overline{v}(t, x)dydx- \rho \displaystyle\int_{\overline{g}(t)}^{\overline{h}(t)}\displaystyle\int_{-\infty}^{\overline{g}(t)}J_1(x-y)\overline{u}(t, x)dydx\right),& 0<t<T,\\
-\overline{g}(0) \leq -h(0),\,\,\, \overline{h}(0)\geq h_0 ,\,\, \overline{u}(0, x) = u_0(x),\,\,\overline{v}(0, x) \geq v_0(x),& x\in \left[-h_0, h_0\right],
\end{array}\right.
\end{align}
then the unique solution $(u,v,g,h)$ of \eqref{main} satisfies
\begin{align}
u(t, x) \leq \overline{u}(t, x),\,\,\,\, v(t, x) \leq \overline{v}(t, x),\,\,\, \overline{g}(t) \geq g(t),\,\,\,\ h(t) \leq \overline{h}(t) \,\,\, \text{for}\,\, 0 < t \leq T ,\,\, g(t) \leq x \leq h(t).
\end{align}
\end{lemma}
\begin{proof}[\bf Proof]
Thanks to assumption \eqref{conditionH_G}, we can repeat the arguments
in \cite[Theorem 3.1]{CDLL} to show the desired conclusion. 
\end{proof}
\begin{remark}
The function $\left(\overline{u}, \overline{v}\right)$ or the pair $\left(\overline{u}, \overline{v}, \overline{g}, \overline{h}\right)$  in Lemma \ref{comparison_1} is usually called an supersolution of problem \eqref{main}. We can define a  subsolution by reversing all the inequalities in suitable places. There is a symmetry version of Lemma \ref{comparison_1}, where the conditions on the left and right boundaries are interchanged. We also have corresponding comparison results for subsolution in each case.
\end{remark}

\subsection{Eigenvalue problems and variational formulations}
\hspace{13cm}\newline
From now on, the integer $d_1,d_2>0$ and $L_1<0<L_2$ are fixed.
 We consider the eigenvalue problem
\begin{align}\label{eigenproblem}
	\left\{\begin{array}{ll}
	d_1\displaystyle\int_{L_1}^{L_2} J_1(x-y) \varphi_1(y)dy-d_1\varphi_1+ a_{11}\varphi_1+ a_{12}\varphi_2+\lambda\varphi_1=0, &  x \in [L_1, L_2],  \\
	 d_2\displaystyle\int_{L_1}^{L_2} J_2(x-y) \varphi_2(y)dy-d_2\varphi_2+ a_{21}\varphi_1+a_{22}\varphi_2+\lambda\varphi_2=0, &  x \in \left[L_1, L_2\right],
	\end{array}\right.
	\end{align}
where $a_{12}=a_{21}>0$. For convenience, we define the space $\mathbf{E},\, \mathbf{E}^+,\, \mathbf{E}^{++}$ as follows:
\begin{align*}
\begin{array}{lll}
\mathbf{A} = \begin{pmatrix}
 a_{11}&a_{12}\\
a_{21} & a_{22}
\end{pmatrix},\,\mathbf{D} = \begin{pmatrix}
d_1&0\\
0 & d_2
\end{pmatrix},\,\\
\mathbf{E} = L^2\left(\left[L_1, L_2\right]\right)\times L^2\left(\left[L_1, L_2\right]\right),\mathbf{C}=C\left(\left[L_1, L_2\right]\right)\times C\left(\left[L_1, L_2\right]\right),\\
\mathbf{E}^+=\left\{\pmb{\varphi}=(\varphi_1,\varphi_2)^T\in \mathbf{E}\; \text{ such that } \varphi_1\geq 0 \text{ and } \varphi_2\geq 0\right\},\\
\mathbf{E}^{++}=\left\{\pmb{\varphi}=(\varphi_1,\varphi_2)^T\in \mathbf{E}\; \text{ such that } \varphi_1> 0 \text{ and } \varphi_2> 0\right\}.
\end{array}
\end{align*}
Moreover, we denote $\left(\varphi_1, \varphi_2\right)\geq \left(\phi_1, \phi_2\right)$ that $\left\{\begin{array}{lll} \varphi_1\geq \phi_1\\ \varphi_2\geq \phi_2\end{array}\right.$. 	Note that $\mathbf{E}$ is Hilbert space with inner product
	\begin{align*}
	\left \langle \pmb{\varphi} ,\pmb{\phi} \right \rangle=\displaystyle\int_{L_1}^{L_2}\varphi_1(x)\phi_1(x)dx+\displaystyle\int_{L_1}^{L_2}\varphi_2(x)\phi_2(x)dx,
	\end{align*}
	where $\pmb{\varphi}=\left(\varphi_1,\varphi_2\right)^T,$  $\pmb{\phi}=\left(\phi_1,\phi_2\right)^T\in \mathbf{E}$. We define $\mathbf{N}:\mathbf{E}\rightarrow \mathbf{C}$
	\begin{align*}
	\left(\mathbf{N}\varphi\right)(x)=\left(\CN_1[\varphi_1](x),\; \CN_2[\varphi_2](x)\right)^T,
	\end{align*}
	where $\CN_i[\varphi_i](x): = 
	\displaystyle\int_{L_1}^{L_2} J_i(x-y)\varphi_i(y)dy, \;\text{for}\; i=1,2$. 
Next, we consider the following auxiliary linear system
\begin{align}\label{PEV}
\mathbf{M}\pmb{\varphi}+\lambda \pmb{\varphi}=0,
\end{align}
where $\mathbf{M}:\mathbf{E}\rightarrow \mathbf{E}$ defined by $\mathbf{M}=\mathbf{D}\mathbf{N}-\mathbf{D}+\mathbf{A}$. Recall that the real $\lambda_p$ is  principal eigenvalue  of problem $\eqref{PEV}$ and associated (principal) eigenfunction $\pmb{\varphi}^p\in \mathbf{E}^{++}$ if satisfies
\begin{enumerate}
\item $\pmb{\varphi}^p$ is simple eigenfunction, i.e., if $\pmb{\varphi}^p$ is eigenfunction assosiated $\lambda_p$, then there exists a real number $k\neq 0$ such that $\pmb{\psi}=k\pmb{\varphi}^p$ in $\Omega$.

\item  If $\pmb{\psi}$ is a positive eigenfunction with eigenvalue $\lambda$ of $\eqref{PEV}$, then $\lambda=\lambda_p$.

\item  For any eigenvalue $\lambda$, $\text{Re} (\lambda)\geq\lambda_p$.
\end{enumerate}
 If we fix $d_1, d_2$ and investigate the property of $\lambda_p$, then we write $\lambda_p=\lambda_p\left(L_1,L_2\right)$ for brevity. Similarly, we write $\lambda_p = \lambda_p\left(d_1, d_2\right)$ to highlight the dependence on $d_1, d_2$. Moreover, we use $\lambda_p\left(\mathbf{A}\right)$ to describe the dependence of $\lambda_p$ on the coefficients of the matrix $\mathbf{A}$.
We first prove several properties of $\lambda_p$, which will be crucial in our analysis of the long time behaviour of \eqref{main}. 

\begin{theorem}\label{theorem_3.4}
Assume that $\left({\bf{J}}\right)$ holds and $a_{12}=a_{21}$.  Then $\mathbf{M}$ is self-adjiont and has a principal eigenvalue $\lambda_p$ given by
\begin{align}\label{3.02}
\lambda_p= -\sup\limits_{\left\|\pmb{\varphi}\right\|_{E} =1}\left \langle \mathbf{M}\pmb{\varphi}, \pmb{\varphi} \right \rangle.
\end{align}
\end{theorem}
\begin{proof}[\bf Proof]
Firstly, for $\pmb{\varphi}=\left(\varphi_1,\varphi_2\right)^T,\pmb{\psi}=\left(\psi_1,\psi_2\right)^T\in \mathbf{E}$, we have
\begin{align*}
\left \langle \mathbf{M}\pmb{\varphi}, \pmb{\psi}\right \rangle=&d_1\displaystyle\int_{L_1}^{L_2}\displaystyle\int_{L_1}^{L_2}J_1(x-y)\varphi_1(x)\psi_1(y)dxdy +d_2\displaystyle\int_{L_1}^{L_2}\displaystyle\int_{L_1}^{L_2}J_2(x-y)\varphi_2(x)\psi_2(y)dxdy\\
&
+\displaystyle\int_{L_1}^{L_2}\left((a_{11}-d_1)\varphi_1(x)\psi_1(x)+a_{12}\varphi_1(x)\psi_2(x)+a_{21}\varphi_2(x)\psi_1(x)+(a_{22}-d_2)\varphi_2(x)\psi_2(x)\right)dx.\\
\left \langle\pmb{\varphi}, \mathbf{M}\pmb{\psi}\right \rangle=&d_1\displaystyle\int_{L_1}^{L_2}\displaystyle\int_{L_1}^{L_2}J_1(x-y)\varphi_1(x)\psi_1(y)dxdy +d_2\displaystyle\int_{L_1}^{L_2}\displaystyle\int_{L_1}^{L_2}J_2(x-y)\varphi_2(x)\psi_2(y)dxdy\\
&
+\displaystyle\int_{L_1}^{L_2}\left((a_{11}-d_1)\varphi_1(x)\psi_1(x)+a_{12}\varphi_2(x)\psi_1(x)+a_{21}\varphi_1(x)\psi_2(x)+(a_{22}-d_2)\varphi_2(x)\psi_2(x)\right)dx.
\end{align*}
Thanks to $a_{12}=a_{21}$, it is easy to see that $\left \langle \mathbf{M}\pmb{\varphi}, \pmb{\psi}\right \rangle=\left \langle\pmb{\varphi}, \mathbf{M}\pmb{\psi}\right \rangle$, which implies that $\mathbf{M}$ is self-adjoint.

 Let us define
\begin{align}\label{lamda+}
\lambda_- =- \sup\limits_{\left\|\pmb{\varphi}\right\|_{E}=1}\left \langle \mathbf{M}\pmb{\varphi}, \pmb{\varphi}\right \rangle.
\end{align}
To complete the proof, we must first establish the following
\begin{align}\label{3.3}
\lambda_-<-\dfrac{a_{11}+a_{22}-(d_1+d_2)+\sqrt{(a_{11}-d_1-a_{22}+d_2)^2+4a_{12}a_{21}}}{2}.
\end{align}
We choose test function $\pmb{\varphi}=(\varphi_1, \varphi_2)$ satisfying
\begin{align*}
\begin{cases}
&\varphi_1=\dfrac{(a_{11}+a_{22}-d_1-d_2)\left(\sqrt{(a_{11}-d_1-a_{22}+d_2)^2+(a_{12}+a_{21})^2}-(a_{22}-d_2-a_{11}+d_1)\right)}{(a_{12}+a_{21})^2}\varphi_2,\\
&\displaystyle\int_{L_1}^{L_2}(\varphi_1^2+\varphi_2^2)dx=1.
\end{cases}
\end{align*}
Therefore, by the condition \textbf{(J)} and  direct calculation we have
\begin{align*}
\left \langle \mathbf{M}\pmb{\varphi}, \pmb{\varphi}\right \rangle=& d_1\displaystyle\int_{L_1}^{L_2}\displaystyle\int_{L_1}^{L_2}J_1(x-y)\varphi_1^2dxdy+d_2\displaystyle\int_{L_1}^{L_2}\displaystyle\int_{L_1}^{L_2}J_2(x-y)\varphi_2^2dxdy\\
&\hspace{3cm}+\displaystyle\int_{L_1}^{L_2}\left((a_{11}-d_1)\varphi_1^2+(a_{12}+a_{21})\varphi_1\varphi_2+(a_{22}-d_2)\varphi_2^2\right)dx\\
>& \displaystyle\int_{L_1}^{L_2}\left((a_{11}-d_1)\varphi_1^2+(a_{12}+a_{21})\varphi_1\varphi_2+(a_{22}-d_2)\varphi_2^2\right)dx\\
=&
 \dfrac{a_{11}+a_{22}-(d_1+d_2)+\sqrt{(a_{11}-d_1-a_{22}+d_2)^2+4a_{12}a_{21}}}{2},
\end{align*}
which implies $\eqref{3.3}$.
From the equation $\eqref{lamda+}$, it is standard that there is a sequence $\left\{\pmb{\varphi}_n\right\}\subset \mathbf{E}$ with $\left\|\pmb{\varphi}_n\right\|_{\mathbf{E}} = 1, (\forall n\in \mathbb{N})$ such that
\begin{align*}
\lim\limits_{n\to+\infty}\left\|\left(\mathbf{M} + \lambda_-I_2\right)\pmb{\varphi}_n\right\|_{\mathbf{E}}=0.
\end{align*}
Let us define the operator : $\widetilde{\mathbf{H}}: \mathbf{E}\to \mathbf{E}$ satisfies
$\left(\widetilde{\mathbf{H}}\pmb{\varphi}\right)(x) = \left[-\lambda_-I_2 - \mathbf{A}\right]\pmb{\varphi}(x).$
It follows by \eqref{3.3}, we deduce that $\det(-\lambda_-I-\mathbf{A})<0$, so $\widetilde{\mathbf{H}}$ has a bounded inverse. It is notice that
\begin{align*}
\left(\mathbf{M} - \lambda_-I\right)\pmb{\varphi}_n = \mathbf{X}\pmb{\varphi}_n - \widetilde{\mathbf{H}}\pmb{\varphi}_n, \,\,\,\text{and}\,\,\, \widetilde{\mathbf{H}}^{-1}\mathbf{X}\pmb{\varphi}_n - \pmb{\varphi}_n = \widetilde{\mathbf{H}}^{-1}\left(\mathbf{X}\pmb{\varphi}_n - \widetilde{\mathbf{H}}\pmb{\varphi}_n\right).
\end{align*}
 Since $\mathbf{X}$ is compact, there is a subsequence, still denoted by $\left\{\pmb{\varphi}_n\right\}$ such that $\mathbf{X}\pmb{\varphi}_n\to {\bf v}\in \mathbf{C}$; let $\pmb{\phi} = \widetilde{\mathbf{X}}^{-1}{\bf v}\in\mathbf{C}$. Then $\lim\limits_{n\to+\infty}\widetilde{\mathbf{H}}^{-1}\mathbf{X}\pmb{\varphi}_n = \widetilde{\mathbf{H}}^{-1}{\bf v}$. This implies $\widetilde{\mathbf{H}}^{-1}\mathbf{X}\pmb{\varphi}_n - \pmb{\varphi}_n \to \pmb{0},\,\,\pmb{\varphi}_n \to \pmb{\phi}$. We deduce that $\widetilde{\mathbf{H}}^{-1}\mathbf{X}\pmb{\phi} = \pmb{\phi}$, which leads to 
 \begin{align}\label{eigenfunction}
 \mathbf{M}\pmb{\phi} +\lambda_-\pmb{\phi}=0.
 \end{align}
 Note that $\pmb{\phi}\neq \pmb{0}$ since $\|\pmb{\varphi}_n\|_{\mathbf{E}}=1$ and $\pmb{\varphi}_n\rightarrow \pmb{\phi}$. Therefore  $\pmb{\psi}=\left(\phi_1,\phi_2\right)$  (assumed normalised) is an eigenfunction of $\mathbf{M}$ corresponding to the eigenvalue $\lambda_p=\lambda_-$, and $\phi$ is continuous. Note that,  with test function $\pmb{\psi}=\left(|\phi_1|,|\phi_2| \right)$, we have
 \begin{align*}
 \begin{array}{lll}
& \left \langle \mathbf{M}\pmb{\psi},\pmb{\psi}\right \rangle+\lambda_-= \left \langle \mathbf{M}\pmb{\psi},\pmb{\psi}\right \rangle-\left \langle \mathbf{M}\pmb{\psi},\pmb{\psi}\right \rangle\\= & d_1\displaystyle\int_{L_1}^{L_2}\displaystyle\int_{L_1}^{L_2}J_1(x-y)\left(|\phi_1(x)\phi_1(y)|-\phi_1(x)\phi_1(y)\right)dxdy+\displaystyle\int_{L_1}^{L_2}(a_{12}+a_{21})\left(|\varphi_1(x)\phi_2(x)|-\phi_1(x)\phi_2(x)\right)dx\\
 &
\hspace{5cm}+d_2\displaystyle\int_{L_1}^{L_2}\displaystyle\int_{L_1}^{L_2}J_2(x-y)\left(|\phi_2(x)\phi_2(y)|-\phi_2(x)\phi_2(y)\right)dxdy  \geq 0.
 \end{array}
 \end{align*} 
 Hence, by $\eqref{lamda+}$ we induce that
 $\displaystyle\int_{L_1}^{L_2}\displaystyle\int_{L_1}^{L_2}J_i(x-y)\left(|\phi_i(x)\phi_i(y)|-\phi_i(x)\phi_i(y)\right)dxdy=0, i=1,2,\;\text{and}$
  $\displaystyle\int_{L_1}^{L_2}(a_{12}+a_{21})\left(|\phi_1(x)\phi_2(x)|-\phi_1(x)\phi_2(x)\right)dx=0.$
 Since the functions under the integrals are continuous and non-negative, we follow that
 \begin{align}\label{3.10}
 \phi_i(x)\phi_i(y)\geq 0,\text{ for all }x,y\in \left[L_1, L_2\right],i=1,2, 
 \end{align}
Now, if there exists $x_0\in (L_1,L_2)$ such that $\phi_i(x_0)<0$, then from $\eqref{3.10}$,
 we have $-\pmb{\phi}\in \mathbf{E}^{+}$. In this case, we can choose $-\pmb{\phi}$, which is non-negative eigenfunction. Note that $\phi$ cannot be a eigenfunction if $\phi_1(x_0)=0$ or $\phi_2(x_0)=0$ for some $x_0\in \left[L_1, L_2\right]$. For if it were, assume $\phi_1(x_0)=0$, putting $x=x_0$ in equation $\eqref{eigenfunction}$ gives
\begin{align}\label{x0}
\left\{\begin{array}{lll}
d_1\displaystyle\int_{L_1}^{L_2}J_1(x_0-y)\phi_1(y)dy+a_{12}\phi_2(x_0)=0,\\
d_2\displaystyle\int_{L_1}^{L_2}J_2(x_0-y)\phi_2(y)dy+(a_{22}-d_2)\phi_2(x_0)=\lambda_+\phi_2(x_0).
\end{array}\right.
\end{align}
Due to ${\bf {(J)}}$ and $\phi_2\geq 0$ then the first equation of $\eqref{x0}$ yields that $\phi_1(x)=0,\forall x\in [L_1,L_2]$ and $\phi_2(x_0)=0$. Now, the second equation of $\eqref{x0}$ leads to $\phi_2(x)=0,\forall x\in [L_1,L_2]$. This contradiction proves that $\pmb{\phi}>0$. The uniqueness is easily obtained by a simple consequence, for if $\pmb{\phi}$ and $\pmb{\psi}$ were different eigenfunctions, $\pmb{\phi}-\pmb{\psi}$ would be an eigenfunction. But this may change sign, contradicting the positive. In orther to show the conclusion, it is sufficient to prove that $\lambda\leq \lambda_-$ for any eigenvalue $\lambda$. Assume by contradiction that $\lambda>\lambda_-$ we have $\left \langle \mathbf{M} \pmb{\varphi}, \pmb{\varphi} \right \rangle \leq \lambda_-\|\pmb{\varphi}\|^{2}, \; \forall \pmb{\varphi} \in E,$
and thus
\begin{align*}
\left \langle \lambda \pmb{\varphi}- \mathbf{M} \pmb{\varphi}, \pmb{\varphi}\right \rangle  \geq\left(\lambda- \mathbf{M}\right)\|\pmb{\varphi}\|^{2}=\alpha\|\pmb{\varphi}\|^{2}, \forall \pmb{\varphi} \in \mathbf{E} \;\text{ with } \alpha>0.
\end{align*}
Applying Lax–Milgram’s theorem, we deduce that $\lambda I-\mathbf{M}$ is bijective
and thus $\lambda\in \rho(\mathbf{M})$. This is impossible as $\lambda$ is eigenvalue of $\mathbf{M}$.
\end{proof}
\begin{remark}
The definition of the principal eigenvalue $\lambda_p$ as in Theorem \eqref{theorem_3.4} , which is equivalent to
\begin{align*}
&\lambda_p(L_1,L_2) \\&= -\sup\limits_{\left\|\pmb{\varphi}\right\|_{\mathbf{E}} =1}\left \langle \mathbf{M} \pmb{\varphi}, \pmb{\varphi} \right \rangle
\\&= \inf\limits_{\left\|\pmb{\varphi}\right\|_{\mathbf{E}} =1}\left\{\begin{array}{lll}\dfrac{d_1}{2}
\displaystyle\int_{L_1}^{L_2}\displaystyle\int_{L_1}^{L_2}J_1(x-y)\left(\varphi_1(x)-\varphi_1(y)\right)^2dxdy + \dfrac{d_2}{2}
\displaystyle\int_{L_1}^{L_2}\displaystyle\int_{L_1}^{L_2}J_2(x-y)\left(\varphi_2(x)-\varphi_2(y)\right)^2dxdy\\
 -\displaystyle\int_{L_1}^{L_2}\left[\left(a_{11}-d_1 +d_1k_1(x)\right)\varphi_1^2(x)+(a_{12}+a_{21})\varphi_1(x)\varphi_2(x)+\left(a_{22}-d_2+ d_2k_2(x)\right)\varphi_2^2(x)\right]dx
\end{array}\right\},
\end{align*}
where $k_i(x):= \displaystyle\int_{L_1}^{L_2}J_i(x-y)dy,\; i=1,2$.
\end{remark}
\begin{remark}\label{remark7}
	Since $(a_{ij})_{i,j=1,2}$ are constants, by using the variational formulations we show that $\lambda_p(L_1,L_2)$ depends only on $L:=L_2-L_1$. Let $\widetilde{\mathbf{E}}= L^2\left([0,L]\right)\times L^2\left([0,L]\right)$, we have the principal eigenvalue $\lambda_p(0,L):=\lambda_p(L)$ of problem \eqref{PEV} in terms of the variational formulations
	$\lambda_p(L)=\sup_{\left\|\widetilde{\pmb{\varphi}}\right\|_{\widetilde{\mathbf{E}}}=1}\left \langle \mathbf{M}\widetilde{\pmb{\varphi}}, \widetilde{\pmb{\varphi}}\right\rangle.$
	
	For any $\pmb{\varphi}\in \mathbf{E}$ such that $\|\pmb{\varphi}\|_{\mathbf{E}}=1$, 
	by integrating transformations $\widetilde{x}=x-L_1,\widetilde{y}=y-L_1$ and seting $\widetilde{\pmb{\varphi}}(x):=\pmb{\varphi}(x+L_1)$ for $x\in (0,L)$, then we have $\|\widetilde{\pmb{\varphi}}\|_{\mathbf{E}}=1$ and arrive at
	\begin{align*}
	\left\langle\mathbf{M}\pmb{\varphi}, \pmb{\varphi}\right\rangle=&d_1\displaystyle\int_{L_1}^{L_2}\displaystyle\int_{L_1}^{L_2}J_1(x-y)\varphi_1(x)\varphi_1(y)dxdy +d_2\displaystyle\int_{L_1}^{L_2}\displaystyle\int_{0}^{L}J_2(x-y)\varphi_2(x)\varphi_2(y)dxdy\\
	&
	+\displaystyle\int_{L_1}^{L_2}\left((a_{11}-d_1)\varphi_1^2(x)+(a_{12}+a_{21})\varphi_1(x)\varphi_2(x)+(a_{22}-d_2)\varphi_2^2(x)\right)dx
	\\
	= & d_1\displaystyle\int_{0}^{L}\displaystyle\int_{L_1}^{L_2}J_1(\widetilde{x}-\widetilde{y})\widetilde{\varphi}_1(\widetilde{x})\widetilde{\varphi}_1(\widetilde{x})d\widetilde{x}d\widetilde{y} +d_2\displaystyle\int_{0}^{L}\displaystyle\int_{0}^{L}J_2(x-y)\widetilde{\varphi}_2(\widetilde{x})\widetilde{\varphi}_2(\widetilde{x})d\widetilde{x}d\widetilde{y}\\
	&
	+\displaystyle\int_{0}^{L}\left((a_{11}-d_1)\widetilde{\varphi}_1^2(\widetilde{x})+(a_{12}+a_{21})\widetilde{\varphi}_1(\widetilde{x})\widetilde{\varphi}_2(\widetilde{x})+(a_{22}-d_2)\widetilde{\varphi}_2^2(\widetilde{x})\right)d\widetilde{x}\\
	=&\left\langle\mathbf{M}\widetilde{\pmb{\varphi}},\widetilde{\pmb{\varphi}}\right\rangle,
	\end{align*}
	which implies that
	$\lambda_p(L_1,L_2)=\lambda_p(L).$
\end{remark}
\begin{proposition}\label{proposition_3.4}
	Assume ${\bf(J)}$ hold and let $\mathbf{A} = \begin{pmatrix}
		a_{11}&a_{12}\\
		a_{21} & a_{22}
	\end{pmatrix},\mathbf{B} = \begin{pmatrix}
	b_{11}&b_{12}\\
	b_{21} & b_{22}
	\end{pmatrix}$, where $a_{12}=a_{21}, b_{12}=b_{21}$.
	
a) If there exists a function $\widetilde{\pmb{\varphi}}=\left(\widetilde{\varphi}_1, \widetilde{\varphi}_2\right)\in C\left(\left[L_1, L_2\right]\right)\times C\left(\left[L_1, L_2\right]\right)$ with $\widetilde{\varphi}_1, \widetilde{\varphi}_2\geq, \not\equiv 0$ in $\left[L_1, L_2\right]$ and a constant $\widetilde{\lambda}$ such that
\begin{align}\label{EPV}
\left(\mathbf{DN} - \mathbf{D} +\mathbf{A}\right)[\widetilde{\pmb{\varphi}}]+\widetilde{\lambda}\widetilde{\pmb{\varphi}}\leq 0,\;\;\text{in}\;\; \left[L_1, L_2\right],
\end{align}
then $\lambda_p(L)\geq \widetilde{\lambda}$, where $\lambda_p(L)$ is the eigenvalue of problem \eqref{PEV}. Moreover, $\lambda_p(L)= \widetilde{\lambda}$ only if equality holds in \eqref{EPV}.

b) If there exists a function $\widetilde{\pmb{\varphi}}=\left(\widetilde{\varphi}_1, \widetilde{\varphi}_2\right)\in C\left(\left[L_1, L_2\right]\right)\times C\left(\left[L_1, L_2\right]\right)$ with $\widetilde{\varphi}_1, \widetilde{\varphi}_2\geq, \not\equiv 0$ in $\left[L_1, L_2\right]$ and a constant $\widetilde{\lambda}$ such that
\begin{align}\label{EPV2}
\left(\mathbf{DN} - \mathbf{D} +\mathbf{A}\right)[\widetilde{\pmb{\varphi}}]+\widetilde{\lambda}\widetilde{\pmb{\varphi}}\geq 0,\;\;\text{in}\;\; \left[L_1, L_2\right],
\end{align}
then $\lambda_p(L)\leq \widetilde{\lambda}$, where $\lambda_p(L)$ is the eigenvalue of problem \eqref{PEV}. Moreover, $\lambda_p(L)= \widetilde{\lambda}$ only if equality holds in \eqref{EPV2}.

 c) If $a_{ij}\leq b_{ij}$ for all $i,j=1,2$ then $\lambda_p(\mathbf{A})+m\leq \lambda_p(\mathbf{B})$, where $m=\min\{a_{11}-b_{11},a_{22}-b_{22}\}$.
\end{proposition}
\begin{proof}[\bf Proof]
First, we just prove the conclusion a) since the proof of b) is similarly. Let $\pmb{\varphi}=\left(\varphi_1, \varphi_2\right)$ be positive eigenfunction corresponding to the principal eigenvalue $\lambda_p(L)$ in \eqref{PEV}. Denote $\mathbf{M}= \mathbf{DN} - \mathbf{D} +\mathbf{A}$. We have
\begin{align*}
\left\langle\mathbf{M}\widetilde{\pmb{\varphi}},\pmb{\varphi}\right\rangle&=\displaystyle\int_{L_1}^{L_2}\left(d_1\displaystyle\int_{L_1}^{L_2}J_1(x-y)\widetilde{\varphi}_1(y)-d_1\widetilde{\varphi}_1(x)+a_{11}\widetilde{\varphi}_1(x)+a_{12}\widetilde{\varphi}_2(x)\right)\varphi_1(x)dx\\
&+\displaystyle\int_{L_1}^{L_2}\left(d_2\displaystyle\int_{L_1}^{L_2}J_2(x-y)\widetilde{\varphi}_2(y)-d_2\widetilde{\varphi}_2(x)+a_{21}\widetilde{\varphi}_1(x)+a_{22}\widetilde{\varphi}_2(x)\right)\varphi_2(x)dx\\
&\leq -\widetilde{\lambda}\int_{L_1}^{L_2}(\widetilde{\varphi}_2(x)\varphi_1(x)+\widetilde{\varphi}_2(x)\varphi_2(x))dx,
\end{align*}
which implies that
$\left\langle\mathbf{M}\widetilde{\pmb{\varphi}},\pmb{\varphi}\right\rangle\leq -\widetilde{\lambda}\left\langle\widetilde{\pmb{\varphi}},\pmb{\varphi}\right\rangle .$
Due to the definition of $\pmb{\varphi}$ we also have
$\left\langle\mathbf{M}\pmb{\varphi},\widetilde{\pmb{\varphi}}\right\rangle= -\lambda_p(L)\left \langle\widetilde{\pmb{\varphi}},\pmb{\varphi}\right\rangle.$
Since $\mathbf{M}$ is self-adjiont and $\left\langle\widetilde{\pmb{\varphi}}, \pmb{\varphi}\right\rangle>0$, we obtain that $\lambda_p(L)\geq \widetilde{\lambda}$. Moreover, if one of the inequalities in \eqref{PEV} is strict at some point $x_0\in \left[L_1, L_2\right]$, then $\lambda_p(L)> \widetilde{\lambda}$.

Next, we prove the conclusion c).	Let $\left(\varphi_1^{\mathbf{A}},\varphi_2^{\mathbf{A}}\right)$ and $\left(\varphi_1^{\mathbf{B}},\varphi_2^{\mathbf{B}}\right)$   be the  corresponding eigenfunction to $\lambda_p\left(\mathbf{A}\right)$ and $\lambda_p\left(\mathbf{B}\right)$, respectively. Then we have
	\begin{align*}
	&d_1\displaystyle\int_{L_1}^{L_2}J_1(x-y)\varphi_1^\mathbf{A}(y)dy-d_1\varphi_1^\mathbf{A}(x)+b_{11}\varphi_1^\mathbf{A}(x)+b_{12}\varphi_2^\mathbf{A}(x)+\left(\lambda_p(\mathbf{A})+m\right)\varphi_1^\mathbf{A}\\
	\leq& d_1\displaystyle\int_{L_1}^{L_2}J_1(x-y)\varphi_1^\mathbf{A}(y)dy-d_1\varphi_1^\mathbf{A}(x)+(b_{11}+m)\varphi_1^\mathbf{A}(x)+a_{12}\varphi_2^\mathbf{A}(x)+\lambda_p(\mathbf{A})\varphi_1^\mathbf{A}\\
	\leq& d_1\displaystyle\int_{L_1}^{L_2}J_1(x-y)\varphi_1^\mathbf{A}(y)dy-d_1\varphi_1^\mathbf{A}(x)+a_{11}\varphi_1^\mathbf{A}(x)+a_{12}\varphi_2^\mathbf{A}(x)+\lambda_p(\mathbf{A})\varphi_1^\mathbf{A},
	\end{align*}
	which implies that
	\begin{align*}
	d_1\displaystyle\int_{L_1}^{L_2}J_1(x-y)\varphi_1^\mathbf{A}(y)dy-d_1\varphi_1^\mathbf{A}(x)+b_{11}\varphi_1^\mathbf{A}(x)+b_{12}\varphi_2^\mathbf{A}(x)+\left(\lambda_p(\mathbf{A})+m\right)\varphi_1^\mathbf{A}\leq 0.
	\end{align*}
	Similarly, we also obtain
	\begin{align*}
	d_2\displaystyle\int_{L_1}^{L_2}J_2(x-y)\varphi_2^\mathbf{A}(y)dy-d_2\varphi_2^\mathbf{A}(x)+b_{21}\varphi_1^\mathbf{A}(x)+b_{22}\varphi_2^\mathbf{A}(x)+\left(\lambda_p(\mathbf{A})+m\right)\varphi_2^\mathbf{A}\leq 0.
	\end{align*}
	Using the conclusion a), we get $\lambda_p(\mathbf{A})+m\leq \lambda_p(\mathbf{B})$.
\end{proof}
We now establish a maximum principle for the operator $\mathbf{M}$ defined in \eqref{PEV}, which will be useful for our later analysis.
\begin{proposition}[Maximum principle]\label{mpnt}
 Assume \textbf{(J)} holds and $a_{12}=a_{21}>0$. If $\lambda_p(L)\geq0$, then for all function $\pmb{\varphi}=\left (\varphi_1,\varphi_2\right)\in C\left(\left[L_1,L_2\right]\right)\times C\left(\left[L_1,L_2\right]\right)$ satisfying
 \begin{align*}
 \mathbf{M}\pmb{\varphi}\leq 0,\qquad &\text{ in } \left(L_1, L_2\right),\\
 \pmb{\varphi} \geq 0, \qquad &x=L_1 \text{ or }x=L_2,
 \end{align*}
 then we have $\pmb{\varphi}\geq 0$ in $\left(L_1, L_2\right)$.
\end{proposition}
\begin{proof}[\bf Proof]
Let $\pmb{\varphi}=\left(\varphi_1,\varphi_2\right)\in C\left(\left[L_1, L_2\right]\right)\times C\left(\left[L_1, L_2\right]\right), \pmb{\varphi}\not \equiv 0$ satisfies $\varphi_i\left(L_j\right)\geq 0$ with $i,j=1,2$ and
\begin{align}\label{mp3}
d_1\displaystyle\int_{L_1}^{L_2}J_1(x-y)\varphi_1(y)dy-d_1\varphi_1(x)+a_{11}\varphi_1(x)+a_{12}\varphi_2(x)\leq 0,\quad x\in \left(L_1, L_2\right),
\end{align}
\begin{align}\label{mp4}
d_2\displaystyle\int_{L_1}^{L_2}J_2(x-y)\varphi_2(y)dy-d_2\varphi_2(x)+a_{21}\varphi_1(x)+a_{22}\varphi_2(x)\leq 0, \quad x\in \left(L_1, L_2\right).
\end{align}
	Let $\pmb{\varphi}^p=(\varphi_1^p,\varphi_2^p)\in  C\left(\left[L_1, L_2\right]\right)\times C\left(\left[L_1, L_2\right]\right)$ be the corresponding eigenfunction to $\lambda_p(L)$. Then we have $\varphi^p>0$ and
	\begin{align}\label{mp8}
	d_1\displaystyle\int_{L_1}^{L_2}J_1(x-y)\varphi_1^p(y)dy-d_1\varphi_1^p(x)+a_{11}\varphi_1^p(x)+a_{12}\varphi_2^p(x)+\lambda_p(L_1,L_2)\varphi_1^p(x)=0,\quad x\in \left[L_1, L_2\right],
	\end{align}
	\begin{align}\label{mp9}
	d_2\displaystyle\int_{L_1}^{L_2}J_2(x-y)\varphi_2^p(y)dy-d_2\varphi_2^p(x)+a_{21}\varphi_1^p(x)+a_{22}\varphi_2^p(x)+\lambda_p(L_1,L_2)\varphi_2^p(x)=0, \quad x\in \left[L_1,L_2\right].
	\end{align}
	 Let us define
	\begin{align*}
	\pmb{\psi}=(\psi_1,\psi_2):=\left(\dfrac{\varphi_1}{\varphi_1^p},\dfrac{\varphi_2}{\varphi_2^p}\right).
	\end{align*}
Combining \eqref{mp3} and \eqref{mp8}, we have
	\begin{align}\label{mp1}
	&d_1\int_{L_1}^{L_2}J_1(x-y)\varphi_1^p(y)(\psi_1(y)-\psi_1(x))dy-\lambda_p(L)\varphi_1^p(x)\psi_1(x)\nonumber\\
	=&d_1\int_{L_1}^{L_2}J_1(x-y)\left(\varphi_1(y)-\varphi_1^p(y)\dfrac{\varphi_1(x)}{\varphi_1^p(x)}\right)dy-\lambda_p(L)\varphi_1^p(x)\dfrac{\varphi_1(x)}{\varphi_1^p(x)}\nonumber\\
	=&d_1\int_{L_1}^{L_2}J_1(x-y)\varphi_1(y)dy-\dfrac{\varphi_1(x)}{\varphi_1^p(x)}\left(d_1\int_{L_1}^{L_2}J_1(x-y)\varphi_1^p(y)dy+\lambda_p(L)\varphi_1^p(x)\right)\nonumber\\
	=& d_1\int_{L_1}^{L_2}J_1(x-y)\varphi_1(y)dy+\dfrac{\varphi_1(x)}{\varphi_1^p(x)}\left[a_{11}\varphi_1^p(x)+a_{12}\varphi_2^p(x)-d_1\varphi_1^p(x)\right]\nonumber\\
	=& d_1\int_{L_1}^{L_2}J_1(x-y)\varphi_1(y)dy-d_1\varphi_1(x)+a_{11}\varphi_1(x)+a_{12}\varphi_1(x)\dfrac{\varphi_2^p(x)}{\varphi_1^p(x)}\nonumber\\
	\leq&-a_{12}\varphi_2(x)+a_{12}\varphi_1(x)\dfrac{\varphi_2^p(x)}{\varphi_1^p(x)}
	=a_{12}\varphi_2^p(x)\left(\psi_1(x)-\psi_2(x)\right), \;\text{for any}\; x\in \left(L_1, L_2\right).
	\end{align}
	Similarly, combining \eqref{mp4} and \eqref{mp9}, we also obtain that
	\begin{align}\label{mp2}
	d_2\int_{L_1}^{L_2}J_2(x-y)\varphi_2^p(y)\left(\psi_2(y)-\psi_2(x)\right)dy-\lambda_p(L)\varphi_2^p(x)\psi_2(x)
	\leq a_{21}\varphi_1^p(x)[\psi_2(x)-\psi_1(x)].
	\end{align}
Since $\psi_1,\psi_2$ are  continuous functions in $\left[L_1,L_2\right]$, $\psi_1,\psi_2$ achieve at some $x_1,x_2\in \left[L_1,L_2\right]$ a minimum, respectively, i.e $\psi_1(x_1)=\min_{[L_1,L_2]}\psi_1(x)$ and $\psi_2(x_2)=\min_{\left[L_1,L_2\right]}\psi_2(x)$. Without loss of generality, we just need to prove $\psi_1(x_1)\geq 0$ . If it is not true and due to $\varphi_1\left(L_1\right),\varphi_1\left(L_2\right)\geq 0$, then $\psi_1(x_1)<0$ with $x_1\in \left(L_1,L_2\right)$. By putting $x=x_1$ into \eqref{mp1} we have
	\begin{align*}
	0\leq \int_{L_1}^{L_2}J_1(x_1-y)\varphi_1^p(y)\left(\psi_1(y)-\psi_1(x_1)\right)dy-\lambda_p(L)\varphi_1^p(x_1)\psi_1(x_1)\leq a_{12}\varphi_2^p(x_1)\left[\psi_1(x_1)-\psi_2(x_1)\right].
	\end{align*}
It follows that $\psi_2(x_1)\leq \psi_1(x_1)< 0$. Therefore, $\psi_2(x_2)=\min_{[L_1,L_2]}\psi_2(x)\leq \psi_2(x_1)< 0$. Then putting $x=x_2$ into \eqref{mp2} we arrive at
	\begin{align*}
		0\leq\int_{L_1}^{L_2}J_2(x_2-y)\varphi_2^p(y)\left(\psi_2(y)-\psi_2(x_2)\right)dy-\lambda_p(L)\varphi_2^p(x)\psi_2(x_2)
	\leq& a_{21}\varphi_1^p(x_1)\left[\psi_2(x_2)-\psi_1(x_2)\right],
	\end{align*}
	which implies that 
	$\psi_1\left(x_2\right)\leq\psi_2\left(x_2\right)\leq \psi_2\left(x_1\right)\leq\psi_1\left(x_1\right).$
	Since $\psi_1(x)$ achieve minimum at $x_1$, we deduce that $\psi_1\left(x_1\right)=\psi_2\left(x_1\right)$ and then
	\begin{align*}
	\int_{L_1}^{L_2}J_1\left(x_1-y\right)\varphi_1^p(y)\left(\psi_1(y)-\psi_1(x_1)\right)dy=0.
	\end{align*}
	Since $\varphi_1^p,\psi_1$ and $J_1$ are continuous and non-negative functions, the above inequality leads to
	\begin{align}\label{mp5}
	J_1(x_1-y)(\psi_1(y)-\psi_1(x_1))=0, \text { for all } y\in\left(L_1,L_2\right).
	\end{align}
    By \textbf{(J)}  there exist constants $\varepsilon>0$ and $\delta>0$ such that
    \begin{align*}
    J_1(x_1-y)\geq \varepsilon>0, \text{ for all } y\in \left(x_1-\delta, x_1+\delta\right).
    \end{align*}
    By \eqref{mp5} we easily see that
    \begin{align}\label{mpe}
    \psi_1(y)=\psi_1(x_1), \text{ for all } y\in \left(x_1-\delta, x_1+\delta\right).
    \end{align}
    So $\psi_1$ also achieve negative minimum at $x_1+\delta/2$ and $x_1-\delta/2$. By repeating the above argument with replacing $x_1$ by $x_1+\delta/2$ or $x_1-\delta/2$, we can extend \eqref{mpe} to
    $\psi_1(y)=\psi_1(x), \text{ for all } y\in \left(x_1-\dfrac{3}{2}\delta, x_1+\dfrac{3}{2}\delta\right).$
    	Since $\delta$ does not change after each iteration, by repeating this process finitely many times we induce that $\psi_1(x)$ is constant in $\left(L_1,L_2\right)$ which leads to $\varphi_1\equiv 0$ in $\left(L_1,L_2\right)$. Similarly, we also have $\varphi_2\equiv 0$ in $\left(L_1,L_2\right)$. This is impossible since $\left(\varphi_1, \varphi_2\right)\not\equiv (0,0)$.
\end{proof}
\begin{remark}\label{mpnt1}
	We note that the conclusion of Proposition \ref{mpnt} is still hold when the assumption 
		$\pmb{\varphi}\geq 0 \text{ for } x=L_1 \text{ or } x=L_2$
 are replaced by
	$\mathbf{M}\pmb{\varphi}\leq 0,\,\text{ in } \left[L_1, L_2\right].$
\end{remark}
\subsection{Effects of parameters on the principal spectrum point {$\lambda_p$}}\hspace{8cm}

In this subsection, we study the effects of the dispersal rate $d$ and the dispersal range characterized by $L$ on the principal spectrum point $\lambda_p$.
We recall the existence of the principal eigenvalue for any non-negative matrix, as guaranteed by the Perron–Frobenius Theorem (\cite{Gan})
\begin{theorem}
Given a real-valued square matrix $A = \left(a_{ij}\right)$, whose off-diagonal terms are non-negative, (i.e. $a_{ij}\geq 0$ if $i=j$), there exists a real eigenvalue $\overline{\lambda}\left(A\right)$, corresponding
to a non-negative eigenvector, with the greatest real part (for any eigenvalue $\lambda=\overline{\lambda}\left(A\right),
\overline{\lambda}\left(A\right) > \text{Re} \lambda$). Moreover, if $a_{ij} > 0$ for any $i=j$, then $\overline{\lambda}\left(A\right)$ is simple with strictly positive eigenvector, and it can be characterized as the unique eigenvalue corresponding to a nonnegative vector.
\end{theorem}
\begin{remark}\label{remark3}
For an $n\times n$ matrix $\bf{C}$, scalars $\lambda$ and vectors ${\bf{x}}_{n\times 1}= 0$ satisfying $\bf{C}\bf{x} = \lambda \bf{x}$ are called eigenvalues and eigenvectors of $\bf{C}$, respectively,
and any such pair, $\left (\lambda, \bf{x}\right)$, is called an eigenpair for $\bf{C}$. The set of
distinct eigenvalues, denoted by $\sigma\left(\bf{C}\right)$, is called the spectrum of $\bf{C}$.
\begin{enumerate}
\item  $\lambda\in \sigma\left(\bf{C}\right)\Leftrightarrow  {\bf C} - \lambda \bf{I}$ is singular  $\Leftrightarrow \det\left(\bf{C} -\lambda \bf{I}\right) = 0$. \label{matrix1}
\item $\left\{x\in {\bf{0}}\mid {\bf{x}}\in N\left(\bf{C}-\lambda\bf{I}\right)\right\}$ is the set of all eigenvectors associated with $\lambda$. From now on, $N\left(\bf{C}-\lambda\bf{I}\right)$ is called an eigenspace for $\bf{C}$.
\end{enumerate}
Let’s now face the problem of finding the eigenvalues and eigenvectors of the matrix $\mathbf{A} = \begin{pmatrix}
 a_{11} &a_{12}\\
a_{21} & a_{22} 
\end{pmatrix} $, with $a_{12}=a_{21}>0$. The proof is rather elementary but for the sake of completeness we include all the details. As noted in \eqref{matrix1}, the eigenvalues are the scalars $\lambda$ for which $\det\left(\bf{A} -\lambda\bf{I}\right) = 0$. Expansion of $\det\left(\bf{A} -\lambda\bf{I}\right)$ produces the second-degree polynomia
\begin{align*}
\text{P}_{\bf{A}}\left(\lambda\right)=\det\left(\bf{A} -\lambda\bf{I}\right)=\left|\begin{matrix}
a_{11}-\lambda & a_{12} \\ 
 a_{21}& a_{22}-\lambda 
\end{matrix}\right| = \lambda^2-\left(a_{11}+a_{22}\right)\lambda +a_{11}a_{22}-a_{12}a_{21}.
\end{align*}
Consequently, the eigenvalues for $\bf{A}$ are the solutions of the characteristic equation $\text{P}_{\bf{A}}\left(\lambda\right) = 0$ (i.e.,
the roots of the characteristic polynomial), and they 
are $\lambda_1 =\dfrac{a_{11}+a_{22}-\sqrt{\left(a_{11}-
a_{22}\right)^2+4a_{12}a_{21}}}{2}=\gamma_1$ and $\lambda_2 =
\dfrac{a_{11}+a_{22}+\sqrt{\left(a_{11}-a_{22}
\right)^2+4a_{12}a_{21}}}{2}=\gamma_2$. The eigenvectors associated with $\lambda_1 =\gamma_1$  and $\lambda_2 =
\gamma_2$ are simply the nonzero vectors in the 
eigenspaces $N\left(\bf{A}-\gamma_1\bf{I}\right)$ and $N
\left(\bf{A}-\gamma_2\bf{I}\right)$, respectively. But 
determining these eigenspaces amounts to nothing more than solving the two homogeneous systems, $\left(\bf{A}-\gamma_1\bf{I}\right)\bf{x}=0$ and $\left(\bf{A}-\gamma_2\bf{I}\right)=0$.

For $\lambda_1 = \gamma_1$, ${\left(\bf{A}-\gamma_1\bf{I}\right)\bf{x}=0}\Rightarrow x_1= \dfrac{2a_{12}}{a_{22}-a_{11}-\sqrt{\left(a_{11}-a_{22}
\right)^2+4a_{12}a_{21}}}$, $x_2$ is free. Thus, we have
$N\left(\bf{A}-\gamma_1\bf{I}\right)=\left\{{\bf{x}}\;\left|\; {\bf{x}}=\alpha \begin{pmatrix}
\dfrac{2a_{12}}{a_{22}-a_{11}-\sqrt{\left(a_{11}-a_{22}
\right)^2+4a_{12}a_{21}}} \\ 
 1 
\end{pmatrix}\right.\right\}$.

For $\lambda_1 = \gamma_2$, ${\left(\bf{A}-\gamma_2\bf{I}\right)\bf{x}=0}\Rightarrow x_2= \dfrac{2a_{21}}{a_{22}-a_{11}+\sqrt{\left(a_{11}-a_{22}
\right)^2+4a_{12}a_{21}}}$, $x_1$ is free. We deduce that
$N\left(\bf{A}-\gamma_1\bf{I}\right)=\left\{{\bf{x}}\;\left|\; {\bf{x}}=\beta \begin{pmatrix}
1\\
\dfrac{2a_{12}}{a_{22}-a_{11}+\sqrt{\left(a_{11}-a_{22}
\right)^2+4a_{12}a_{21}}} 
\end{pmatrix}\right.\right\}$.

In other words, the eigenvectors of ${\bf{A}}$ associated with $\lambda_1=\gamma_1$ are all nonzero
multiples of ${\bf{x}}=  \begin{pmatrix}
\dfrac{2a_{21}}{a_{22}-a_{11}+\sqrt{\left(a_{11}-a_{22}
\right)^2+4a_{12}a_{21}}} \\ 
 1 
\end{pmatrix}^T$, and the eigenvectors associated with $\lambda_2 = \gamma_2$ are all nonzero multiples of ${\bf{y}}=\begin{pmatrix}
1\\
\dfrac{2a_{12}}{a_{22}-a_{11}+\sqrt{\left(a_{11}-a_{22}
\right)^2+4a_{12}a_{21}}}
\end{pmatrix}^T$. Therefore, we deduce that

 $\overline{\lambda}({\bf{A}})= \dfrac{a_{11}+a_{22}+\sqrt{\left(a_{11}-a_{22}\right)^2+4a_{12}a_{21}}}{2}$  is the greatest real part eigenvalue of a square matrix $\mathbf{A}$, corresponding to non-negative eigenvector.
\end{remark}

We turn to the study of the effects of the dispersal domain characterized by $L_1,L_2$ on the principal spectrum point $\lambda_p(L_1,L_2)$. 
It is strongly linked to local reaction diffusion equations,  which are established in the important works of Dancer \cite{D}  and Lam-Lou  \cite{LL} for cooperative elliptic systems. We prove the following result.
\begin{proposition}\label{pro.3.4}
Assume the dispersal kernel $J_i$, $i=1,2$ satisfies $\left({\bf J}\right)$,\;\; $\mathbf{Q} = \begin{pmatrix}
 a_{11} -d_1&a_{12}\\
a_{21} & a_{22} -d_2
\end{pmatrix},$ $\mathbf{B} = \begin{pmatrix}
 a_{11} &a_{12}\\
a_{21} & a_{22} 
\end{pmatrix} $ be a constants matrix with $a_{12}=a_{21}$,  and $-\infty<L_1<L_2<\infty$. Then the following hold true.
\begin{enumerate}
\item $\lambda_p\left(L_1,L_2\right)$ is strictly decreasing and continuous in $L:=L_2-L_1$; \label{1a}\\
\item $\lim\limits_{L_2-L_1\to\infty}\lambda_p\left(L_1,L_2\right):= -\overline{\lambda}({\bf{B}})$;\label{1b}\\
\item $\lim\limits_{L_2-L_1\to 0}\lambda_p\left(L_1,L_2\right):= - \overline{\lambda}({\bf{Q}})$.\label{1c}
\end{enumerate} 
\end{proposition}
\begin{proof}[\bf Proof]
Thanks to Remark \ref{remark7}, we only need to prove the stated conlusions for $\lambda_p(L)=\lambda_p(L_1,L_2)$ where $L:=L_2-L_1$.

\noindent (1) The proof of \eqref{1a} can be derived from the definition of eigenvalue.

\noindent (2) From the properties of the kernel function $J_1, J_2$, we obtain

\begin{align*}
\displaystyle\int_{0}^{L}\displaystyle\int_{0}^{L}J_1(x-y)\varphi_1(y)\varphi_1(x)dxdy\leq \displaystyle\int_{0}^{L}\displaystyle\int_{0}^{L}J_1(x-y)\left(\dfrac{\varphi_1^2(y)+\varphi_1^2(x)}{2}\right)dxdy\leq  \displaystyle\int_{0}^{L}\varphi_1^2(x)dx,\\
\displaystyle\int_{0}^{L}\displaystyle\int_{0}^{L}J_2(x-y)\varphi_2(y)\varphi_2(x)dxdy\leq \displaystyle\int_{0}^{L}\displaystyle\int_{0}^{L}J_2(x-y)\left(\dfrac{\varphi_2^2(y)+\varphi_2^2(x)}{2}\right)dxdy\leq \displaystyle\int_{0}^{L}\varphi_2^2(x)dx.
\end{align*}
For all $\pmb{\varphi}=
\left(\varphi_1, \varphi_2\right)\in \mathbf{E}$ and $\left\|\pmb{\varphi}\right\|_{\mathbf{E}}=1$, then 
\begin{align*}
\left\langle\mathbf{M}\pmb{\varphi},\pmb{\varphi}\right\rangle=&d_1\displaystyle\int_{0}^{L}\displaystyle\int_{0}^{L}J_1(x-y)\varphi_1(x)\varphi_1(y)dxdy +d_2\displaystyle\int_{0}^{L}\displaystyle\int_{0}^{L}J_2(x-y)\varphi_2(x)\varphi_2(y)dxdy\\
&
+\displaystyle\int_{0}^{L}\left((a_{11}-d_1)\varphi_1^2(x)+(a_{12}+a_{21})\varphi_1(x)\varphi_2(x)+(a_{22}-d_2)\varphi_2^2(x)\right)dx\\&
 \leq \displaystyle\int_{0}^{L}\left(a_{11}\varphi_1^2(x)+(a_{12}+a_{21})\varphi_1(x)\varphi_2(x)+a_{22}\varphi_2^2(x)\right)dx\\&
\leq  \displaystyle\int_{0}^{L}\dfrac{a_{11}+a_{22}+\sqrt{\left(a_{11}-a_{22}\right)^2+4a_{12}a_{21}}}{2}\left(\varphi_1^2(x) +\varphi_2^2(x)\right)dx\\& = \dfrac{a_{11}+a_{22}+\sqrt{\left(a_{11}-a_{22}\right)^2+4a_{12}a_{21}}}{2}.
\end{align*}
The second inequality can be obtained from analyzing the minimum or maximum value of a function $F$, that is, $F(x,y) = \dfrac{a_{11}x^2+(a_{12}+a_{21})xy+a_{22}y^2}{x^2+y^2} $, and using  the elementary results of the function,  we can find a maximum and minimum value of the function $F$. The last equality  achieved by using the factor $\displaystyle\int_{0}^{L}\left(\varphi_1^2(x) +\varphi_2^2(x)\right)dx=1$.
It then follows from the above variational characterization that for any $L>0$,
\begin{align*}
\lambda_p(L) = -\sup\limits_{\left\|\pmb{\varphi}\right\|_{E} =1}\left\langle\mathbf{M}\pmb{\varphi}, \pmb{\varphi}\right\rangle\geq - \dfrac{a_{11}+a_{22}+\sqrt{\left(a_{11}-a_{22}\right)^2+4a_{12}a_{21}}}{2}.
\end{align*}
To complete the proof, we only need to show that 
\begin{align}\label{eq:3.22}
\lim\limits_{L\rightarrow+\infty}\lambda_p(L)\leq -\dfrac{a_{11}+a_{22}+\sqrt{\left(a_{11}-a_{22}\right)^2+4a_{12}a_{21}}}{2}.
\end{align}
Set $\alpha= a_{11}-a_{22}+ \sqrt{\left(a_{11}-a_{22}\right)^2+4a_{12}a_{21}},\,\,\beta = a_{12}+a_{21}$, let $\pmb{\psi} = \left(\psi_1, \psi_2\right) = \left(\dfrac{\alpha}{\sqrt{L\left(\alpha^2+\beta^2\right)}}, \dfrac{\beta}{\sqrt{L\left(\alpha^2+\beta^2\right)}}\right)$ be the test  function in the
variational characterization of $\lambda_p\left(L\right)$.
We deduce that for any $L>0$,
\begin{align}\label{eq:3.23}
\lambda_p(L)&\leq -\left\langle\mathbf{M}\pmb{\psi}, \pmb{\psi}\right\rangle\nonumber\\ &=-\dfrac{d_1\alpha^2\displaystyle\int_{0}^{L}\displaystyle\int_{0}^{L}J_1(x-y) dxdy+ d_2\beta^2\displaystyle\int_{0}^{L}\displaystyle\int_{0}^{L}J_2(x-y)dx dy}{L\left(\alpha^2+\beta^2\right)}\nonumber\\
&\hspace{6cm}-\dfrac{\left(a_{11}-d_1\right)\alpha^2 +\left(a_{12}+a_{21}\right)\alpha\beta +\left(a_{22}-d_2\right)\beta^2}{\left(\alpha^2+\beta^2\right)}.
\end{align}
To prove \eqref{eq:3.22}, it suffices to show that 
\begin{align}\label{eq:3.24}
\dfrac{1}{L}\displaystyle\int_{0}^{L}\left(\displaystyle\int_{0}^{L}J_i(x-y) dy\right)dx\rightarrow 1,\; \text{ as } L\rightarrow \infty,\; i=1,2.
\end{align}
Indeed, we set the inner integral $y=x-z$, then $x=Lw$, and obtain
\begin{align*}
\dfrac{1}{L}\displaystyle\int_{0}^{L}\left(\displaystyle\int_{0}^{L}J_i(x-y) dy\right)dx=\displaystyle\int_{0}^{1}\left(\displaystyle\int_{Lw-L}^{L}J_i(z) dz\right)dw.
\end{align*}
For fixed $w$ with $w\in (0,1)$. By $\textbf{(J)}$ and  thanks to the dominated convergence theorem, we have
\begin{align*}
\displaystyle\int_{Lw-L}^{L}J_i(z) dz\rightarrow \int_{\mathbb R}J_i(z)dz=1,\qquad \text{ as } L\rightarrow \infty, \; i=1,2,
\end{align*}
which implies \eqref{eq:3.24}.
Now, by letting $L\to \infty$ in \eqref{eq:3.23}, we can derive the stated inequality in \eqref{eq:3.22},
\begin{align*}
\lim\limits_{L\rightarrow\infty}\lambda_p(L)\leq -\dfrac{a_{11}\alpha^2+(a_{12}+a_{21})\alpha\beta+a_{22}\beta^2}{\alpha^2+\beta^2}=-\dfrac{a_{11}+a_{22}+\sqrt{\left(a_{11}-a_{22}\right)^2+4a_{12}a_{21}}}{2}.
\end{align*}
\noindent (3) We show that 
\begin{align}\label{eq:3.25}
\lim\limits_{L\to 0}\lambda_p\left(L\right) = -\dfrac{a_{11}+a_{22} - \left(d_1+d_2\right)+\sqrt{\left(a_{11}-a_{22} -d_1+d_2\right)^2+4a_{12}a_{21}}}{2}.
\end{align}
Since $\lambda_p\left(L\right)$ is a principal eigenvalue, there exists a strictly positive function $\pmb{\varphi}^p=\left(\varphi^p_1, \varphi^p_2\right)\in \left(C\left([0, L]\right)\right)^2$ such that
\begin{align*}
	\left\{\begin{array}{ll}
	d_1\displaystyle\int_{0}^{L} J_1(x-y) \varphi^p_1(y)dy-d_1\varphi^p_1+ a_{11}\varphi_1+ a_{12}\varphi^p_2 + \lambda_p\left( L\right)\varphi^p_1=0, &  \text{in}\; [0, L],  \\
	d_2\displaystyle\int_{0}^{L} J_2(x-y) \varphi^p_2(y)dy-d_2\varphi^p_2+ a_{21}\varphi^p_1+a_{22}\varphi^p_2 + \lambda_p\left(L\right)\varphi^p_2=0, &  \text{in}\; [0, L].
	\end{array}\right.
	\end{align*}
	Therefore, we have
\begin{align*}
\left\langle\mathbf{M}\pmb{\varphi}^p, \pmb{\varphi}^p\right\rangle & =d_1\displaystyle\int_{0}^{L}\displaystyle\int_{0}^{L}J_1(x-y)\varphi^p_1(x)\varphi^p_1(y)dxdy +d_2\displaystyle\int_{0}^{L}\displaystyle\int_{0}^{L}J_2(x-y)\varphi^p_2(x)\varphi^p_2(y)dxdy\\
&
+\displaystyle\int_{0}^{L}\left((a_{11}-d_1)(\varphi^p_1)^2(x)+(a_{12}+a_{21})\varphi^p_1(x)\varphi^p_2(x)+(a_{22}-d_2)(\varphi^p_2)^2(x)\right)dx\\&
 \leq 
d_1\left\|J_1\right\|_{\infty}\left(\displaystyle\int_{0}^{L}\varphi^p_1(x)dx\right)^2+d_2\left\|J_2\right\|_{\infty}\left(\displaystyle\int_{0}^{L}\varphi^p_2(x)dx\right)^2\\& +\dfrac{a_{11}+a_{22} - \left(d_1+d_2\right)+\sqrt{\left(a_{11}-a_{22} -d_1+d_2\right)^2+4a_{12}a_{21}}}{2}\\& \leq \left(d_1\left\|J_1\right\|_{\infty}+ d_2\left\|J_2\right\|_{\infty}\right)L+\dfrac{a_{11}+a_{22} - \left(d_1+d_2\right)+\sqrt{\left(a_{11}-a_{22} -d_1+d_2\right)^2+4a_{12}a_{21}}}{2}.
\end{align*}
Let $L\to 0$ and get the desired inequality
\begin{align}\label{eq:3.26}
\lim\limits_{L\to 0}\lambda_p\left(0,L\right)=-\lim\limits_{L\to 0}\left\langle\mathbf{M}\pmb{\varphi}^p, \pmb{\varphi}^p\right\rangle \geq- \dfrac{a_{11}+a_{22} - \left(d_1+d_2\right)+\sqrt{\left(a_{11}-a_{22} -d_1+d_2\right)^2+4a_{12}a_{21}}}{2}.
\end{align}
Let $\alpha = a_{11}-d_1-a_{22}+d_2+ \sqrt{\left(a_{11}-d_1-a_{22}+d_2\right)^2+4a_{12}a_{21}},\,\beta= a_{12}+a_{21}$ be constants. Denote by $\pmb{\psi} = \left(\dfrac{\alpha}{\sqrt{\alpha^2 + \beta^2}}, \dfrac{\beta}{\sqrt{\alpha^2 + \beta^2}}\right)$  the test  function in the variational characterization of $\lambda_p\left(0, L\right)$. It then follows that
$\lambda_p(L) \leq -\left\langle\mathbf{M}\pmb{\psi}, \pmb{\psi}\right\rangle< -\dfrac{a_{11}+a_{22} - \left(d_1+d_2\right)+\sqrt{\left(a_{11}-a_{22} -d_1+d_2\right)^2+4a_{12}a_{21}}}{2}.$

Passing $L\to 0$, we derive 
$\liminf\limits_{L\to 0}\lambda_p\left(L\right) \leq -\dfrac{a_{11}+a_{22} - \left(d_1+d_2\right)+\sqrt{\left(a_{11}-a_{22} -d_1+d_2\right)^2+4a_{12}a_{21}}}{2}.$
This together with \eqref{eq:3.26}, we proved \eqref{eq:3.25}.
\end{proof}
\begin{remark}
In the proof \eqref{1b}, we utilize a more flexible way rather than Cao.et.al \cite{CDLL}. More precisely, we proved $\dfrac{1}{L}\displaystyle\int_{0}^{L}\left(\displaystyle\int_{0}^{L}J_i(x-y) dy\right)dx\rightarrow 1,\; \text{ as } L\rightarrow \infty,\; i=1,2$, which is obtained by using assumption ${\bf (J)}$ and dominated convergence theorem.
\end{remark}
Next, we investigate the effects of the dispersal rate characterized by $d$ on the principal spectrum point. 
 We prove the following result.
\begin{proposition}\label{pro.3.6}
Assume that the dispersal kernel $J_i$, $i=1,2$ satisfy $\left({\bf J}\right)$, $\mathbf{B} = \begin{pmatrix}
 a_{11} &a_{12}\\
a_{21} & a_{22} 
\end{pmatrix} $ be a constants matrix with $a_{12}=a_{21}$, $d=d_1= d_2$  and fix $L_1,L_2>0$. Then the statements below about $\lambda_p\left(d\right)$ hold.
\begin{enumerate}
\item $\lambda_p\left(d\right)$ is a strictly monotone increasing function in $d$;\label{1}

\item $\lim\limits_{d\to \infty}\lambda_p\left(d\right)= +\infty$;\label{2}

\item $\lim\limits_{d\to 0}\lambda_p\left(d\right) =- \overline{\lambda}\left(\mathbf{B}\right)$\label{3}.
\end{enumerate} 
\end{proposition}
\begin{proof}[\bf Proof]

\noindent (1) Let $\pmb{\varphi}^p = \left(\varphi_1^p, \varphi_2^p\right)$ be the  corresponding eigenfunction to $\lambda_p\left(d\right)$ and normalized it is $\left\|\pmb{\varphi}^p\right\|_{E}=1$, then 
\begin{align*}
\begin{array}{lll}
\lambda_p(d) &=\dfrac{d}{2}
\displaystyle\int_{0}^{L}\displaystyle\int_{0}^{L}J_1(x-y)\left(\varphi_1(x)-\varphi_1^p(y)\right)^2dxdy + \dfrac{d}{2}
\displaystyle\int_{0}^{L}\displaystyle\int_{0}^{L}J_2(x-y)\left(\varphi_2^p(x)-\varphi_2^p(y)\right)^2dxdy\\& -\displaystyle\int_{0}^{L}\left[\left(a_{11}-d +dk_1(x)\right)\left(\varphi_1^p\right)^2(x)+(a_{12}+a_{21})\varphi_1(x)\varphi_2(x)+\left(a_{22}-d+ dk_2(x)\right)\left(\varphi_2^p\right)^2(x)\right]dx.
\end{array}
\end{align*}
Obviously $\varphi_1^p, \varphi_2^p$ are not a constant.  Assume that $d_1>d_2$, with $d_1, d_2$ are positive constants, then following the variational characterization of $\lambda_p(d)$, we have
\begin{align*}
\lambda_p\left(d_1\right)&\geq \dfrac{d_2}{2}
\displaystyle\int_{0}^{L}\displaystyle\int_{0}^{L}J_1(x-y)\left(\varphi_1^p(x)-\varphi_1^p(y)\right)^2dxdy + \dfrac{d_2}{2}
\displaystyle\int_{0}^{L}\displaystyle\int_{0}^{L}J_2(x-y)\left(\varphi_2^p(x)-\varphi_2^p(y)\right)^2dxdy\\& -\displaystyle\int_{0}^{L}\left[\left(a_{11}-d_2 +d_Sk_1(x)\right)\left(\varphi_1^p\right)^2(x)+(a_{12}+a_{21})\varphi_1^p(x)\varphi_2^p(x)+\left(a_{22}-d_2+ d_2k_2(x)\right)\left(\varphi_2^p\right)^2(x)\right]dx\\& \geq \lambda_p\left(d_2\right).
\end{align*}
This completes the proof \eqref{1}.

\noindent (2) First of all, in order to prove \eqref{2}, we recall the results in \cite[Theorem 2.2]{SX}, that is
\begin{align}\label{3.17}
\displaystyle\int_{\O}\displaystyle\int_{\O}J(x-y)\left(u(x)-u(y)\right)^2dxdy\geq 2C \displaystyle\int_{\O} u^2(x)dx,\; C>0.
\end{align}
Let $\pmb{\varphi}^p=\left(\varphi^p_1, \varphi^p_2\right)$ be the corresponding eigenfunction, which is normalized $\left\|\pmb{\varphi}^p\right\|_{E}=1$, associated to the eigenvalue $\lambda_p(d)$. We deduce
\begin{align*}
\begin{array}{lll}
\lambda_p(d) &=\dfrac{d}{2}
\displaystyle\int_{0}^{L}\displaystyle\int_{0}^{L}J_1(x-y)\left(\varphi^p_1(x)-\varphi^p_1(y)\right)^2dxdy + \dfrac{d}{2}
\displaystyle\int_{0}^{L}\displaystyle\int_{0}^{L}J_2(x-y)\left(\varphi^p_2(x)-\varphi^p_2(y)\right)^2dxdy\\& -\displaystyle\int_{0}^{L}\left[\left(a_{11}-d +dk_1(x)\right)(\varphi^p_1)^2(x)+(a_{12}+a_{21})\varphi^p_1(x)\varphi^p_2(x)+\left(a_{22}-d+ dk_2(x)\right)(\varphi^p_2)^2(x)\right]dx.
\end{array}
\end{align*}
By \eqref{3.17}, we have 
\begin{align*}
\begin{array}{lll}
\dfrac{d}{2}
\displaystyle\int_{0}^{L}\displaystyle\int_{0}^{L}J_1(x-y)\left(\varphi^p_1(x)-\varphi^p_1(y)\right)^2dxdy\geq dC_1\displaystyle\int_{0}^{L} (\varphi^p_1)^2(x)dx,\\
\dfrac{d}{2}
\displaystyle\int_{0}^{L}\displaystyle\int_{0}^{L}J_2(x-y)\left(\varphi^p_2(x)-\varphi^p_2(y)\right)^2dxdy\geq dC_2\displaystyle\int_{0}^{L} (\varphi^p_2)^2(x)dx,
\end{array}
\end{align*}
where $C_1, C_2>0$. Denote $C = \min\left\{C_1, C_2\right\}>0$, $M=\dfrac{\left|a_{12}+a_{21}\right|-4L(a_{11}+a_{22})}{2}$,  we derive
\begin{align*}
\begin{array}{lll}
\lambda_p(d) &=\dfrac{d}{2}
\displaystyle\int_{0}^{L}\displaystyle\int_{0}^{L}J_1(x-y)\left(\varphi_1(x)-\varphi_1(y)\right)^2dxdy + \dfrac{d}{2}
\displaystyle\int_{0}^{L}\displaystyle\int_{0}^{L}J_2(x-y)\left(\varphi_2(x)-\varphi_2(y)\right)^2dxdy\\& -\displaystyle\int_{0}^{L}\left[\left(a_{11}-d +dk_1(x)\right)\varphi_1^2(x)+(a_{12}+a_{21})\varphi_1(x)\varphi_2(x)+\left(a_{22}-d+ dk_2(x)\right)\varphi_2^2(x)\right]dx\\&\geq dC\left(\displaystyle\int_{0}^{L} \varphi_1^2(x)dx+ \displaystyle\int_{0}^{L} \varphi_2^2(x)dx\right)-\displaystyle\int_{0}^{L}\dfrac{\left|a_{12}+a_{21}\right|}{2}\left(\varphi_1^2(x)+\varphi_2^2(x)\right)dx-\displaystyle\int_{0}^{L}\left(a_{11}+a_{22}\right)dx\\&
\geq dC -M,
\end{array}
\end{align*}
where we use the fact that $-d +dk_i(x)\leq 0,\; i=1,2$. Therefore, $\lim\limits_{d\to\infty}\lambda_p\left(d\right)= +\infty$.

\noindent (3) The later statement can be
proven by similar arguments as in the proof \eqref{1a} of Proposition \ref{pro.3.4}, we deduce that 
\begin{align}\label{3.23}
\lambda_p(d)= -\sup\limits_{\left\|\pmb{\varphi}\right\|_{E} =1}\left\langle\mathbf{M}\pmb{\varphi}, \pmb{\varphi}\right\rangle \geq -\dfrac{a_{11}+a_{22}+\sqrt{\left(a_{11}-a_{22}\right)^2+ 4a_{12}a_{21}}}{2}.
\end{align}
Setting
\begin{align*}
 \alpha= a_{11}-a_{22}+ \sqrt{\left(a_{11}-a_{22}\right)^2+\left(a_{12}+a_{21}\right)^2},\,\,\beta = a_{12}+a_{21},
 \end{align*}
 and let  $\pmb{\varphi} = \left(\varphi_1, \varphi_2\right) = \left(\dfrac{\alpha}{\sqrt{L\left(\alpha^2+\beta^2\right)}}, \dfrac{\beta}{\sqrt{L\left(\alpha^2+\beta^2\right)}}\right)$ be the test  function in the
variational characterization of $\lambda_p\left(d\right)$. We infer that for all large $d> 0$,
\begin{align*}
\lambda_p(d)&\leq-\dfrac{d\displaystyle\int_{0}^{L}\displaystyle\int_{0}^{L}J_1(x-y)\alpha^2 dy+ d\displaystyle\int_{0}^{L}\displaystyle\int_{0}^{L}J_2(x-y)\beta^2 dy}{L\left(\alpha^2+\beta^2\right)}-\dfrac{\left(a_{11}-d\right)\alpha^2 +\left(a_{12}+a_{21}\right)\alpha\beta +\left(a_{22}-d\right)\beta^2}{\left(\alpha^2+\beta^2\right)}\\&\leq -\dfrac{a_{11}+a_{22}-2d+\sqrt{\left(a_{11}-a_{22}\right)^2+4a_{12}a_{21}}}{2}.
\end{align*}
Let $d\to 0$, we deduce that $\lim\limits_{d\to 0}\lambda_p(d)\leq -\dfrac{a_{11}+a_{22}+\sqrt{\left(a_{11}-a_{22}\right)^2+4a_{12}a_{21}}}{2}$. This together with \eqref{3.23},  we get the desired limits, which implies \eqref{1c}.
\end{proof}
\subsection{A fixed boundary problem}\hspace{8cm}

For $L_1<0$ and $L_2>0$, set $\Omega_L=(0,+\infty)\times (L_1,L_2)$, we consider the following fixed boundary problem 
\begin{align}\label{fixed}
\begin{cases}
U_t=d_1\displaystyle\int_{L_1}^{L_2} J_1(x-y)U(t,y)dy-d_1U(t,x)-aU(t,x)+H(V(t,x)),  & (t,x)\in \Omega_L,\\
V_t=d_2\displaystyle\int_{L_1}^{L_2} J_2(x-y)V(t,y)dy-d_2V(t,x)-bV(t,x)+G(U(t,x)), \qquad & (t,x)\in \Omega_L,\\
U(0,x)=U_0(x),V(0,x)=V_0(x), &x\in [L_1, L_2],
\end{cases}
\end{align}
where $U_0,V_0\in C([L_1,L_2])\setminus \{0\}$, and denote
\begin{align}
\Gamma_1=\left\{\begin{array}{lll}
\max\left\{\|U_0\|,\dfrac{b}{G'(0)}\|V_0\|\right\},&\text{if}\; H'(0)G'(0)\leq ab,\\
\max\left\{\|U_0\|,\dfrac{K_1}{K_2}\|V_0\|,K_1\right\},&\text{if}\; ab<H'(0)G'(0),
\end{array}\right.,\;
\Gamma_2=\begin{cases}
\dfrac{G'(0)}{b}\Gamma_1,&\text{if}\; H'(0)G'(0)\leq ab,\\
\dfrac{K_2}{K_1}\Gamma_1,&\text{if}\; ab<H'(0)G'(0) .
\end{cases}
\end{align}
It is well-known that $\eqref{fixed}$ has a unique positive solution which is defined for all $t>0$. The corresponding steady state problem of $\eqref{fixed}$ is
\begin{align}\label{steady}
\begin{cases}
d_1\displaystyle\int_{L_1}^{L_2} J_1(x-y)U(y)dy-d_1U(x)-aU(x)+H(V(x))=0,  & x\in \left[L_1, L_2\right],\\
d_2\displaystyle\int_{L_1}^{L_2} J_2(x-y)V(y)dy-d_2V(x)-bV(x)+G(U(x))=0, & x\in \left[L_1, L_2\right].
\end{cases}
\end{align}
\begin{definition}
	A function pair $(\overline{\psi},\overline{\varphi})\in C\left(\left[-L_1, L_2\right]\right)\times C\left(\left[L_1, L_2\right]\right)$ is said to be an supersolution of $\eqref{steady}$ if
	\begin{align*}
	\begin{cases}
	d_1\displaystyle\int_{L_1}^{L_2} J_1(x-y)\overline{\psi}(y)dy-d_1\overline{\psi}(x)-a\overline{\psi}(x)+H((\overline{\varphi}(x))\leq 0, & x\in \left[L_1, L_2\right],\\
	d_2\displaystyle\int_{L_1}^{L_2}J_2(x-y)\overline{\varphi}(y)dy-d_2\overline{\varphi}(x)-b\overline{\varphi}(x)+G(\overline{\psi}(x))\leq 0, & x\in \left[L_1, L_2\right].
	\end{cases}
	\end{align*}
\end{definition}
We define $\lambda_p\left(L_1, L_2, d_1, d_2\right)$ is the principal eigenvalue of problem
	\begin{align}\label{eigenvalue}
\begin{cases}
\dfrac{d_1}{H'(0)}\displaystyle\int_{L_1}^{L_2} J_1(x-y)\psi(y)dy-\dfrac{d_1}{H'(0)}\psi(x)-\dfrac{a}{H'(0)}\psi(x)+\varphi(x)+\lambda\psi(x)= 0, & x\in \left[L_1, L_2\right],\\
\dfrac{d_2}{G'(0)}\displaystyle\int_{L_1}^{L_2}J_2(x-y)\varphi(y)dy-\dfrac{d_2}{G'(0)}\varphi(x)+\psi(x)-\dfrac{b}{G'(0)}\varphi(x)+\lambda\varphi(x)= 0,  & x\in \left[L_1, L_2\right].
\end{cases}
\end{align}
By reversing the above inequalities, we can define a subsolution.
\begin{proposition}\label{nd}
	Assume \textbf{(J)} holds and $H,G$ satisfy \eqref{conditionH_G}. Then the problem $\eqref{steady}$ has a unique positive solution $(\mathcal{U},\mathcal{V})\in C\left(\left[L_1, L_2\right]\right)\times C\left(\left[L_1, L_2\right]\right)$ satisfying $0<\mathcal{U}\leq \Gamma_1,0<\mathcal{V}\leq \Gamma_2$ if $\lambda_p(L)<0$, and $\left(0, 0\right)$ is the only nonnegative steady-state if $\lambda_p(L)\geq 0$.
\end{proposition}
\begin{proof}[\bf Proof]
Assume that $\lambda_p(L)<0$. By Proposition \eqref{pro.3.6}, we can obtain that $\lim\limits_{L\rightarrow +\infty}\lambda_p(L)<0$, then $H'(0)G'(0)>ab$. This together with assumptions \eqref{conditionH_G}, there exists $\left(K_1, K_2\right)>0$ such that
	\begin{align*}
	\left\{\begin{array}{rrr}
	H\left(K_2\right)=aK_1,\\
	G\left(K_1\right)=bK_2.
	\end{array}\right.
	\end{align*} 
	Let $\left(\psi,\varphi\right)$ is a positive eigenfunction pair corresponding to $\lambda_p(L)$. Since $H(z)/z\;\text{and}\;G(z)/z$ are non-increasing and $\lim\limits_{z\rightarrow 0}\dfrac{H(z)}{z}=H'(0),\; \lim\limits_{z\rightarrow 0}\dfrac{G(z)}{z}=G'(0)$, we can choose $\varepsilon$ small enough such that for all $x\in \left[L_1,\;L_2\right]$,
	\begin{align*}
	\dfrac{H\left(\varepsilon\varphi\right)}{\varepsilon\varphi}-H'(0)\geq \dfrac{H(\varepsilon\max_{\left[L_1, L_2\right]}\varphi)}{\varepsilon\max_{\left[L_1, L_2\right]}\varphi}-H'(0) \geq \lambda_p(L)\min_{\left[L_1, L_2\right]}\dfrac{\psi}{\varphi}\geq \lambda_p(L)\dfrac{\psi}{\varphi},\\
\dfrac{G(\varepsilon\psi)}{\varepsilon\psi}-G'(0)\geq	\dfrac{G(\varepsilon\max_{\left[L_1, L_2\right]}\psi)}{\varepsilon\max_{[L_1, L_2]}\psi}-G'(0)\geq \lambda_p(L)\min_{[L_1, L_2]}\dfrac{\varphi}{\psi}\geq \lambda_p(L)\dfrac{\varphi}{\psi}.
	\end{align*}
	Then it is easy to check that $(K_1, K_2)$ and $\left(U^0,V^0\right)=(\varepsilon\psi,\varepsilon\varphi)$ are a pair of supper and subsolutions of $\eqref{steady}$ and $(U^0, V^0)\leq (K_1, K_2)$ for small enough $\varepsilon$. Setting $K=\max\{K_1, K_2\}$. Then we can choose $\alpha$ small enough such that the functions $H(z)-\alpha H'(0)z$ and $G(z)-\alpha G'(0)z$ are non-decreasing in $(0,K)$. Next, we show that the following problem has unique solution $(U^1,V^1)\in C([L_1,L_2])\times C([L_1, L_2])$ for $k>0$ big enough.
	\begin{align}\label{axc}
	\begin{cases}
	\dfrac{d_1}{a+d_1+k}\displaystyle\int_{L_1}^{L_2} J_1(x-y)\phi_1(y)dy+\dfrac{\alpha H'(0)}{a+d_1+k}\phi_2+\dfrac{H(V^0)+kU^0-\alpha H'(0)V^0}{a+d_1+k}=\phi_1,\text{ in } \left[L_1, L_2\right],\\
	\dfrac{d_2}{b+d_2+k}\displaystyle\int_{L_1}^{L_2} J_2(x-y)\phi_2(y)dy+\dfrac{\alpha G'(0)}{b+d_2+k}\phi_1+\dfrac{G(U^0)+kV^0-\alpha G'(0)U^0}{b+d_2+k}=\phi_2,\text{ in } \left[L_1, L_2\right].
	\end{cases}
	\end{align}
	We first note that $\mathbf{C}= C([L_1,L_2])\times C([L_1,L_2])$ is a complete metric space with the metric
	\begin{align*}
	d(\phi,\omega)=\|\phi_1-\omega_1\|_{C([L_1,L_2])}+\|\phi_2-\omega_2\|_{C\left(\left[L_1, L_2\right]\right)}.
	\end{align*} 
	Let us define
	\begin{align*}
&F_1(\phi_1,\phi_2)=	\dfrac{d_1}{a+d_1+k}\displaystyle\int_{L_1}^{L_2} J_1(x-y)\phi_1(y)dy+\dfrac{\alpha H'(0)}{a+d_1+k}\phi_2(x)+\dfrac{H(\varphi^0(x))+k\psi^0(x)-\alpha H'(0)\varphi^0(x)}{a+d_1+k},\\
& F_2(\phi_1,\phi_2)=	\dfrac{d_2}{b+d_2+k}\displaystyle\int_{L_1}^{L_2} J_2(x-y)\phi_2(y)dy+\dfrac{\alpha G'(0)}{b+d_2+k}\phi_1(x)+\dfrac{G(\psi^0(x))+k\varphi^0(x)-\alpha G'(0)\psi^0(x)}{b+d_2+k}.
	\end{align*}
	and the mapping $F:\mathbf{C}\rightarrow \mathbf{C}$ given by
	\begin{align*}
	F(\phi)=\left(F_1(\phi), F_2(\phi)\right).
	\end{align*}
	We also see that the problem \eqref{axc} becomes $F(\phi)=\phi$. Next, we prove that $F$ has a unique fixed point in $\mathbf{C}$ by the contraction mapping theorem; namely we show that there exist $k$ such that $F$ is a contraction mapping on $\mathbf{C}$. For $\pmb{\phi}=(\phi_1,\phi_2)$ and $\pmb{\omega}=(\omega_1,\omega_2)$ in $\mathbf{C}$, we have
	\begin{align*}
	&|F_1(\phi)-F_1(\omega)|\\
	\leq &\dfrac{d_1}{a+d_1+k}\displaystyle\int_{L_1}^{L_2} J_1(x-y)|\phi_1(y)-\omega_1(y)|dy+\dfrac{\alpha H'(0)}{a+d_1+k}|\phi_2(x)-\omega_2(x)|\\
	\leq& \dfrac{d_1}{a+d_1+k}\|\phi_1-\omega_1\|_{C([L_1,L_2])}\displaystyle\int_{L_1}^{L_2} J_1(x-y)dy+\dfrac{\alpha H'(0)}{a+d_1+k}\|\phi_2-\omega_2\|_{C\left(\left[L_1, L_2\right]\right)}\\
	\leq&\dfrac{d_1}{a+d_1+k}\|\phi_1-\omega_1\|_{C([L_1,L_2])}+\dfrac{H'(0)}{a+d_1+k}\|\phi_2-\omega_2\|_{C\left(\left[L_1, L_2\right]\right)}.
	\end{align*}
Therefore,
	$\left\|F_1(\phi)-F_1(\omega)\right\|_{C\left(\left[L_1, L_2\right]\right)}	\leq \dfrac{d_1}{a+d_1+k}\left\|\phi_1-\omega_1\right\|_{C\left(\left[L_1, L_2\right]\right)}+\dfrac{\alpha H'(0)}{b+d_2+k}\left\|\phi_2-\omega_2\right\|_{C\left(\left[L_1, L_2\right]\right)}.$

Similarly, we obtain that
\begin{align*}
	\left\|F_2(\phi)-F_2(\omega)\right\|_{C\left(\left[L_1, L_2\right]\right)}	\leq&\dfrac{d_2}{b+d_2+k}\left\|\phi_2-\omega_2\right\|_{C\left(\left[L_1, L_2\right]\right)}+\dfrac{\alpha G'(0)}{b+d_2+k}\left\|\phi_1-\omega_1\right\|_{C\left(\left[L_1, L_2\right]\right)}.
\end{align*}
	Therefore,
	\begin{align*}
	&d(F(\phi),F(\omega))\\=&\left\|F_1(\phi)-F_1(\omega)\right\|_{C\left(\left[L_1, L_2\right]\right)}+\left\|F_2(\phi)-F_2(\omega)\right\|_{C\left(\left[L_1, L_2\right]\right)}\\
	\leq&\left(\dfrac{d_1}{a+d_1+k}+\dfrac{\alpha H'(0)}{b+d_2+k}\right)\left\|\phi_1-\omega_1\right\|_{C\left(\left[L_1, L_2\right]\right)}+\left(\dfrac{d_2}{b+d_2}+\dfrac{\alpha G'(0)}{a+d_1+k}\right)\left\|\phi_2-\omega_2\right\|_{C\left(\left[L_1, L_2\right]\right)}\\
	\leq &\max\left\{\dfrac{d_1}{a+d_1+k}+\dfrac{\alpha H'(0)}{b+d_2}, \dfrac{d_2}{b+d_2+k}+\dfrac{\alpha G'(0)}{a+d_1+k}\right\}\left(\|\phi_1-\omega_1\|_{C\left(\left[L_1, L_2\right]\right)}+\|\phi_2-\omega_2\|_{C\left(\left[L_1, L_2\right]\right)}\right)\\
	\leq& \beta d(\phi,\omega).
	\end{align*}
	Now, by choosing  $k$ big enough satisfying
	\begin{align*}
	\beta=\max\left\{\dfrac{d_1}{a+d_1+k}+\dfrac{\alpha H'(0)}{b+d_2+k}, \dfrac{d_2}{b+d_2+k}+\dfrac{\alpha G'(0)}{a+d_1+k}\right\} <1.
	\end{align*}
	For such $k$ we may now apply the contraction mapping theorem to conclude that $F$ has a unique fixed point $(U^1,V^1)$ in $\mathbf{C}$. Thus, the problem \eqref{axc} has a unique $(U^1,V^1)$ in $\mathbf{C}$. Moreover, we see that
	\begin{align*}
	&d_1\displaystyle\int_{L_1}^{L_2} J_1(x-y)\left[U^1(y)-U^0(y)\right]dy-\left(d_1+a+k\right)\left[U^1(x)-U^0(x)\right]+\alpha H'(0)\left[V^1(x)-V^0(x)\right]\\
	=&-H\left(V^0(x)\right)-kU^0(x)+\alpha H'(0)V^0(x)-d_1\displaystyle\int_{L_1}^{L_2}J_1(x-y)U^0(y)dy+\left(d_1+a+k\right)U^0(x)-\alpha H'(0)V^0(x)\\
	=&-\left(d_1\displaystyle\int_{L_1}^{L_2}J_1(x-y)U^0(y)dy-\left(d_1+a\right)U^0(x)+H\left(V^0(x)\right)\right).
	\end{align*}
	Since $\left(U^0, V^0\right)$ is a subsolutions of $\eqref{steady}$, we arrive at
	\begin{align}\label{ul1}
	d_1\displaystyle\int_{L_1}^{L_2} J_1(x-y)\left[U^1(y)-U^0(y)\right]dy-\left(d_1+a+k\right)\left[U^1(x)-U^0(x)\right]+\alpha H'(0)\left[V^1(x)-V^0(x)\right]\leq 0.
	\end{align}
	Similarly,
	\begin{align}\label{ul2}
	d_2\displaystyle\int_{L_1}^{L_2} J_2(x-y)\left[V^1(y)-V^0(y)\right]dy-\left(d_2+b+k\right)\left[V^1(x)-V^0(x)\right]+\alpha G'(0)\left[U^1(y)-U^0(y)\right]\leq 0.
	\end{align}
Moreover, we note that 
\begin{align*}
&d_1\displaystyle\int_{L_1}^{L_2} J_1(x-y)\left[U^1(y)-K_1\right]dy-\left(d_1+a+k\right)\left[U^1(x)-K_1\right]+\alpha H'(0)\left[V^1(x)-K_2\right]\\
=&-H\left(V^0(x)\right)-kU^0(x)+\alpha H'(0)V^0(x)-d_1\displaystyle\int_{L_1}^{L_2}J_1(x-y)K_1dy+\left(d_1+a+k\right)K_1-\alpha H'(0)K_2\\
\geq &\left[H\left(K_2\right)-\alpha H'(0)K_2\right]-\left[H\left(V^0(x)\right)-\alpha H'(0)V^0(x)\right]+k\left[K_1-\psi^0(x)\right].
\end{align*}
Since $0<U^0(x)\leq K_1, 0<V^0(x)\leq K_2$ and  $H(z)-\alpha H'(0)z$ is non-decreasing in $(0,K)$, we deduce that
\begin{align}\label{ul3}
d_1\displaystyle\int_{L_1}^{L_2} J_1(x-y)\left[U^1(y)-K_1\right]dy-\left(d_1+a+k\right)\left[U^1(x)-K_1\right]+\alpha H'(0)\left[V^1(x)-K_2\right]\geq 0.
\end{align}
Similarly, we also obtain
\begin{align}\label{ul4}
d_2\displaystyle\int_{L_1}^{L_2} J_2(x-y)\left[V^1(y)-K_2\right]dy-\left(d_2+b+k\right)\left[V^1(x)-K_2\right]+\alpha G'(0)\left[U^1(x)-K_1\right]\geq 0.
\end{align}
	Next, we denote $\lambda_p^{k,\alpha}$ is the principal eigenvalue of problem
\begin{align}\label{eg1}
\begin{cases}
\dfrac{d_1}{H'(0)}\displaystyle\int_{L_1}^{L_2} J_1(x-y)\psi(y)dy-\dfrac{d_1+a+k}{H'(0)}\psi(x)+\alpha\varphi(x)+\lambda\psi(x)= 0,  & x\in \left[L_1, L_2\right],\\
\dfrac{d_2}{G'(0)}\displaystyle\int_{L_1}^{L_2}J_2(x-y)\varphi(y)dy-\dfrac{d_2+b+k}{G'(0)}\varphi(x)+\alpha\psi(x)+\lambda\varphi(x)= 0, & x\in \left[L_1, L_2\right].
\end{cases}
\end{align}
By Proposition \ref{proposition_3.4} with
\begin{align*}
\mathbf{D} = \begin{pmatrix}
\dfrac{d_1}{H'(0)}&0\\
0 & \dfrac{d_2}{G'(0)}
\end{pmatrix},\,
\mathbf{A} = \begin{pmatrix}
-\dfrac{a}{H'(0)}&1\\
1 & -\dfrac{b}{G'(0)}
\end{pmatrix},\,
\mathbf{B} = \begin{pmatrix}
-\dfrac{a+k}{H'(0)}&\alpha\\
\alpha & -\dfrac{b+k}{G'(0)}
\end{pmatrix},
\end{align*}
we deduce that for $k$ big enough
\begin{align*}
\lambda^{k,\alpha}_p=\lambda_p(\mathbf{B})\geq \lambda_p(\mathbf{A})+\dfrac{k}{\max\{H'(0),G'(0)\}}=\lambda_p+\dfrac{k}{\max\{H'(0),G'(0)\}}>0.
\end{align*}
Now, we apply Proposition \ref{mpnt} and Remark \ref{mpnt1}  to \eqref{ul1},\eqref{ul2},\eqref{ul3},\eqref{ul4} and get 
\begin{align*}
\left(U^0, V^0\right)\leq \left(U^1, V^1\right)\leq \left(K_1, K_2\right).
\end{align*}
Now, let $\left(U^2, V^2\right)$ be the solution of \eqref{axc} with $\left(U^1, V^1\right)$ instead of $\left(U^0, V^0\right)$. Using the  comparison principle and  a increasing monotonicity of $H(z)-\alpha H'(0)z$ and $H(z)-\alpha G'(0)z$, we obtain that 
\begin{align*}
\left(U^0, V^0\right)\leq \left(U^1, V^1\right)\leq\left(U^2, V^2\right) \leq \left(K_1, K_2\right).
\end{align*}
 By induction, we can construct an
increasing sequence of function $\left(U^n, V^n\right)$ satisfying
\begin{align}\label{ul6}
\left(U^0, V^0\right)\leq \left(U^n, V^n\right)\leq \left(K_1, K_2\right),
\end{align}
 and for any $x\in \left[L_1, L_2\right]$,
\begin{align}\label{ul5}
\begin{cases}
d_1\displaystyle\int_{L_1}^{L_2} J_1(x-y)U^{n+1}(y)dy-\left(d_1+a+k\right)U^{n+1}+\alpha H'(0)V^{n+1}+H(V^n)+kU^{n}-\alpha H'(0)V^n=0,\\
d_2\displaystyle\int_{L_1}^{L_2} J_2(x-y)V^{n+1}(y)dy-\left(d_2+b+k\right)V^{n+1}+\alpha G'(0)U^{n+1}+G\left(V^n\right)+kV^{n}-\alpha G'(0)U^n=0.\\
\end{cases}
\end{align}
Since the sequence $\left(U^n, V^n\right)$ is increasing and bounded, $\left(\CU(x),\CV(x)\right)=\lim\limits_{n\rightarrow\infty}\left(U^n(x), V^n(x)\right)$ is well defined. Passing to the limit in \eqref{ul6}, \eqref{ul5} and using the dominated convergence theorem, we deduce that $\left(\CU,\CV\right)$ is a solution of \eqref{steady} and $(0,0)<\left(\CU,\CV\right)\leq \left(K_1, K_2\right)$. Moreover, by Dini's theorem, the convergence is uniformly in $\left[L_1, L_2\right]$, which implies that $\left(\CU,\CV\right)\in C\left(\left[L_1, L_2\right]\right)\times \left(\left[L_1, L_2\right]\right)$. \\
Let us show that when $\lambda_p(L)\geq 0$ then there exists no non-trivial negative solution to $\eqref{steady}$. Assume by contradiction that $\lambda_p(L)\geq 0$ and $\left(\psi,\varphi\right)$ is a positive steady state solution of $\eqref{steady}$. Then $\left(\psi,\varphi\right)$ satisfies
\begin{align*}
\begin{cases}
d_1\displaystyle\int_{L_1}^{L_2} J_1(x-y)U(y)dy-d_1\psi(x)-a\psi(x)+H'(0)\varphi(x)< 0, & x\in \left[L_1, L_2\right],\\
d_2\displaystyle\int_{L_1}^{L_2} J_2(x-y)\varphi(y)dy-d_2\varphi(x)-b\psi(x)+G'(0)\psi(x)< 0, & x\in \left[L_1, L_2\right].
\end{cases}
\end{align*}
We now apply Proposition \ref{pro.3.4} with $\widetilde{\lambda}=0$ to induce that $\lambda_p(L)<0$. This contradiction proves the non-existence result.\\

Next, we show that when a solution of $\eqref{steady}$ exists then it is unique. We first show that if $\left(u, v\right)$ is non-negative and bounded solution of $\eqref{steady}$, then  $\varsigma =\min\limits_{\left[L_1, L_2\right]}u(x),\,\omega=\min\limits_{\left[L_1, L_2\right]}v(x)>0$. Indeed, if there exists some $x_0\in \left[L_1, L_2\right]$ such that $\min_{\left[L_1, L_2\right]}u(x)=u\left(x_0\right)=0$, then from the first equation of $\eqref{steady}$, we have
\begin{align*}
d_1\displaystyle\int_{L_1}^{L_2} J_1(x_0-y)u(y)dy=\left(a+d_1\right)u(x_0)-H\left(v\left(x_0\right)\right)=-H\left(v\left(x_0\right)\right).
\end{align*}
Since $J_1, H, u, v$ are non-negative quantities, and therefore, we deduce
\begin{align*}
\displaystyle\int_{L_1}^{L_2}J_1\left(x_0-y\right)u(y)dy=0\text{ and } H\left(v\left(x_0\right)\right)=0.
\end{align*} 
This leads to $v(x_0)=0$ and $u(x)\equiv0$ in $\left[L_1, L_2\right]$, due to $J_1(0)>0$ and $u$ is continuous. From the second equation of $\eqref{steady}$, the similarity argument induces that $v\equiv 0$ in $\left[L_1, L_2\right]$. This contradiction implies that $\varsigma , \omega>0$. Now, let $\left(U_1, V_1\right)$ and $\left(U_2, V_2\right)$ are two positive steady state solutions. The above argument yields that $\left(U_1, V_1\right)$ and $\left(U_2,V_2\right)$ are bounded and stricly positive, then the following quantity is well defined
\begin{align*}
\gamma^*:=\inf\left\{\gamma>0\mid \gamma\left(U_1,V_1\right)\geq \left(U_2, V_2\right)\right\}.
\end{align*}
We claim that $\gamma^*\leq 1$.  Assume by contradiction that $\gamma^*>1$. It is nice to be followed by the equation $\eqref{steady}$, replace $\left(U, V\right)$ by $\left(\gamma^*U_1,\gamma^*V_1\right)$ we follow that
\begin{align}\label{gamma}
&d_1\displaystyle\int_{L_1}^{L_2}J_1(x-y)\gamma^*U_1(y)dy-d_1\gamma^*U_1(x)-a\gamma^*U_1(x)+H\left(\gamma^*V_1(x)\right)\\
&=H\left(\gamma^*V_1(x)\right)-\gamma^*H\left(V_1(x)\right)
=\gamma^*V_1(x)\left(\dfrac{H\left(\gamma^*V_1(x)\right)}{\gamma^*V_1(x)}-\dfrac{H\left(V_1(x)\right)}{V_1(x)}\right)
\leq 0\nonumber.
\end{align}
The last inequality can obtain by using the decreasing of $H(z)/z$ function and $\gamma^*>1$. Similarily, we can show that
$d_2\displaystyle\int_{L_1}^{L_2}J_2(x-y)\gamma^*V_1(y)dy-d_2\gamma^*V_1(x)-b\gamma^*V_1(x)+G\left(\gamma^*U_1(x)\right)\leq 0.$
By definition of $\gamma^*$, there exists $x_0\in \left[L_1, L_2\right]$ such that $\gamma^*\left(U_1\left(x_0\right), V_1\left(x_0\right)\right)=\left(U_2\left(x_0\right), V_2\left(x_0\right)\right)$. It  follows from the equation $\eqref{steady}$, we can easily check that
\begin{align*}
&d_1\displaystyle\int_{L_1}^{L_2} J_1\left(x_0-y\right)\gamma^*U_1(y)dy-d_1\gamma^*U_1\left(x_0\right)-a\gamma^*U_1\left(x_0\right)+H\left(\gamma^*V_1\left(x_0\right)\right)\\
=& d_1\displaystyle\int_{L_1}^{L_2} J_1(x-y)\gamma^*U_2(y)dy-d_1U_2\left(x_0\right)-aU_2\left(x_0\right)+H\left(V_2\left(x_0\right)\right)
=d_1\displaystyle\int_{L_1}^{L_2} J_1\left(x_0-y\right)\left[\gamma^*U_2(y)-U_2(y)\right]dy
\geq 0.
\end{align*}
This implies that 
$\displaystyle\int_{L_1}^{L_2} J_1\left(x_0-y\right)\left[\gamma^*U_2(y)-U_2(y)\right]dy=0.$
Similarity to the above argument, we deduce that $\gamma^*U_1=U_2$. By the same argument, we also have $\gamma^*V_1=V_2$. By $\eqref{gamma}$, we see that
\begin{align*}
0=&d_1\displaystyle\int_{L_1}^{L_2} J_1(x-y)U_2(y)dy-d_1U_1(x)-aU_1(x)+H\left(V_2(x)\right)\\
=&d_1\displaystyle\int_{L_1}^{L_2} J_1(x-y)\gamma^*U_1(y)dy-d_1\gamma^*U_1(x)-a\gamma^*U_1(x)+H\left(\gamma^*V_1(x)\right)
=\gamma^*V_1(x)\left(\dfrac{H\left(\gamma^*V_1(x)\right)}{\gamma^*V_1(x)}-\dfrac{H\left(V_1(x)\right)}{V_1(x)}\right),
\end{align*}
which implies that
$\dfrac{H\left(\gamma^*V_1(x)\right)}{\gamma^*V_1(x)}=\dfrac{H\left(V_1(x)\right)}{V_1(x)},$
and then $\gamma^*V_1(x)=V_1(x),\;\; \forall x\in \left[L_1, L_2\right]$. This is impossible due to $\gamma^*>1$. Hence, $\gamma^*\leq 1$ and as a consequence $\left(U_1, V_1\right)\geq \left(U_2, V_2\right)$. Observe that the role of $\left(U_1, V_1\right)$ and $\left(U_2, V_2\right)$ can be interchanged in the above argumention. Therefore, we also have $\left(U_1, V_1\right)\leq \left(U_2, V_2\right)$, which shows that the uniqueness of the solution.
\end{proof}
\begin{remark}\label{re10}
	Let  $\left(\underline{U}, \underline{V}\right)$ is a subsolutions of \eqref{steady} and assume $\lambda_p(L)<0$. Since  
	$H\left(G(z)/z\right)-az<0 \text{ for all } z>K_1,$	we can choose $\alpha_1>\max\left\{K_1,\max_{\left[L_1, L_2\right]}\underline{U}(x)\right\}$ big enough such that 
	
	$$G\left(\alpha_1\right)\geq b\max_{\left[L_1, L_2\right]}\underline{V}(x).$$
 Then we set $\alpha_2:=G\left(\alpha_1\right)/b$. It is easy to check that $\left(\alpha_1,\alpha_2\right)$ is supersolution of \eqref{steady} and 	$\left(\underline{U},\underline{V}\right)\leq \left(\alpha_1,\alpha_2\right)$. Moreover, we can observe that if $\left(U,V\right)$ is a unique positive solution of \eqref{steady} then
	$\left(\underline{U},\underline{V}\right)\leq \left(U,V\right).$
\end{remark}
By Proposition \ref{pro.3.6}, when $\CR_0>1$, there exists $L_0>0$ such that $\lambda_p(L)<0$ for all $L\leq L_0$. Then, we apply Proposition \ref{nd} to obtain that \eqref{steady} has a unique positive solution $\left(\CU,\CV\right)$ for $L_2-L_1\geq L_0$. To stress its dependence on $L_1,L_2$, we denote it by $\left(\CU_{\left[L_1,L_2\right]},\CV_{\left[L_1,L_2\right]}\right)$.
\begin{proposition}
	Assume \textbf{(J)} holds and $\CR_0>1$. Then
	\begin{align*}
	\lim_{-L_1,L_2\rightarrow \infty}(\CU_{\left[L_1, L_2\right]}(x),\CV_{\left[L_1, L_2\right]}(x))=\left(K_1, K_2\right)\text{ locally uniformly in }\mathbb{R},
	\end{align*}
	where $\left(K_1, K_2\right)$ is defined by \eqref{K}.
\end{proposition}
\begin{proof}[\bf Proof]
	We first show that 
	\begin{align*}
	\left(\CU_{\left[L_1, L_2\right]}(x),\CV_{\left[L_1, L_2\right]}(x)\right)\leq \left(\CU_{\left[L'_1, L'_2\right]}(x),\CV_{\left[L'_1, L'_2\right]}(x)\right) \text{ for } x\in I \text{ if } I:=\left[L_1, L_2\right]\subset I'=\left[L_1', L_2'\right].
	\end{align*}
	Let us define
	\begin{align*}
	\left(\widetilde{\CU}_{\left[L_1, L_2\right]}(x),\widetilde{\CV}_{\left[L_1, L_2\right]}(x)\right)=\begin{cases}
	\left(\CU_{\left[L_1, L_2\right]}(x),\CV_{\left[L_1, L_2\right]}(x)\right), &\text{ in } I,\\
	\left(0, 0\right), &\text{ in } I'\setminus I.
	\end{cases}
	\end{align*}
Since $\left(\CU_{\left[L_1, L_2\right]}(x),\CV_{\left[L_1, L_2\right]}(x)\right)$ is unique positive solution of \eqref{steady} and 
	\begin{align*}
	\left\{\begin{array}{lll}
	\displaystyle\int_{L_1'}^{L_2'}J_1(x-y)\widetilde{\CU}_{\left[L_1, L_2\right]}(y)dy\geq 	\displaystyle\int_{L_1}^{L_2}J_1(x-y)\widetilde{\CU}_{\left[L_1, L_2\right]}(y)dy,\\	
	\displaystyle\int_{L_1'}^{L_2'}J_2(x-y)\widetilde{\CV}_{\left[L_1, L_2\right]}(y)dy\geq 	\displaystyle\int_{L_1}^{L_2}J_1(x-y)\widetilde{\CV}_{\left[L_1, L_2\right]}(y)dy.
	\end{array}\right.
	\end{align*}
	We can check that $\left(\widetilde{\CU}_{\left[L_1, L_2\right]}(x),\widetilde{\CV}_{\left[L_1, L_2\right]}(x)\right)$ is subsolution of \eqref{steady} with $\left[L_1', L_2'\right]$. Using Remark \ref{re10}, we deduce that
	$\left(\widetilde{\CU}_{\left[L_1, L_2\right]}(x),\widetilde{\CV}_{\left[L_1, L_2\right]}(x)\right)\leq \left(\CU_{\left[L'_1, L'_2\right]}(x),\CV_{\left[L'_1, L'_2\right]}(x)\right), \text{ in } I'.$
	Hence 
	\begin{align*}
	\left({\CU}_{\left[L_1, L_2\right]}(x),{\CV}_{\left[L_1, L_2\right]}(x)\right)\leq \left(\CU_{\left[L'_1, L'_2\right]}(x),\CV_{\left[L'_1, L'_2\right]}(x)\right), \text{ in } I.
	\end{align*}
	Moreover, by Proposition \ref{nd}, for every closing interval $\left[L_1, L_2\right]$ with $L_2-L_1>L_0$ we have
	\begin{align*}
	\left({\CU}_{\left[L_1, L_2\right]}(x),{\CV}_{\left[L_1, L_2\right]}(x)\right)\leq \left(K_1, K_2\right) \text{ in } \left[L_1, L_2\right].
	\end{align*} 
	Then we can define
	$\left(\CU^*(x),\CV^*(x)\right):=\lim\limits_{-L_1,L_2\rightarrow+\infty}\left({\CU}_{\left[L_1, L_2\right]}(x),{\CV}_{\left[L_1, L_2\right]}(x)\right), \text{ for } x\in \mathbb{R}.$
	Clearly, 
	\begin{align*}
	\left(0, 0\right)<\left(\CU^*(x),\CV^*(x)\right)\leq \left(K_1, K_2\right), \text{ in }\mathbb{R}.
	\end{align*}
	Moreover, by the dominated convergence theorem, it is easy to check that $\left(\CU^*,\CV^*\right)$ is a positive solution of \eqref{steady} with $\left(L_1, L_2\right)$ replaced by $\left(-\infty, +\infty\right)$.\\
	
	We next prove that both $\CU^*$ and $\CV^*$ are positive constants. It suffices to show that 
	\begin{align*}
	\left(\CU^*\left(x_0\right),\CV^*\left(x_0\right)\right)=\left(\CU^*(0),\CV^*(0)\right) \text{ for any given } x_0\in \mathbb{R}\setminus\{0\}.
	\end{align*}
We define
	$\left(\CU^1_{\left[L_1, L_2\right]}(x),\CV^1_{\left[L_1, L_2\right]}(x)\right)=\left(\CU_{\left[L_1-\left|x_0\right|, L_2+\left|x_0\right|\right]}\left(x+x_0\right),\CV_{\left[L_1-\left|x_0\right|, L_2+\left|x_0\right|\right]}\left(x+x_0\right)\right).$
Note that
	
	$\left[L_1+2\left|x_0\right|, L_2-2\left|x_0\right|\right]\subset \left[L_1+\left|x_0\right|-x_0, L_2-\left|x_0\right|-x_0\right]\subset \left[L_1, L_2\right].$
	By the same argument as above, we see that
	$\left(\CU_{\left[L_1+2\left|x_0\right|, L_2-2\left|x_0\right|\right]},\CV_{\left[L_1+2\left|x_0\right|, L_2-2\left|x_0\right|\right]}\right)\leq \left(\CU^1_{\left[L_1, L_2\right]},\CV^1_{\left[L_1, L_2\right]}\right)\leq \left(\CU_{\left[L_1, L_2\right]},\CV_{\left[L_1, L_2\right]}\right).$

Letting $-L_1,L_2\rightarrow \infty$ in the above inequalities and using the definition of $\left(\CU^*, \CV^*\right)$, we induce that 
	\begin{align*}
	\left(\CU^*\left(x+x_0\right),\CV^*\left(x+x_0\right)\right)=\left(\CU^*(x),\CV^*(x)\right) \text{ for all } x\in \mathbb{R}.
	\end{align*}
	Putting $x=0$, we obtain that
	$\left(\CU^*\left(x_0\right),\CV^*\left(x_0\right)\right)=\left(\CU^*(0),\CV^*(0)\right).$
	Since $\left(K_1, K_2\right)$ is the only constant solution of \eqref{steady}, we can arrive at
	$\left(\CU^*(x),\CV^*(x)\right)=\left(K_1, K_2\right), \text{ for all } x\in \mathbb{R}.$
	By Dini's theorem, the following convergence is locally uniform in $x\in \mathbb{R}$,
	$\lim\limits_{-L_1, L_2\rightarrow \infty}\left({\CU}_{\left[L_1, L_2\right]}(x),{\CV}_{\left[L_1, L_2\right]}(x)\right)=\left(K_1, K_2\right).$
\end{proof}
\begin{proposition}\label{steady_solution}
	Assume \textbf{(J)} holds and $\left(U,V\right)$ is the unique positive solution of $\eqref{fixed}$. Then we have the following conclusions.\\
	\textbf{(i)}\;If $\lambda_p(L)< 0$, then $(U,V)$ converges to $\left(\mathcal{U},\mathcal{V}\right)$ as $t\rightarrow \infty$ uniformly for $x\in \left[L_1, L_2\right]$, where $\left(\mathcal{U},\mathcal{V}\right)$ is the unique positive solution of \eqref{steady} given by Proposition \ref{nd}.\\
	\textbf{(ii)} \;If $\lambda_p(L)\geq  0$, then $(U,V)$ converges to $(0, 0)$ as $t\rightarrow \infty$ uniformly for $x\in \left[L_1, L_2\right]$.
\end{proposition}
\begin{proof}[\bf Proof]
\noindent\textit{\textbf{(i)}} Let us first assume that $\lambda_p(L)<0$. From Remark \ref{re10}, for $\left(\Gamma_1, \Gamma_2\right)\geq \left(\|U_0\|,\|V_0\|\right)$ big enough, $\left(\Gamma_1, \Gamma_2\right)$ is a supersolution of $\eqref{steady}$. Denote $\left(\overline{U}, \overline{V}\right)$ be the unique positive solution of $\eqref{fixed}$ with initial datum $\left(\Gamma_1, \Gamma_2\right)$. A standard argument using comparison principle in Lemma \ref{comparison_1} we can derive that
\begin{align*}
\left\{\begin{array}{lll}
U\leq \overline{U}\leq \Gamma_1, \text{ for}\; (t,x) \; \text{in }\; \Omega_L,\\
V\leq \overline{V}\leq \Gamma_2, \text{ for}\; (t,x) \; \text{in }\; \Omega_L.
\end{array}\right.
\end{align*}
Next, we show that $(\overline{U},\overline{V})$ is non-increasing in $t$. For fixed $\tau>0$, since $\left(\overline{U}(\tau,x),\overline{V}(\tau,x)\right)\leq \left(\Gamma_1, \Gamma_2\right)$, we may use  Lemma \ref{comparison_1} to compare  $\left(\overline{U}(t+\tau,x), \overline{V}(t+\tau, x)\right)$ with  $\left(\overline{U}(t,x), \overline{V}(t,x)\right)$ to deduce that 
\begin{align*}
\left(\widetilde{U}(t+\tau, x), \widetilde{V}(t+\tau, x)\right)\leq \left(\widetilde{U}(t, x), \widetilde{V}(t, x)\right)\text{ in } \Omega_L.
\end{align*} 
This proves the monotonicity in $t$. Therefore, we can define $\left(\overline{\CU}(x),\overline{\CV}(x)\right)=\lim\limits_{t\rightarrow \infty}(\overline{U}(t,x),\overline{V}(t,x))$ which is a positive function pair, and by the dominated convergence theorem, it is a positive solution of $\eqref{steady}$.\\
Using Proposition \ref{maxi1}, we see that $\left(U(1,x),V(1,x)\right)>(0,0)$ in $\left[L_1, L_2\right]$ since $\left(U, V\right)\geq (0,0)$ and $U,V\not\equiv 0$. 	Let $\left(\psi,\varphi\right)$ is a positive eigenfunction pair corresponding to $\lambda_p(L)$. From the proof of Proposition \ref{nd}, for $\varepsilon$ small enough $\left(\varepsilon\psi,\varepsilon\varphi\right)$ is supersolution of \eqref{steady} and  $\left(\varepsilon\psi,\varepsilon\varphi\right)\leq \left(U(1,x),V(1,x)\right)$. Let $\left(\underline{U},\underline{V}\right)$ be the solution of  \eqref{fixed} with initial function pair $\left(\varepsilon\psi,\varepsilon\varphi\right)$. By the same argument as above, $\underline{U},\underline{V}$ are non-decreasing in $t$. Then, we also define $\left(\underline{\CU}(x),\underline{\CV}(x)\right):=\lim\limits_{t\rightarrow \infty}\left(\underline{U}(t,x),\underline{V}(t,x)\right)$ which is a positive function pair, and by the dominated convergence theorem, it is a positive solution of $\eqref{steady}$.
Moreover,  Lemma \ref{comparison_1} yields that
 $\left(\underline{U}(t,x),\underline{V}(t,x)\right)\leq \left(U(t+1,x),V(t+1,x)\right)\leq \left(\overline{U}(t+1,x),\overline{V}(t+1,x)\right),\text{ in }\Omega_L.$

Letting $t\rightarrow  \infty$ in the above inequalities, we have
\begin{align*}
\left(\underline{\CU}(x),\underline{\CV}(x)\right)\leq \lim\limits_{t\rightarrow \infty}\left(U(t,x),V(t,x)\right)\leq \left(\overline{\CU}(x),\overline{\CV}(x)\right), \text{ in } \left[L_1, L_2\right]. 
\end{align*}
Since  \eqref{steady} has a unique positive solution $\left(\CU,\CV\right)$, we derive
\begin{align*}
\lim\limits_{t\rightarrow \infty}\left(U(t,x),V(t,x)\right)=(\CU(x),\CV(x)),\text{ in } \left[L_1, L_2\right].
\end{align*}
The convergence is uniform by Dini's theorem.\\
\noindent\textbf{\textit{(ii)}} In the case $\lambda_p(L)\leq 0$,  as above we have $(0,0)\leq \left(U(t,x),V(t,x)\right)\leq  \left(\overline{U}(t,x),\overline{V}(t,x)\right)$ and $\left(\overline{U}(t,x),\overline{V}(t,x)\right)$ converges pointwise to $\left(\overline{\CU}(x),\overline{\CV}(x)\right)$ a solution of \eqref{steady}. By Proposition \ref{nd}, in this situation we induce that $\left(\overline{\CU}(x),\overline{\CV}(x)\right)\equiv \left(0,0\right)$, hence $\lim\limits_{t\rightarrow \infty}\left(U(t,x),V(t,x)\right)=(0,0)$.
\end{proof}
\section{ Long-time behavior: Spreading-Vanishing dichotomy}\label{sec.4}
In this section, we investigate the long time dynamics of problem \eqref{main}, says the spreading-vanishing phenomena.  We see that
the free boundaries $h(t), -g(t)$ are strictly increasing functions with respect to time $t$. Therefore, $h_{\infty} := \lim\limits_{t\to \infty} h(t)$ and $g_{\infty} := \lim\limits_{t\to\infty} g(t)$ are well-defined, and $h_{\infty}, g_{\infty}\leq \infty$.
Before proving Theorem \ref{theorem_1.2} and \ref{theorem1.4}, we prove the following lemmas :
\begin{lemma}\label{lem.4.1}
If $\lim\limits_{t\to\infty}h(t)-g(t)<\infty$, then $\lim\limits_{t\to\infty} \left(u\left(t, \cdot\right), v\left(t, \cdot\right)\right)=\left(0, 0\right)$ and $\lambda_p\left(g_{\infty} , h_{\infty}\right)\geq 0$.
\end{lemma}
\begin{proof}[\bf Proof]
The idea of this proof comes from \cite[Lemma 4.6]{DN}. We first claim that $\lambda_p\left(g_{\infty}, h_{\infty}\right)\geq 0$, where $\lambda_p\left(g_{\infty}, h_{\infty}\right)$ is the principal eigenvalue
of \eqref{eigenproblem} with $\left[-L, L\right]$ replaced with $\left[g_{\infty}, h_{\infty}\right]$. Assume by the contradiction that $\lambda_p\left( g_{\infty}, h_{\infty}\right)<0$, and by (1) in  proposition \ref{pro.3.4}, we can find a large constant $T>0$ such that $\lambda_p\left( g(T), h(T)\right)>0$. It is obvious that $J_i$ are continuous and $J_i(0)>0,\; i=1,2$ then we may also assume that $g(T)$ and $h(T)$ satisfy $\left|g(T)- g_{\infty}\right|<\varepsilon$ and $\left|h(T)- h_{\infty}\right|<\varepsilon$. For $\varepsilon\in \left(0, h_0\right)$ small enough such that $J_i(x)>0$ for all $x\in \left[-2\varepsilon, 2\epsilon\right],\; i=1,2$. Let $\left(u_1(t, x), v_1(t, x)\right)$ be the solution of the problem \eqref{main} on $\O_T = \left(0,\infty\right)\times\left(g(T), h(T)\right)$, and initial functions $\left(u_0, v_0\right)=\left(u(T, x), v(T, x)\right)$. By Lemma \ref{comparison_1}, we deduce that
\begin{align*}
u_1\left(\tau, x\right)\leq u(T+\tau, x),\;\;\; v_1\left(\tau, x\right)\leq v(T+\tau, x)\;\;\;\text{for}\;\, \left(\tau,x\right)\in \left[0, \infty\right)\times \left[g(T), h(T)\right].
\end{align*}
Due to $\lambda_p\left(g(T), h(T)\right)<0$, then $\left(u, v\right)$ converges to $\left(\widetilde{u}, \widetilde{v}\right)$, which is a solution of \eqref{fixed}, as $\tau\to \infty$ uniformly for $x$. We further infer
\begin{align*}
\left\{\begin{array}{lll}
0< \widetilde{u}_1(x)=\lim\limits_{t\to\infty}u_1(t, x)\leq \liminf\limits_{t\to\infty}u(t,x),\\
0< \widetilde{v}_1(x)=\lim\limits_{t\to\infty}v_1(t, x)\leq \liminf\limits_{t\to\infty}v(t,x),
\end{array}\right.
\end{align*}
with uniform convergence  for $x\in \left(g(T), h(T)\right)$. Thus, there exists $T_1\geq T$ such that
\begin{align*}
\left\{\begin{array}{lll}
0< \dfrac{\widetilde{u}_1(x)}{2}< u(t,x),\\
0< \dfrac{\widetilde{v}_1(x)}{2}< v(t,x),
\end{array}\right.\;\;\text{for}\;\; t\geq T_1,\; x\in \left[g(T), h(t)\right].
\end{align*}
Denote $A = \min\left\{\mu_1, \mu_2\right\}$. Since $J_i$ are continuous and $J_i(0) > 0$, for all $i=1, 2$, there exists $\varepsilon > 0$ such
that $m =  \min\left\{\inf\limits_{x\in [-\varepsilon, \varepsilon]}J_1(x), \inf\limits_{x\in [-\varepsilon, \varepsilon]}J_2(x)\right\}>0$. From the fact that $\left[h(t)-\varepsilon, h(t)+\varepsilon\right]\subset \left[g(T), h(T)\right]$ for $t\geq T$, we obtain
\begin{align*}
h^{\prime}(t) &=\mu_1 \displaystyle\int_{g(t)}^{h(t)}\displaystyle\int_{h(t)}^{\infty}J_1(x-y)u(t, x)dydx  +\mu_2\displaystyle\int_{g(t)}^{h(t)}\displaystyle\int_{h(t)}^{\infty}J_2(x-y)v(t, x)dydx\\& \geq \mu_1 \displaystyle\int_{h(t)-\varepsilon}^{h(t)}\displaystyle\int_{h(t)}^{h(t)+\varepsilon}J_1(x-y)u(t, x)dydx + \mu_2\displaystyle\int_{h(t)-\varepsilon}^{h(t)}\displaystyle\int_{h(t)}^{h(t)+\varepsilon}J_2(x-y)v(t, x)dydx\\& \geq m\varepsilon A\displaystyle\int_{h(t)-\varepsilon}^{h(t)}\left[u(t,x)+v(t,x)\right]dx \geq  m\varepsilon A\displaystyle\int_{h(t)-\varepsilon}^{h(t)}\left[\dfrac{\widetilde{u}_1(x)+\widetilde{v}_1(x)}{2}\right]dx\\&\geq \dfrac{m\varepsilon^2 A}{2} \min\left\{\min\limits_{x\in \left[g(T), h(T)\right]}\widetilde{u}_1(x), \min\limits_{x\in \left[g(T), h(T)\right]}\widetilde{v}_1(x)\right\}>0.
\end{align*}
However, this contradicts to the fact that $h_{\infty}<\infty$. Therefore $\lambda_p\left(g_{\infty}, h_{\infty}\right)\geq 0$. Next, we consider $\left(u_2(t,x), v_2(t,x)\right)$ that be the solution of \eqref{main} with $\O_T= \left(0, \infty\right)\times\left(g_{\infty}, h_{\infty}\right)$ and $\left(u_0, v_0\right)=\left(K_1, K_2\right)$. By the comparison principal, Lemma \ref{comparison_1}, we have
\begin{align*}
0\leq u(t, x)\leq u_2(t, x),\;\;\; 0\leq v(t,x)\leq v_2(t, x)\;\;\;\text{for}\; t>0,\; x\in \left[g(t), h(t)\right].
\end{align*}
Since $\lambda_p\left(g_{\infty}, h_{\infty}\right)\geq 0$, we have $\lim\limits_{t\to\infty}\left(u_2, v_2\right) = \left(0, 0\right)$ uniformly for $x\in \left[ g_{\infty}, h_{\infty}\right]$.
\end{proof}

\begin{lemma}\label{lem.4.3}
$h_{\infty}<\infty$ if and only of $g_{\infty}>-\infty$.
\end{lemma}
\begin{proof}[\bf Proof]
 Suppose on the contrary, we may assume without loss of generality that $h_{\infty}= \infty$  and $g_{\infty}<-\infty$. We claim that $\CR_0>1$ and there exists $T>0$ such that
\begin{align}\label{4.3}
\lambda_p\left(g(t), h(t)\right)>0\;\;\text{for}\;\; t\geq T.
\end{align}
If $\CR_0\leq 1$, it follows from Lemma \ref{lem.4.2} that $h_{\infty}-g_{\infty}<\infty$ which contradicts to $h_{\infty}= \infty$. Thus, when $\CR_0>1$ then $\lambda_{\infty}>0$, which implies \eqref{4.3}. We may now making use of \eqref{4.3} and $g_{\infty}>-\infty$, as in the proof of \ref{lem.4.2}, to find some constants $\chi >0$ and $T_1>T$ such that
\begin{align*}
g^{\prime}(t) &=-\mu_1 \displaystyle\int_{g(t)}^{h(t)}\displaystyle\int_{h(t)}^{\infty}J_1(x-y)u(t, x)dydx  -\mu_2\displaystyle\int_{g(t)}^{h(t)}\displaystyle\int_{h(t)}^{\infty}J_2(x-y)v(t, x)dydx\\& \leq -\mu_1 \displaystyle\int_{h(t)-\varepsilon}^{h(t)}\displaystyle\int_{h(t)}^{h(t)+\varepsilon}J_1(x-y)u(t, x)dydx - \mu_2\displaystyle\int_{h(t)-\varepsilon}^{h(t)}\displaystyle\int_{h(t)}^{h(t)+\varepsilon}J_2(x-y)v(t, x)dydx\\&\leq -\chi <0.
\end{align*}
This contradicts to the fact that $g_{\infty}>-\infty$.
\end{proof}
\begin{lemma}\label{lem.4.4}
If  $\lim\limits_{t\to \infty}h(t)  -g(t) = +\infty$, then $\lim\limits_{t\to\infty}\left(u(t, x), v(t, x)\right) = \left(K_1, K_2\right)\;\;\text{locally\;uniformly\;in}\; \R$.
\end{lemma}
\begin{proof}[\bf Proof]
Thanks to Lemma \ref{lem.4.3}, $\lim\limits_{t\to \infty}h(t)  -g(t) = \infty$, which implies that $\lim\limits_{t\to\infty}h(t) = -\lim\limits_{t\to\infty}g(t)=\infty$. Let $\left\{t_n\right\}_{n\geq 1}$
be a increasing sequence satisfying
$\lim\limits_{t\to\infty}t_n=\infty,\;\;\text{and}\;\; \lambda_p\left(g\left(t_n\right), h\left(t_n\right)\right)>0, \;\;\text{for\; all}\; n\geq 1.$
We define $A_n = g\left(t_n\right),\; B_{n} = h\left(t_n\right)$, and let $\left(\underline{u}_n, \underline{v}_n\right)$ be a unique solution of the following problem
\begin{align}\label{4.04}
\left\{\begin{array}{lll}
 \underline{u}_t =  d_1\displaystyle\int_{A_n}^{B_n}J_1(x-y)\underline{u}(t,y)dy - d_1\underline{u}(t,x) - a\underline{u}(x,t) + H\left(\underline{v}(t, x)\right), & t>t_n,\,\,\,\, \, A_n<x<B_n, \\
\underline{v}_t = d_2\displaystyle\int_{A_n}^{B_n}J_2(x-y)\underline{v}(t,y)dy - d_2\underline{v}(t,x) -b\underline{v}(t,x)+ G\left(\underline{u}(t,x)\right), &  t>t_n,\,\,\,\, \, A_n<x<B_n, \\
\underline{u}\left(t_n, x\right) = u\left(t_n, x\right),  \underline{v}\left(t_n, x\right)= v\left(t_n, x\right), &     A_n<x<B_n.
\end{array}\right.
\end{align}
By Lemma \ref{lem_3.1} and the comparison argument we give
\begin{align}\label{4.05}
u(t, x)\geq \underline{u}_n\left(t, x\right),\; v(t, x)\geq \underline{v}_n\left(t, x\right)\;\;\text{in}\;\; \left[t_n, \infty\right)\times\left[A_n, B_n\right].
\end{align}
Due to $\lambda_p\left(g\left(t_n\right), h\left(t_n\right)\right)>0$, and by Proposition \ref{steady_solution}, it follows that  \eqref{4.04} admits a unique positive steady sate $\left(\underline{u}_n\left(x\right), \underline{v}_n\left(x\right)\right)$ and satisfy
\begin{align}\label{4.06}
\lim\limits_{t\to\infty}\left(\underline{u}_n\left(t, x\right), \underline{v}_n\left(t, x\right)\right) =  \left(\underline{u}_n\left(x\right), \underline{v}_n\left(x\right)\right)\;\;\text{uniformly\;in}\; \left[A_n, B_n\right].
\end{align}
By ${\bf(iii)}$, Proposition \ref{steady_solution}, we have
$\lim\limits_{n\to \infty}\left(\underline{u}_n(x), \underline{v}_n(x)\right)=\left(K_1, K_2\right)\;\;\text{localy uniformly in}\;\R.$
Combining \eqref{4.05}, \eqref{4.06}, we further infer
$\liminf\limits_{t\to \infty}\left(u(t, x), v(t, x)\right)\geq \left(K_1, K_2\right)\;\text{localy uniformly\; in}\;\R.$
To complete the proof, we need to prove that
\begin{align}\label{3.016}
\limsup\limits_{t\to \infty}\left(u(t, x), v(t, x)\right)\leq \left(K_1, K_2\right)\;\text{localy uniformly\; in}\;\R.
\end{align} 
Let $\left(\widetilde{u}(t), \widetilde{v}(t)\right)$ be the solution of the ODE problem
\begin{align}\label{3.13}
\left\{\begin{array}{lll}
u^{\prime}(t) = -au + H(v), & t>0,\\
v^{\prime}(t) = -bv + G(u), & t>0,\\
u(0) = \left\|u\right\|_{\infty}, \,\, v(0) = \left\|v\right\|_{\infty}.
\end{array}\right.
\end{align}
Using Lemma \ref{comparison_1}, we have $u(t,x)\leq \widetilde{u}(t)$ and $v(t,x)\leq \widetilde{v}(t)$ for $t>0$ and $x\in \left[g(t), h(t)\right]$. Since $\CR_0>1$, the unique positive equilibrium $(K_1,K_2)$ of \eqref{3.13} is globally attractive. Hence, $\left(\widetilde{u}(t), \widetilde{v}(t)\right)\rightarrow (K_1,K_2)$ as $t\rightarrow \infty$, which implies \eqref{3.016}
\end{proof}
Clearly Theorem \ref{theorem_1.2} follows directly from Lemmas \ref{lem.4.1}, \ref{lem.4.3}, \ref{lem.4.4}.
\begin{lemma}\label{lem.4.2}
	If $\CR_0\leq 1$ then vanishing happens.
\end{lemma}
\begin{proof}[\bf Proof]
	For any $\alpha$ is a positive constant to be determined late. Direct calculations gives
	\begin{align}\label{3.12}
	&\dfrac{d}{dt}\displaystyle\int_{g(t)}^{h(t)}\left[u(t,x)+\alpha v(t,x)\right]dx\\&
	=\displaystyle\int_{g(t)}^{h(t)}\left[d_1\displaystyle\int_{g(t)}^{h(t)}J_1(x-y)u(t,y)dy - d_1u(t,x) - au(x,t) + H\left(v(t, x)\right)\right]dx\nonumber \\& +\displaystyle\int_{g(t)}^{h(t)}\alpha\left[d_2\displaystyle\int_{g(t)}^{h(t)}J_2(x-y)v(t,y)dy - d_2v(t,x)-bv(t,x)+ G\left(u(t,x)\right)\right]dx\nonumber\\&
	\leq -M_1 h^{\prime}(t) + M_2g^{\prime}(t) + \displaystyle\int_{g(t)}^{h(t)}\left[-au(t,x)+ H(v(t,x))-b\alpha v(t, x) + \alpha G\left(u(t, x)\right) \right]dx\nonumber,
	\end{align}
	where 
	$M_1 = \min\left\{\dfrac{d_1}{\mu_1}, \alpha\dfrac{d_2}{\mu_2}\right\},  M_2 = \max\left\{\dfrac{d_1}{\mu_1}, \alpha\dfrac{d_2}{\mu_2} \right\}$.
	Integrating from $0$ to $t$ of \eqref{3.12}, we obtain
	
	\begin{align*}
	M_1 h(t) - M_2g(t) &\leq \displaystyle\int_{-h_0}^{h_0}\left[u(0, x)+\alpha v(0, x)\right]dx+M_1h(0)-M_2g(0) \\&\hspace{3cm} +\displaystyle\int_{0}^{t}\displaystyle\int_{g(s)}^{h(s)}\left[-au(t,x)+ H(v(t,x))-b\alpha v(t, x) + \alpha G\left(u(t, x)\right) \right]dxds.
	\end{align*}
	We claim that $-au(t,x)+ H^{\prime}(0)v(t,x)-b\alpha v(t, x) + \alpha G^{\prime}\left(0\right)u(t, x) <0$. This is equality $ \alpha G^{\prime}\left(0\right)-a<0$ , and  $H^{\prime}(0)-b\alpha <0$. 
	
	Thanks to $\CR_0\leq 1$, we can choose the constant $\alpha$ that satisfies
	\begin{align}\label{3.15}
	\dfrac{H^{\prime}(0)}{b}\leq\alpha\leq\dfrac{a}{G^{\prime}(0)}.
	\end{align}
	For all $0 \leq s \leq t$, by \eqref{3.15} and Fubini's theorem, we have
	\begin{align*}
	&\displaystyle\int_{g(s)}^{h(s)}\left[-au(t,x)+ H(v(t,x))-b\alpha v(t, x) + \alpha G\left(u(t, x)\right) \right]dx\\&\leq\displaystyle\int_{g(s)}^{h(s)}\left[-au(t,x)+ H^{\prime}(0)v(t,x)-b\alpha v(t, x) + \alpha G^{\prime}\left(0\right)u(t, x) \right]dx
	\\&\leq \left(\alpha G^{\prime}\left(0\right)-a\right) \displaystyle\int_{g(s)}^{h(s)} u(t,x)dx +  \left(H^{\prime}(0)-b\alpha\right)\displaystyle\int_{g(s)}^{h(s)} v(t,x)dx\leq 0.
	\end{align*}
	Thus, we get $M_1 h_{\infty} - M_2 g_{\infty} < \infty$ , by letting $t\to\infty$. This together with Lemma \ref{lem.4.1} imply
	that vanishing happens. 
\end{proof}
Recall that $\lambda_p(L_1,L_2)$ is the principal eigenvalue of \eqref{eigenvalue}. By Proposition \ref{pro.3.4}, we can set
\begin{align*}
&\lambda_{\infty}:=\lim\limits_{L_2-L_1\rightarrow\infty}\lambda_p(L_1,L_2)=\dfrac{1}{2}\left(\dfrac{a}{H'(0)}+\dfrac{b}{G'(0)}-\sqrt{\left(\dfrac{a}{H'(0)}-\dfrac{b}{G'(0)}\right)^2+4}\right),\\
&\lambda_0:=\lim\limits_{L_2-L_1\rightarrow 0}\lambda_p(L_1,L_2)=\dfrac{1}{2}\left(\dfrac{a}{H'(0)}+\dfrac{b}{G'(0)}+\dfrac{d_1}{H'(0)}+\dfrac{d_2}{G'(0)}-\sqrt{\left(\dfrac{a}{H'(0)}-\dfrac{b}{G'(0)}+\dfrac{d_1}{H'(0)}-\dfrac{d_2}{G'(0)}\right)^2+4}\right).
\end{align*}
Applying Proposition \ref{pro.3.4} to \eqref{eigenvalue}, we follow that if $\CR^*\geq 1$
then $\lambda_p(L_1,L_2)<\lambda_0\leq 0$ for any finite interval $(L_1,L_2)$. Combining this with Lemma \ref{lem.4.1} we obtain the following theorem.
\begin{theorem}\label{th:theorem_4.5}
If $\CR^{*}\geq 1$ then spreading always happens for \eqref{main}.
\end{theorem}
We consider the case  $\CR^{*}<1< \CR_0$. In this case, we have $\lambda_{\infty}>0$ and $\lambda_{0}<0$, by Proposition \ref{pro.3.4}, there exist $\CL^*>0$ such that 
\begin{align*}
 \left\{\begin{array}{lll}
\lambda_p\left(I\right) =0,&\text{if}\;\;\left|I\right| = \CL^*,\\
\lambda_p\left(I\right) <0,&\text{if}\;\;\left|I\right| > \CL^*,\\
\lambda_p\left(I\right) >0,&\text{if}\;\;\left|I\right| < \CL^*.
\end{array}\right.
\end{align*}
where $I$ stands for a finite open interval in $\R$, and $|I|$ denotes its length.

\begin{theorem}\label{th:theorem_4.6}
Suppose that $\CR^{*}<1< \CR_0$ holds. If $h_0\geq \CL^*/2$ then spreading always happens for \eqref{main}.
\end{theorem}
\begin{proof}[\bf Proof]
Arguing indirectly we assume that $h_{\infty}-g_{\infty}>2h_0\geq \CL^*$. Hence $\lambda_p(h_{\infty},g_{\infty})<0$, which contradicts to Lemma \ref{lem.4.1}. Therefore when $h_0\geq \CL^*/2$, spreading always occurs for \eqref{main}.
\end{proof}
\begin{lemma}\label{lem-4.8}
Suppose that $\CR^{*}<1< \CR_0$, $\mu_1=\mu_2=\mu$ and $h_0< \CL^*/2$ holds. Then there exist $\Theta$ such that vanishing happen to \eqref{main} if $\mu\leq \Theta$.
\end{lemma}
\begin{proof}[\bf Proof]
Fix $h^*\in \left(h_0, \CL^*/2\right)$. 
Let $\left(\widehat{u}_1, \widehat{v}_1\right)$ be a  solution of the problem
\begin{align}\label{eq:4.9}
\left\{\begin{array}{lll}
 {u_1}_{t} =  d_1\left[\displaystyle\int_{-h^*}^{h^*}J_1(x-y)u(t,y)dy - u_1(t,x)\right] - au_1(x,t) + H\left(v_1(t, x)\right), & t>0,\,\,\,\, \, -h^*<x<h^*, \\
{v_1}_{t} = d_2\left[\displaystyle\int_{h^*}^{h^*}J_2(x-y)v_1(t,y)dy - v_1(t,x)\right] -bv_1(t,x)+ G\left(u_1(t,x)\right), & t>0,\,\,\,\, \, -h^*<x<h^*, \\
u_1\left(0, x\right) = {u_1}_0\left( x\right), & t>0,\;\; x\in [-h_0, h_0],\\
v_1(0, x) = {v_1}_0(x),& t>0,\; x\in [-h_0, h_0], \\  v_1(0, x) = u_1(0, x) = 0,&\;\text{for}\;\; x\in \left[-h^*, h_0\right)\bigcup \left(h_0, h^*\right],
\end{array}\right.
\end{align}
By choosing $h^*$, we have $\lambda^* = \lambda_p(-h^*, h^*)>0$. Since $2h^*>2h_0>\CL^*$, let $\phi = \left(\phi_1, \phi_2\right)>0$ be a solution of 
\eqref{steady} with $L= h^*$, it easily follows that
\begin{align*}
\left\{\begin{array}{lll}
   d_1\displaystyle\int_{-h^*}^{h^*}J_1(x-y)\phi_1(y)dy - d_1\phi_1(x) - a\phi_1(x) + H^{\prime}\left(0\right)\phi_2(x)+H'(0)\lambda^*\phi_1(x)=0, &  x\in [-h^*, h^*], \\
 d_2\displaystyle\int_{-h^*}^{h^*}J_2(x-y)\phi_2(y)dy - d_2\phi_2(x) -b\phi_2(x)+ G^{\prime}\left(0\right)\phi_1(x)+G'(0)\lambda^*\phi_2(x)=0, & x\in [-h^*, h^*].
\end{array}\right.
\end{align*}
On the other hand, for any $C_1, C_2>0$ and $\left(\widetilde{u_1}, \widetilde{v_1}\right)=\left(C_1 e^{\sigma t}\phi_1, C_2 e^{\sigma t}\phi_2\right)$. Direct calculation gives
\begin{align}\label{4.6}
&d_1\displaystyle\int_{-h^*}^{h^*}J_1(x-y)\widetilde{u_1}(t,y)dy - d_1\widetilde{u_1}(t,x) - a\widetilde{u_1}(x,t) + H^{\prime}\left(0\right)\widetilde{v_1}(t, x)-{\widetilde{u_1}}_t(t, x)\nonumber\\&
=e^{\sigma t}\left\{C_1d_1\displaystyle\int_{-h^*}^{h^*}J_1(x-y)\phi_1(y)dy - C_1d_1\phi_1(x) - C_1 a\phi_1(x)+C_2H^{\prime}(0)\phi_2(x)-\sigma C_1\phi_1(x)\right\}\\&
=
e^{\sigma t}\left\{-C_1\phi_1(x)[\lambda^*H'(0)+\sigma] +\left(C_2-C_1\right) H^{\prime}(0)\phi_2(x)\right\}\nonumber.
\end{align}
Similarly, we obtain
\begin{align}\label{4.7}
&d_2\displaystyle\int_{-h^*}^{h^*}J_2(x-y)\widetilde{v_1}(t,y)dy - d_2\widetilde{v_1}(t,x) - b\widetilde{v_1}(x,t) + G^{\prime}\left(0\right)\widetilde{u_1}(t, x)-{\widetilde{v_1}}_t(t, x)\nonumber\\&=
e^{\sigma t}\left\{-C_2\phi_2(x)[\lambda^*G'(0)+\sigma] +\left(C_1-C_2\right) G^{\prime}(0)\phi_1(x)\right\}.
\end{align}
Combining \eqref{4.6}, \eqref{4.7},  choosing $C_1=C_2=C>0$  large such that $C_1\left(\phi_1, \phi_2\right)>\left(u_0, v_0\right)$ and $\sigma=\max\{-\lambda^*H'(0),-\lambda^*G'(0)\}<0$, we derive 
\begin{align*}
\left\{\begin{array}{lll}
d_1\displaystyle\int_{-h^*}^{h^*}J_1(x-y)\widetilde{u_1}(t,y)dy - d_1\widetilde{u_1}(t,x) - a\widetilde{u_1}(x,t) + H^{\prime}\left(0\right)\widetilde{v_1}(t, x)-{\widetilde{u_1}}_t(t, x)<0,\\
d_2\displaystyle\int_{-h^*}^{h^*}J_2(x-y)\widetilde{v_1}(t,y)dy - d_2\widetilde{v_1}(t,x) - b\widetilde{v_1}(x,t) + G^{\prime}\left(0\right)\widetilde{u_1}(t, x)-{\widetilde{v_1}}_t(t, x)<0.
\end{array}\right.
\end{align*}
Applying the comparison principle, Lemma \ref{comparison_1}, and from the equation \eqref{eq:4.9}, we have
\begin{align}\label{4.8}
\left(\widehat{u}_1, \widehat{v}_1\right)\leq \left(\widetilde{u_1}, \widetilde{v_1}\right)=\left(C e^{\sigma t}\phi_1, C e^{\sigma t}\phi_2\right),\;\;\text{for}\;\; t>0\; x\in \left[-h^*, h^*\right].
\end{align}
Next, we define
$\xi(t) = h_0+2\mu Ch^*\displaystyle\int_{0}^{t}e^{\sigma \tau}d\tau,\;\;\text{and}\;
\psi(t) = -h_0-2\mu Ch^*\displaystyle\int_{0}^{t}e^{\sigma \tau}d\tau,\;\;\text{for}\; t\geq 0.$

First, we claim that $\left(\widehat{u}_1, \widehat{v}_1, \xi, \psi\right)$ is an supersolution of \eqref{main}.
By direct computations, we obtain
\begin{align*}
\left\{\begin{array}{lll}
\xi(t) = h_0-2\mu\dfrac{Ch^*}{\sigma}\left(1- e^{\sigma t}\right)\leq h_0-2\mu\dfrac{Ch^*}{\sigma}\leq h^*,\\
\psi(t) = -h_0+2\mu\dfrac{Ch^*}{\sigma}\left(1- e^{\sigma t}\right)\geq -h_0+2\mu\dfrac{Ch^*}{\sigma}> -h^*,
\end{array}\right.\;\;\text{for\; any}\; t>0.
\end{align*}
Hence, there exists $\mu^*$ such that  
$0<\mu\leq \Theta: =\dfrac{-\sigma\left(h^*-h_0\right)}{2Ch^*}.$
Therefore, from the equation \eqref{eq:4.9}, we have
\begin{align*}
\left\{\begin{array}{lll}
 \widehat{u_1}_{t} \geq  d_1\left[\displaystyle\int_{\psi(t)}^{\xi(t)}J_1(x-y)\widehat{u_1}(t,y)dy - \widehat{u_1}(t,x)\right] - a\widehat{u_1}(x,t) + H\left(\widehat{v_1}(t, x)\right), & t>0,\,\,\,\, \, \psi(t)\leq x\leq \xi(t), \\
\widehat{v_1}_{t}\geq d_2\left[\displaystyle\int_{\psi(t)}^{\xi(t)}J_2(x-y)\widehat{v_1}(t,y)dy - \widehat{v_1}(t,x)\right] -b\widehat{v_1}(t,x)+ G\left(\widehat{u_1}(t,x)\right), & t>0,\,\,\,\, \, \psi(t)\leq x\leq \xi(t).
\end{array}\right.
\end{align*}
Second, by \eqref{4.8}, we have
\begin{align*}
\left\{\begin{array}{lll}
\xi^{\prime}(t) =2\mu Ch^*e^{\sigma t}\geq \mu\left(  \displaystyle\int_{\psi(t)}^{\xi(t)}\displaystyle\int_{\xi(t)}^{\infty}J_1(x-y)u(t, x)dydx  +\displaystyle\int_{\psi(t)}^{\xi(t)}\displaystyle\int_{\xi(t)}^{\infty}J_2(x-y)v(t, x)dydx\right),\\
\psi^{\prime}(t) =-2\mu Ch^*e^{\sigma t}\leq -\mu\left(  \displaystyle\int_{\psi(t)}^{\xi(t)}\displaystyle\int_{-\infty}^{\psi(t)}J_1(x-y)u(t, x)dydx+ \displaystyle\int_{\psi(t)}^{\xi(t)}\displaystyle\int_{-\infty}^{\psi(t)}J_2(x-y)v(t,x)dydx\right).
\end{array}\right.
\end{align*}
We proved that $\left(\widetilde{u_1}, \widetilde{v_1}, \xi, \psi\right)$ is an supersolution of \eqref{main}. By the comparison argument, Lemma \ref{comparison_1}, we have
$u(t, x)\leq \widehat{u}_1(t, x),\;\; v(t, x)\leq \widehat{v}_1(t, x),\;\;g(t)\geq \psi(t)\;\;\text{and}\;\; h(t)\leq \xi(t)\; \text{for}\; t>0,\; x\in \left[g(t), h(t)\right].$
Therefore, $\lim\limits_{t\to\infty}\left(h(t)-g(t)\right)\leq \lim\limits_{t\to\infty}\left(\xi(t)-\psi(t)\right)\leq 2h^*<\infty.$
\end{proof}
\begin{lemma}\label{lem-4.9}
Suppose that $\CR^{*}<1< \CR_0$, $\mu_1=\mu_2=\mu$  and $h_0< \CL^*/2$ hold .Then there exist $\Lambda>0$ such that spreading happen to \eqref{main} if $\mu>\Lambda$.
\end{lemma}
\begin{proof}[\bf Proof]
We assume that  $h_{\infty}-g_{\infty}<\infty$ for any $\mu>0$ and shall derive a contradiction. By  Lemma $\ref{lem.4.1}$, we have $\lambda(g_{\infty},h_{\infty})\leq 0$, then $h_{\infty}-g_{\infty}\leq \CL^*$. We will write $(u_{\mu},v_{\mu},h_{\mu},g_{\mu})$ in place of $(u,v,h,g)$ to clarify the dependence of the solution of $\eqref{main}$ on $\mu$. Thanks to Lemma $\ref{comparison_1}$, $(u_{\mu},v_{\mu},h_{\mu},g_{\mu})$ are increasing in $\mu>0$. Define
	\begin{align*}
	h_{\mu, \infty}:=\lim _{t \rightarrow\infty} h_{\mu}(t), \quad g_{\mu, \infty}:=\lim _{t \rightarrow\infty} g_{\mu}(t).
	\end{align*}
	It is clear that both $h_{\mu, \infty}$ and $g_{\mu, \infty}$ are increasing in $\mu>0$ and bounded. Therefore, we can denote
	\begin{align*}
	H_{\infty}:=\lim _{\mu \rightarrow\infty} h_{\mu, \infty}<\infty, \quad G_{\infty}:=\lim _{\mu \rightarrow\infty} g_{\mu, \infty}<\infty.
	\end{align*} 
	By the condition ${(\bf{J})}$, there exist constants $\varepsilon>0$ and $\delta_0>0$ such that
	\begin{align*}
	J_i(x-y)\geq \delta_0\,\,\;\text{ for }\,\; \left|x-y\right|\leq \varepsilon_0.
	\end{align*}
Then there exists $\overline{\mu}, \overline{t}$ such that for all $\mu\geq \overline{\mu}, t\geq \overline{t}$, we have $h_{\mu}(t)>H_{\infty}-\epsilon_{0} / 4$. It follows that
	\begin{align*}
	&h_{\mu, \infty}-h_{\mu}\left(t_{1}\right)\\
	=& \mu\left(\displaystyle\int_{t_{1}}^{\infty} \displaystyle\int_{g_{\mu}(\tau)}^{h_{\mu}(\tau)} \displaystyle\int_{h_{\mu}(\tau)}^{\infty} J_1(x-y) u_{\mu}(\tau, x) d y d x d \tau+\displaystyle\int_{t_{1}}^{\infty} \displaystyle\int_{g_{\mu}(\tau)}^{h_{\mu}(\tau)} \displaystyle\int_{h_{\mu}(\tau)}^{\infty} J_2(x-y) v_{\mu}(\tau, x) d y d x d \tau\right)\\
	\geq& \mu\delta_0\left(\displaystyle\int_{\overline{t}}^{\overline{t}+1} \displaystyle\int_{h_{\overline{\mu}(\tau)-\epsilon_{0} / 2}(\tau)}^{h_{\overline{\mu}}(\tau)} \displaystyle\int_{h_{\overline{\mu}}(\tau)+\epsilon_{0} / 4}^{\infty} u_{\overline{\mu}}(\tau, x) d y d x d\tau +\displaystyle\int_{\overline{t}}^{\overline{t}+1} \displaystyle\int_{h_{\overline{\mu}(\tau)-\epsilon_{0} / 2}(\tau)}^{h_{\mu_{1}}(\tau)} \displaystyle\int_{h_{\overline{\mu}}(\tau)+\epsilon_{0} / 4}^{\infty} v_{\overline{\mu}}(\tau, x) d y d x d \tau\right)\\
	\geq& \dfrac{1}{4}\mu\delta_0\varepsilon_0\left(\displaystyle\int_{\overline{t}}^{\overline{t}+1}  \int\limits_{h_{\overline{\mu}}(\tau)+\epsilon_{0} / 4}^{\infty} u_{\overline{\mu}}(\tau, x) d x d\tau +\displaystyle\int_{\overline{t}}^{\overline{t}+1}  \displaystyle\int_{h_{\overline{\mu}}(\tau)+\epsilon_{0} / 4}^{+\infty} v_{\overline{\mu}}(\tau, x) d x d \tau\right).
	\end{align*}
	This gives a contradiction since $\mu$ is not  upper bounded. Moreover, with choose $\Lambda$ satisfies
	\begin{align*}
	\Lambda=\max\left\{\overline{\mu},\left[\dfrac{1}{4}\delta_0\varepsilon_0\left(\displaystyle\int_{\overline{t}}^{\overline{t}+1}  \displaystyle\int_{h_{\overline{\mu}}(\tau)+\epsilon_{0} / 4}^{\infty} u_{\overline{\mu}}(\tau, x) d x d\tau +\displaystyle\int_{\overline{t}}^{\overline{t}+1}  \displaystyle\int_{h_{\overline{\mu}}(\tau)+\epsilon_{0} / 4}^{\infty} v_{\overline{\mu}}(\tau, x) d x d \tau\right)\right]^{-1}\CL^*\right\},
	\end{align*}
	we can derive that spreading happens to $\eqref{main}$ for $\mu\geq \Lambda$.
\end{proof}
\begin{theorem}\label{th:theorem_4.10}
	Suppose that $\CR^{*}<1< \CR_0$, $\mu_1=\mu_2=\mu$ and $h_0< \CL^*/2$ hold. Then there exist $\mu^*>0$ such that spreading happen to \eqref{main} if $\mu> \mu^*$ and vanishing happens to \eqref{main} if $\mu\leq  \mu^*$. 
\end{theorem}
\begin{proof}[\bf Proof]
	We first still denote $(u_{\mu},v_{\mu},h_{\mu},g_{\mu})$ and $h_{\mu,\infty},g_{\mu,\infty}$ as in Lemma \ref{lem-4.9}. Define
	\begin{align*}
	\Sigma=\left\{\mu: \mu>0 \text { such that } h_{\mu,\infty}-g_{\mu,\infty}<+\infty\right\}.
	\end{align*} 
By Lemma \ref{lem-4.8} and \ref{lem-4.9}, we induce that $0<\sup \sum<+\infty$ and define
	$\mu^*=\sup \Sigma.$
	Then we can follow the similar argument in proof of \cite[Theorem 4.13]{CDLL} to  derive that vanishing happens for $0<\mu\leq \mu^*$ and spreading happens for $\mu> \mu^*$.
\end{proof}
Theorem \ref{theorem1.4} is a consequence of the following Lemmas \ref{lem.4.2}, \ref{lem-4.8}, \ref{lem-4.9}, and  Theorems \ref{th:theorem_4.5}, \ref{th:theorem_4.6}, \ref{th:theorem_4.10}.
\section{ Effect of the dispersal rate}\label{sec.5}
In this section, we discuss the effect of the dispersal rate on the transmission of disease, namely
Theorem \ref{theorem_1.5}. First, we claim that vanishing occurs only when $d>d^*$.

Recall that $\lambda_p(d)$ is the principal eigenvalue of \eqref{eigenvalue} when fix $L$. By Proposition \ref{pro.3.6}, we can set
\begin{align*}
 &\lambda^0=\dfrac{1}{2}\left(\dfrac{a}{H'(0)}+\dfrac{b}{G'(0)}+2d-\sqrt{\left(\dfrac{a}{H'(0)}-\dfrac{b}{G'(0)}\right)^2+4}\right),\; \lambda^{\infty}=\infty .
\end{align*}
By Proposition \ref{pro.3.6}, we have the following consequence.
\begin{theorem}
Assume that ${\bf{(J)}}$ and \eqref{conditionH_G} hold. Then for any fixed $h_0 > 0$, there exists $d^* = d^*\left(h_0\right) > 0$ such that  $\lambda_p(d) < 0$ for $0 <d <d^*$, $\lambda_p(d) = 0$ for $d=d^*$, and $\lambda_p\left(d\right)>0$ for $d>d^*$.
\end{theorem}

\begin{theorem}\label{th:5.2}
Suppose that ${(\bf J)}$, \eqref{conditionH_G} hold. If $\CR^*\geq 1$,  $d>d^*$, and  for any given $h_0 > 0$, $u_0(x), v_0(x)$ (satisfying \eqref{initial}) are sufficiently small, then vanishing occurs.
\end{theorem}
\begin{proof}[\bf Proof]
We prove this result by constructing some appropriate supersolution. Denote the principal eigenvalue of \eqref{eigenvalue} by $\lambda_p\left(d\right)$, and the corresponding positive eigenfunction by $\pmb{\varphi}=\left(\varphi^p_1, \varphi^p_2\right)$. By Proposition \ref{pro.3.6}, we see that $\lambda_p\left(d\right)>0$ as $d>d^*$.
Define 
\begin{align*}
\Lambda:=\min\left\{\dfrac{\lambda_p\left(d\right)H^{\prime}(0)}{2}, \dfrac{\lambda_p\left(d\right)G^{\prime}(0)}{2}\right\},\; C:= \dfrac{d-d^*}{\mu_1+\mu_2},\; M = \left(\mu_2\displaystyle\int_{-h_0}^{h_0}\varphi_1(x)dx +\mu_1\displaystyle\int_{-h_0}^{h_0}\varphi_2(x)dx\right)^{-1},
\end{align*}
and for any $t\geq 0,\; x\in\left[-h_0, h_0\right]$, 
\begin{align*}
\overline{h}(t)= h_0+\left(\mu_1+\mu_2\right)C\left[1+\Lambda-e^{-\Lambda t}\right],\; \overline{g}(t)=-\overline{h}(t),\;
\overline{u}(t, x)= M e^{-\Lambda t}\varphi_1(x), \; \overline{v}(t, x)= M e^{-\Lambda t}\varphi_2(x).
\end{align*} 
Clearly $\overline{h}(t)\geq h_0, \;\overline{g}(t)\leq -h_0$ for $t\geq 0$. We now prove that $\left(\overline{u}, \overline{v}, \overline{h}, \overline{g}\right)$ is an supersolution of \eqref{main}, which allows us to use comparison argument to conclude that 
\begin{align*}
\overline{g}(t)\leq g\left(t\right),\; \overline{h}(t)\geq h(t)\;\;\text{and}\; h_{\infty}-g_{\infty}\leq \overline{h}\left(\infty\right)- \overline{g}\left(\infty\right)=2h_0.
\end{align*}
Using \eqref{local1}, we obtain, by direct computations, that
\begin{align*}
 \overline{u}_t -& d\displaystyle\int_{\overline{g}(t)}^{\overline{h}(t)}J_1(x-y)\overline{u}(t,y)dy +d \overline{u}(t,x)+ a\overline{u}(x,t) - H\left(\overline{v}(t, x)\right)\\&\geq \overline{u}_t - d\displaystyle\int_{\overline{g}(t)}^{\overline{h}(t)}J_1(x-y)\overline{u}(t,y)dy +d \overline{u}(t,x)+ a\overline{u}(x,t) - H^{\prime}(0)\overline{v}(t, x) 
 \geq 0\;\;\text{for}\;t>0,\; x\in \left(\overline{g}(t), \overline{h}(t)\right).
\end{align*}
Similarly, for $t>0,\; x\in \left(\overline{g}(t), \overline{h}(t)\right)$ we have
\begin{align*}
 \overline{v}_t -& d\displaystyle\int_{\overline{g}(t)}^{\overline{h}(t)}J_2(x-y)\overline{v}(t,y)dy +d \overline{v}(t,x)+ b\overline{v}(x,t) - G\left(\overline{u}(t, x)\right)\geq  -\Lambda \overline{v} +\lambda_p\left(d\right)G^{\prime}(0)\overline{v}= \dfrac{\lambda_p\left(d\right)G^{\prime}(0)}{2}\overline{v}\geq 0.
\end{align*}
Note that $\left[\overline{g}(t), \overline{h}(t)\right]\subset \left[-h_0, h_0\right]$, we further obtain
\begin{align*}
&\mu_1 \displaystyle\int_{\overline{g}(t)}^{\overline{h}(t)}\displaystyle\int_{\overline{h}(t)}^{\infty}J_1(x-y)\overline{u}(t, x)dydx  +\mu_2\displaystyle\int_{\overline{g}(t)}^{\overline{h}(t)}\displaystyle\int_{\overline{h}(t)}^{\infty}J_2(x-y)\overline{v}(t, x)dydx\\& \leq 
\mu_1 \displaystyle\int_{-h_0}^{h_0}\overline{u}(t, x)dx  +\mu_2\displaystyle\int_{-h_0}^{h_0}\overline{v}(t, x)dx\leq \left(\mu_1+\mu_2\right)C\Lambda e^{-\Lambda t}=\overline{h}^{\prime}(t),\;\; t>0,
\end{align*}
and similarly, 
$-\mu_1 \displaystyle\int_{g(t)}^{h(t)}\displaystyle\int_{h(t)}^{\infty}J_1(x-y)\overline{u}(t, x)dydx  -\mu_2\displaystyle\int_{g(t)}^{h(t)}\displaystyle\int_{h(t)}^{\infty}J_2(x-y)\overline{v}(t, x)dydx \geq 
\overline{g}^{\prime}(t),\;\; t>0.$
Clearly, we also have
$\overline{u}\left(t, \overline{g}(t)\right)>0, \; \overline{u}\left(t, \overline{h}(t)\right)>0,\; \overline{v}\left(t, \overline{g}(t)\right)>0,\; \overline{v}\left(t, \overline{h}(t)\right)>0,\;\; t>0.$
Define 
\begin{align*}
\sigma: = \min\left\{\min\limits_{x\in \left[-h_0, h_0\right]}\varphi_1(x), \min\limits_{x\in \left[-h_0, h_0\right]}\varphi_2(x)\right\},
\end{align*}
then we have 
$u_0(x)\leq M\varphi_1(x)=\overline{u}(0, x),\;\; v_0(x)\leq M\varphi_2(x)=\overline{u}(0, x)\;\;\text{for}\; x\in \left[-h_0, h_0\right].$
Thus, for such $\Lambda, C, M$ and $u_0, v_0, \sigma$, we have, for all $x\in\left(g(t), h\left(t\right)\right)$,
\begin{align*}
\left\{\begin{array}{lll}
 \overline{u}_t - d\displaystyle\int_{\overline{g}(t)}^{\overline{h}(t)}J_1(x-y)\overline{u}(t,y)dy +d \overline{u}(t,x)+ a\overline{u}(x,t) - H\left(\overline{v}(t, x)\right)\geq 0, & t>0,  \\
\overline{v}_t - d\displaystyle\int_{\overline{g}(t)}^{\overline{h}(t)}J_2(x-y)\overline{v}(t,y)dy +d \overline{v}(t,x)+ b\overline{v}(x,t) - G\left(\overline{u}(t, x)\right)\geq 0, & t>0,  \\
\overline{u}\left(t, \overline{g}(t)\right)>0, \; \overline{u}\left(t, \overline{h}(t)\right)>0, \overline{v}\left(t, \overline{g}(t)\right)>0,\; \overline{v}\left(t, \overline{h}(t)\right)>0 & t>0,   \\  
\overline{h}^{\prime}(t) \geq \mu_1 \displaystyle\int_{g(t)}^{h(t)}\displaystyle\int_{h(t)}^{\infty}J_1(x-y)\overline{u}(t, x)dydx  +\mu_2\displaystyle\int_{g(t)}^{h(t)}\displaystyle\int_{h(t)}^{\infty}J_2(x-y)\overline{v}(t, x)dydx,& t>0,\\
\overline{g}^{\prime}(t) \leq -\mu_1 \displaystyle\int_{g(t)}^{h(t)}\displaystyle\int_{-\infty}^{g(t)}J_1(x-y)\overline{u}(t, x)dydx- \mu_2\displaystyle\int_{g(t)}^{h(t)}\displaystyle\int_{-\infty}^{g(t)}J_2(x-y)\overline{v}(t,x)dydx,& t>0,\\
\overline{u}(0, x) = M\varphi_1(x)\geq u_0(x),\,\,\overline{v}(0, x) = M\varphi_2(x)\geq v_0(x),& x\in \left[-h_0, h_0\right],\\
-g(0) = h(0) = h_0+\left(\mu_1+\mu_2\right)C\Lambda\geq h_0.
\end{array}\right.
\end{align*}
By the comparison principle,
\begin{align*}
&u(t, x)\leq \overline{u}(t, x),\; v(t, x)\leq \overline{v}(t, x)\;\;\text{for}\; \left(t, x\right)\in \left(0, \infty\right)\times \left[-h_0, h_0\right],\;
\overline{g}(t)\geq g(t),\; h(t)\leq \overline{h}(t)\;\;\; \text{for}\; t>0.
\end{align*}
This implies $\lim\limits_{t\to\infty}h(t)-g(t)\leq \lim\limits_{t\to\infty}\overline{h}(t) - \overline{g}(t)=2\left(h_0+\left(\mu_1+\mu_2\right)C\left(1+\Lambda\right)\right)<\infty$ and\\
$ \lim\limits_{t\to\infty}\left\|\left(u(t, \cdot), v(t, \cdot)\right)\right\|_{C\left([g(t), h(t)\right])}=(0, 0)$, i.e, vanishing occurs. This proof is complete.
\end{proof}
\begin{theorem}\label{th:5.3}
Let ${(\bf J)}$ and \eqref{conditionH_G} hold. If $0<d<d^*$, $\CR^*<1<\CR_0$, and  for any given $h_0 > 0$, $u_0(x), v_0(x)$ satisfy \eqref{initial}, then spreading happens.
\end{theorem}
\begin{proof}[\bf Proof]
As in the above theorem, we still use $\lambda_p\left(d\right)$ and $\varphi = \left(\varphi_1,\varphi_2\right)$ to denote the principal eigenvalue and the corresponding eigenfunction of problem \eqref{eigenvalue}, respectively. We first suppose that $0<d<d^*$. Then $\lambda_p(d)<0$. Next, we construct an subsolution. Define
\begin{align*}
\underline{u}(t, x)= \left\{\begin{array}{lll}
\delta \varphi_1(x), &\; \text{for}\; t>0,\; -h_0\leq x\leq h_0,\\
0, &\; \text{for}\; t>0,\; x>h_0,
\end{array}\right.\;\;
\underline{v}(t, x)= \left\{\begin{array}{lll}
\delta \varphi_2(x), &\; \text{for}\; t>0,\; -h_0\leq x\leq h_0,\\
0, &\; \text{for}\; t>0,\; x>h_0,
\end{array}\right.
\end{align*}
 where $\delta>0$ will be determined so small that $ \delta\left( \varphi_1(x),\; \varphi_2(x)\right)\leq \min\left\{-\lambda_p(d)H^{\prime}(0), -\lambda_p(d)G^{\prime}(0), u_0(x), v_0(x)\right\}$. Then by a direct calculation, we obtain

\begin{align*}
\left\{\begin{array}{lll}
 \underline{u}_t - d\displaystyle\int_{-h_0}^{h_0}J_1(x-y)\overline{u}(t,y)dy +d \underline{u}(t,x)+ a\underline{u}(x,t) - H\left(\underline{v}(t, x)\right)=\lambda_p(d)H^{\prime}(0)\delta \varphi_1(x)\leq 0,\; t>0, x\in[-h_0, h_0], \\
\underline{v}_t - d\displaystyle\int_{-h_0}^{h_0}J_2(x-y)\underline{v}(t,y)dy +d \underline{v}(t,x)+ b\underline{v}(x,t) - G\left(\underline{u}(t, x)\right)=\lambda_p(d)G^{\prime}(0)\delta \varphi_2(x)\leq 0,\; t>0, x\in[-h_0, h_0],\\
\underline{u}\left(t, x\right)=0\leq u(t,x),  \underline{v}\left(t, x\right)=0\leq v(t,x),\; \;t>0,\; \, x\geq h_0,\\
0= h^{\prime}_0 \leq 2h_0 \delta\left(\mu_1\varphi_1(h_0)+ \mu_2 \varphi_2(h_0)\right),\;-2h_0 \delta\left(\mu_1\varphi_1(h_0)+ \mu_2 \varphi_2(h_0)\right)\leq -h^{\prime}_0=0,\; t>0,\\
\underline{u}(0, x) = \delta\varphi_1(x)\leq u_0(x),\,\,\underline{v}(0, x) = \delta\varphi_2(x)\leq v_0(x),\; x\in \left[-h_0, h_0\right].
\end{array}\right.
\end{align*}
By the comparison principle, we further infer
\begin{align*}
u(t, x)\geq \underline{u}(t, x),\;\; v(t, x)\geq \underline{v}(t, x)\;\text{in}\; \left[0, \infty\right]\times \left[-h_0, h_0\right],
\end{align*}
which implies that 
$\liminf\limits_{t\to\infty}\left\|\left(u(t, \cdot), v(t, \cdot)\right)\right\|_{C\left([g(t), h(t)\right])}\geq \delta\left(\varphi_1\left(-h_0\right), \varphi_2\left(-h_0\right)\right)>0.$
It follows from Theorem \ref{theorem_1.2}, we derive  $\lim\limits_{t\to\infty}h(t)-g(t) = \infty$. Therefore, spreading happens.
\end{proof}
Clearly Theorem \ref{theorem_1.5} follows directly from Theorems \ref{th:5.2}, \ref{th:5.3}.


\def\cprime{$'$} \def\polhk#1{\setbox0=\hbox{#1}{\ooalign{\hidewidth
  \lower1.5ex\hbox{`}\hidewidth\crcr\unhbox0}}}
  \def\cfac#1{\ifmmode\setbox7\hbox{$\accent"5E#1$}\else
  \setbox7\hbox{\accent"5E#1}\penalty 10000\relax\fi\raise 1\ht7
  \hbox{\lower1.15ex\hbox to 1\wd7{\hss\accent"13\hss}}\penalty 10000
  \hskip-1\wd7\penalty 10000\box7}



\begin{bibdiv}
\begin{biblist}

\bib{ACS}{article}{
    AUTHOR = {Y. An},
    AUTHOR = {J-L Chern},
    AUTHOR = {J. Shi}
     TITLE = {Uniqueness of positive solutions to some coupled cooperative variational elliptic systems},
   JOURNAL = {Trans. Amer. Math. Soc.},
  FJOURNAL = {Trans. Amer. Math. Soc.},
    VOLUME = {370},
      YEAR = {2018},
    PAGES = {5209-5243},
}


\bib{ABL}{article}{
    AUTHOR = {I. Ahn},
    AUTHOR = {S. Beak},
    AUTHOR = {Z. Lin}
     TITLE = { The spreading fronts of an infective environment in a man-environment-man epidemic
model},
   JOURNAL = {Appl. Math. Model.},
  FJOURNAL = {Appl. Math. Model.},
    VOLUME = {40},
      YEAR = {2016},
    PAGES = {7082-7101},
}
\bib{BCV}{article}{
    AUTHOR = { H. Berestycki},
    AUTHOR = {J. Coville}
    AUTHOR = {H-H. Vo}
     TITLE = {On the definition and the properties of the principal eigenvalue
of some nonlocal operators},
   JOURNAL = {J. Funct. Anal.},
  FJOURNAL = {J. Funct. Anal.},
    VOLUME = {271},
      YEAR = {2016},
    NUMBER = {10},
    PAGES = {2701-2751},
}
\bib{BCV1}{article}{
    AUTHOR = {H. Berestycki},
    AUTHOR = {J. Coville}
    AUTHOR = {H-H. Vo}
     TITLE = {Persistence criteria for populations with non-local dispersion},
   JOURNAL = {J. Math. Biol.},
  FJOURNAL = {J. Math. Biol.},
    VOLUME = {72},
      YEAR = {2016},
    PAGES = {1693–1745},
}

\bib{BS}{article}{
    AUTHOR = {X. Bao}
    AUTHOR = {W. Shen},
     TITLE = {Criteria for the existence of principal eigenvalues of time periodic
cooperative linear systems with nonlocal dispersal},
   JOURNAL = {Proc. Amer. Math. Soc.},
  FJOURNAL = {Proc. Amer. Math. Soc.},
    VOLUME = {145},
      YEAR = {2017},
    PAGES = {2881-2894},
}
\bib{BLS}{article}{
    AUTHOR = {X. Bao},
    AUTHOR = {W-T. Li},
    AUTHOR = {W. Shen}
     TITLE = {Traveling wave solutions of Lotka-Volterra competition systems with nonlocal dispersal in periodic habitats},
   JOURNAL = {J. Differ. Equ.},
  FJOURNAL = {J. Differ. Equ.},
    VOLUME = {260},
      YEAR = {2016},
    NUMBER = {12},
    PAGES = { 8590-8637},
}

\bib{BZ}{article}{
    AUTHOR = {P-W. Bates},
    AUTHOR = {G. Zhao}
     TITLE = {Existence, uniqueness and stability of the stationary solution to a nonlocal evolution equation arising in population dispersal},
   JOURNAL = {J. Math. Anal. Appl.},
  FJOURNAL = {J. Math. Anal. Appl.},
    VOLUME = {332},
      YEAR = {2007},
    NUMBER = {12},
    PAGES = { 428–440},
}
\bib{BDK}{article}{
	AUTHOR = {G. Bunting},
	AUTHOR = {Y. Du}
	author={K. Krakowski}
	TITLE = { Spreading speed revisited: Analysis of a free boundary model },
	JOURNAL = {Netw. Heterog.
		Media(special issue dedicated to H. Matano)},
	VOLUME = {7},
	YEAR = {2012},
	PAGES = {583–603},
}
\bib{CDLL}{article}{
    AUTHOR = { J.-F. Cao},
    AUTHOR = {Y. Du},
    AUTHOR = {F. Li}
    AUTHOR = {W.-T. Li},
     TITLE = {The dynamics of a Fisher-KPP nonlocal diffusion model with free boundaries},
   JOURNAL = {J. Funct. Anal.},
  FJOURNAL = {Journal of Functional Analysis.},
    VOLUME = {277},
      YEAR = {2019},
    NUMBER = {8},
     PAGES = {2772-2814},
}
\bib{CK}{article}{
    AUTHOR = {V. Capasso}
    AUTHOR = {K. Kunisch},
     TITLE = {A reaction-diffusion system arising in modelling manenvironment diseases},
   JOURNAL = {Quart. Appl. Math.},
  FJOURNAL = {Quart. Appl. Math.},
    VOLUME = {46},
      YEAR = {1988},
     PAGES = {431-450.},
    
}
\bib{CM1}{article}{
    AUTHOR = {V. Capasso}
    AUTHOR = { L. Maddalena},
     TITLE = {Saddle point behaviour for a reaction-diffusion system: Application to a class of epidemic models},
   JOURNAL = {Math. Comput. Simulation.},
  FJOURNAL = {Mathematics and Computers in Simulation.},
    VOLUME = {24},
      YEAR = {1982},
    NUMBER = {6},
     PAGES = {540-547.},
    
}
\bib{CF}{article}{
    AUTHOR = {V. Capasso}
    AUTHOR = { S.L Paveri-Fontana},
     TITLE = {A mathematical model for the 1973 cholera epidemic in the European Mediterranean region},
   JOURNAL = {Revue d’Epidemiologie et de Sant\'e Publique.},
  FJOURNAL = {Revue d’Epidemiologie et de Sant\'e Publique.},
    VOLUME = {27},
      YEAR = {1979},
     PAGES = { 121-132},
    
}
\bib{CM}{article}{
    AUTHOR = {V. Capasso}
    AUTHOR = { L. Maddalena},
     TITLE = {Convergence to equilibrium states for a reaction-diffusion system modelling the
spatial spread of a class of bacterial and viral diseases},
   JOURNAL = {J. Math. Biol.},
  FJOURNAL = {Journal of Mathematical Biology.},
    VOLUME = {13},
      YEAR = {1981},
     PAGES = {173-184.},
    
}

\bib{CS}{article}{
    AUTHOR = {V. Capasso}
    AUTHOR = {G. Serio},
     TITLE = {A Generalization of the
     	Kermack-McKendrick Deterministic Epidemic Model},
   JOURNAL = {Math. Biosci},
    VOLUME = {42},
      YEAR = {1978},
     PAGES = {41-61.},
  
}

\bib{CW}{article}{
    AUTHOR = {V. Capasso}
    AUTHOR = {R. E. Wilson},
     TITLE = {Analysis of a reaction-diffusion system modeling man-environment-man epidemics.},
   JOURNAL = {SIAM J. Appl. Math.},
    VOLUME = {57},
      YEAR = {1997},
     PAGES = {327–346.},
  
}
 

\bib{CBG}{book}{
	AUTHOR = {F. Courchamp},
	AUTHOR = {L. Berec}
	author={J. Gascoigne}
	YEAR = {2008},
	TITLE = {Allee effects in ecology and
		conservation},
ADDRESS={ New York, }
	publisher={NY: Oxford University Press}
}





\bib{DN}{article}{
    AUTHOR = {Y. Du},
    AUTHOR = {W. Ni}
     TITLE = {Analysis of a West Nile virus model with nonlocal diffusion and free
boundaries},
   JOURNAL = {Nonlinearity.},
  FJOURNAL = {Nonlinearity.},
    VOLUME = {33},
      YEAR = {2020},
    PAGES = { 4407–4448},
}
\bib{DS}{article}{
    AUTHOR = {Y. Du},
    AUTHOR = {J. Shi}
     TITLE = {Allee effect and bistability in a spatially heterogeneous predator-prey model
},
   JOURNAL = {Trans. Amer. Math. Soc.},
  FJOURNAL = {Trans. Amer. Math. Soc.},
    VOLUME = {359},
      YEAR = {2007},
    NUMBER = {9},
    PAGES = {4557–4593},
}
\bib{DL}{article}{
    AUTHOR = {Y. Du},
    AUTHOR = {Y. Lou}
     TITLE = {Some uniqueness and exact multiplicity results for a predator-prey model. 
},
   JOURNAL = {Trans. Amer. Math. Soc.},
  FJOURNAL = {Trans. Amer. Math. Soc.},
    VOLUME = {349},
      YEAR = {1997},
    NUMBER = {6},
    PAGES = {2443–2475},
}



\bib{D}{article}{
    AUTHOR = {E. N. Dancer},
     TITLE = {On the principal eigenvalue of linear cooperating elliptic systems with
small diffusion},
   JOURNAL = {J. Evol. Equations.},
  FJOURNAL = {J. Evol. Equations.},
    VOLUME = {9},
      YEAR = {2009},
    NUMBER = {3},
    PAGES = {419–428},
}
\bib{DDF}{article}{
    AUTHOR = {R. D. Demasse},
    AUTHOR = {A. Ducrot},
    AUTHOR = {F. Fabre}
     TITLE = {Steady state concentration for
a phenotypic structured problem modeling the evolutionary epidemiology of spore producing pathogens},
   JOURNAL = { Math. Models Methods Appl. Sci.},
  FJOURNAL = { Math. Models Methods Appl. Sci.},
    VOLUME = {27},
      YEAR = {2017},
    NUMBER = {2},
    PAGES = {385–426},
}
\bib{EFSA}{article}{
    AUTHOR = {M. Emch},
    AUTHOR = {C. Feldacker},
    AUTHOR = { M. S. Islam},
     AUTHOR = { M. Ali}
     TITLE = {Seasonality of cholera from 1974 to 2005: a review of global patterns},
   JOURNAL = {Int. J. Health Geogr.},
  FJOURNAL = {Int. J. Health Geogr.},
    VOLUME = {31},
      YEAR = {2008},
    NUMBER = {7},
    PAGES = {1-13},
}

\bib{G}{article}{
    AUTHOR = {Q. Griette}
     TITLE = {Singular measure traveling waves in an epidemiological model with continuous phenotypes.},
   JOURNAL = {Trans. Amer. Math. Soc.},
  FJOURNAL = {Trans. Amer. Math. Soc.},
    VOLUME = {371},
      YEAR = {2019},
    PAGES = {4411-4458},
}
\bib{Gan}{article}{
    AUTHOR = {F.Gantmacher}
     TITLE = {Theory of Matrices.},
   JOURNAL = {AMS Chelsea publishing, New York.},
  FJOURNAL = {AMS Chelsea publishing, New York.},
      YEAR = {1959},
}
\bib{HY}{article}{
    AUTHOR = { C. Hsu}
    AUTHOR = {T. Yang},
     TITLE = {Existence, uniqueness, monotonicity and asymptotic behaviour of travelling waves
for epidemic models},
   JOURNAL = {Nonlinearity.},
  FJOURNAL = {Nonlinearity.},
    VOLUME = {26},
      YEAR = {2013},
     PAGES = {121-139},
}



\bib{JLLLW}{article}{
    AUTHOR = {D. Jiang},
    AUTHOR = {K-Y Lam},
    AUTHOR = {Y. Lou},
    AUTHOR = { Z-C Wang}
     TITLE = {Monotonicity and Global Dynamics of a Nonlocal Two-Species Phytoplankton Model},
   JOURNAL = {SIAM J. Appl. Math.},
  FJOURNAL = {SIAM J. Appl. Math.},
    VOLUME = {79},
      YEAR = {2019},
    NUMBER = {2},
    PAGES = {716–742},
}
\bib{KDL}{article}{
	AUTHOR = {A. M. Kramer},
	AUTHOR = {B. Dennis}
	author={A. Liebhol}
	author={J.M. Drake}
	TITLE = {The evidence for Allee effects },
	JOURNAL = {Population Ecology},
	VOLUME = {51},
	YEAR = {2009},
	PAGES = {341-354},
}



\bib{LL}{article}{
    AUTHOR = {K-Y Lam},
    AUTHOR = {Y. Lou}
     TITLE = {Asymptotic Behavior of the Principal Eigenvalue for Cooperative Elliptic Systems and Applications},
   JOURNAL = { J. Dynam. Differential Equations.},
  FJOURNAL = {J. Dynam. Differential Equations.},
    VOLUME = {28},
      YEAR = {2016},
    NUMBER = {1},
    PAGES = {29–48},
}
\bib{LVW}{article}{
    AUTHOR = {C. Lederman},
    AUTHOR = {J. L. Vazquez}
    AUTHOR = {N. Wolanski}
     TITLE = {Uniqueness of Solution to a Free Boundary Problem from Combustion
},
   JOURNAL = {Trans. Amer. Math. Soc.},
  FJOURNAL = {Trans. Amer. Math. Soc.},
    VOLUME = {353},
      YEAR = {2000},
    NUMBER = {2},
    PAGES = {655-692},
}



\bib{LZ}{article}{
    AUTHOR = {T. Lim},
    AUTHOR = {A. Zlato$\breve{s}$}
     TITLE = {Transition fronts for inhomogeneous Fisher-KPP reactions and non-local diffusion},
   JOURNAL = {Trans. Amer. Math. Soc.},
  FJOURNAL = {Trans. Amer. Math. Soc.},
    VOLUME = {368},
      YEAR = {2016},
    PAGES = {8615-8631},
}
\bib{LX}{article}{
		year = {2017},
		publisher = {American Institute of Mathematical Sciences ({AIMS})},
		volume = {37},
		number = {5},
		pages = {2483-2512},
		author = {W-T. Li},
			author = {W-B. Xu},
				author = {L. Zhang},
		title = {Traveling waves and entire solutions for an epidemic model with asymmetric dispersal},
		journal = {Discrete {\&} Continuous Dynamical Systems - A}
	}
\bib{L}{article}{
	year = {2007},
	publisher = {Nonlinearity},
	volume = {20},
	pages = {1883–1892},
	author = {Z. Lin},
	title = {A free boundary problem for a predator–prey model},
}
\bib{MR}{book}{
	AUTHOR = {J. D. Murray},
	TITLE = {Mathematical Biology},
	PUBLISHER = {Springer},
	ADDRESS={Berlin-Heidelberg-New York}
	YEAR = {1993},
}
\bib{NK}{book}{
	AUTHOR = {R. Nathan},
	AUTHOR = {E. Klein},
	AUTHOR = {J.J. Robledo-Arnuncio},
	AUTHOR = {E. Revilla},
	publisher = {Oxford University Press},
PAGES = {186-210},
TITLE = {Dispersal kernels: review Dispersal Ecology and Evolution},
	YEAR = {2012},
}

\bib{RS}{article}{
    AUTHOR = {N. Rawal}
    AUTHOR = {W. Shen}
     TITLE = {Criteria for the existence and lower bounds of principal eigenvalues of time periodic nonlocal dispersal operators and applications},
   JOURNAL = {J. Dynam. Differential Equations.},
  FJOURNAL = {J. Dynam. Differential Equations.},
    VOLUME = {24},
      YEAR = {2012},
    NUMBER = {4},
    PAGES = {927-954},
}



\bib{SX}{article}{
    AUTHOR = {W. Shen},
    AUTHOR = {X. Xie}
     TITLE = {On principal spectrum points/principal eigenvalues of nonlocal dispersal operators and applications},
   JOURNAL = {Discrete Contin. Dyn. Syst.},
  FJOURNAL = {Discrete Contin. Dyn. Syst.},
    VOLUME = {35},
      YEAR = {2015},
    NUMBER = {4},
    PAGES = {1665-1696},
}

\bib{SV1}{article}{
    AUTHOR = {Z. Shen},
    AUTHOR = {H.-H. Vo}
     TITLE = {Nonlocal dispersal equations in time-periodic media: principal spectral theory,
limiting properties and long-time dynamics},
   JOURNAL = { J. Differ. Equ.},
  FJOURNAL = { J. Differential Equations.},
    VOLUME = {267},
      YEAR = {2020},
    NUMBER = {2},
    PAGES = {1423-1466},
}


\bib{WWZ}{article}{
    AUTHOR = {C. Wu},
    AUTHOR = {Y. Wang},
    AUTHOR = {X. Zou}
     TITLE = {Spatial-temporal dynamics of a Lotka-Volterra competition
model with nonlocal dispersal under shifting environment},
   JOURNAL = { J. Differ. Equ.},
  FJOURNAL = { J. Differential Equations.},
    VOLUME = {267},
      YEAR = {2019},
    PAGES = { 4890-4921},
}
\bib{RZ}{article}{
    AUTHOR = {R. Wu},
    AUTHOR = {X.-Q. Zhao}
     TITLE = {Spatial Invasion of a Birth Pulse Population with Nonlocal Dispersal},
   JOURNAL = {SIAM J. Appl. Math.},
  FJOURNAL = {SIAM J. Appl. Math.},
    VOLUME = {79},
      YEAR = {2019},
    NUMBER = {3},
    PAGES = {1075–1097},
}
\bib{WLR}{article}{
    AUTHOR = {Z-C. Wang},
    AUTHOR = {W-T. Li},
    AUTHOR = {S. Ruan}
     TITLE = {Entire solutions in bistable reaction-diffusion equations with nonlocal delayed nonlinearity},
   JOURNAL = {Trans. Amer. Math. Soc.},
  FJOURNAL = {Trans. Amer. Math. Soc.},
    VOLUME = {361},
      YEAR = {2009},
    NUMBER = {4},
    PAGES = {2047–2084},
}
\bib{WLL}{article}{
    AUTHOR = {H. Weinberger},
    AUTHOR = {M. Lewis},
    AUTHOR = {B. Li}
     TITLE = {Analysis of linear determinacy for spread in cooperative models},
   JOURNAL = {J. Math. Biol.},
  FJOURNAL = {Journal of Mathematical Biology.},
    VOLUME = {45},
      YEAR = {2002},
    NUMBER = {3},
    PAGES = {183-218},
}

\bib{WW}{article}{
    AUTHOR = {J. Wang},
    AUTHOR = {X. Wu}
     TITLE = {Dynamics and Profiles of a Diffusive Cholera Model with Bacterial Hyperinfectivity and Distinct Dispersal Rates},
   JOURNAL = {J. Dynam. Differential Equations.},
  FJOURNAL = {J. Dynam. Differential Equations.},
      YEAR = {2021},
    PAGES = {https://doi.org/ 10.1007/s10884-021-09975-3},
}




\bib{YLR}{article}{
    AUTHOR = {F-F.Yang},
    AUTHOR = {W-T. Li},
    AUTHOR = {S. Ruan}
     TITLE = {Dynamics of a nonlocal dispersal SIS epidemic model with Neumann boundary conditions},
   JOURNAL = { J. Differ. Equ.},
  FJOURNAL = { J. Differential Equations.},
    VOLUME = {267},
      YEAR = {2018},
    NUMBER = {3},
    PAGES = {2011-2051},
}
\bib{XZ}{article}{
    AUTHOR = {D. Xu},
    AUTHOR = { X-Q. Zhao}
     TITLE = {Bistable Waves in an Epidemic Model},
   JOURNAL = {, J. Dynam. Differential Equations.},
  FJOURNAL = {, J. Dynam. Differential Equations.},
    VOLUME = {16},
      YEAR = {2004},
    PAGES = {679-707},
}

\bib{WH}{article}{
    AUTHOR = {S.L. Wu},
    AUTHOR = {C.H. Hsu}
     TITLE = {Existence of entire solutions for delayed monostable epidemic models.},
   JOURNAL = {,Trans. Amer. Math. Soc.},
  FJOURNAL = {Trans. Amer. Math. Soc.},
    VOLUME = {368},
      YEAR = {2016},
    NUMBER = {no. 9},
    PAGES = {6033–6062},
}





\bib{ZZ}{article}{
    AUTHOR = {G.B. Zhang},
    AUTHOR = {X-Q. Zhao}
     TITLE = {Propagation phenomena for a two-species Lotka-Volterra strong competition system with nonlocal dispersal},
   JOURNAL = { Calc. Var. Partial Differential Equations},
  FJOURNAL = { Calc. Var. Partial Differential Equations.},
    VOLUME = {16},
      YEAR = {2004},
    PAGES = {679-707},
}



\bib{WWZ1}{article}{
    AUTHOR = {M. Zhao},
    AUTHOR = {Y. Zhang},
     AUTHOR = {W-T. Li},
    AUTHOR = {Y. Du}
     TITLE = {The dynamics of a degenerate epidemic model with nonlocal diffusion and free boundaries},
   JOURNAL = { J. Differ. Equ.},
  FJOURNAL = { J. Differential Equations.},
    VOLUME = {269},
      YEAR = {2020},
    NUMBER = {4},
    PAGES = {3347-3386},
}





\bib{ZR}{article}{
    AUTHOR = {G. Zhao},
    AUTHOR = {S. Ruan}
     TITLE = {Spatial and Temporal Dynamics of a Nonlocal Viral Infection Model},
   JOURNAL = {SIAM J. Appl. Math.},
  FJOURNAL = {SIAM J. Appl. Math.},
    VOLUME = {78},
      YEAR = {2018},
    NUMBER = {4},
    PAGES = {1954–1980},
}
\bib{ZX}{article}{
    AUTHOR = {P. Zhou},
    AUTHOR = {D. Xiao}
     TITLE = {The diffusive logistic model with a free boundary in heterogeneous environment},
    JOURNAL = { J. Differ. Equ.},
  FJOURNAL = { J. Differential Equations.},
    VOLUME = {256},
      YEAR = {2014},
    NUMBER = {6},
    PAGES = {1927-1954},
}







\end{biblist}
\end{bibdiv}
\end{document}